\documentclass[11pt,leqno]{amsart}
\usepackage{euscript, amssymb, amsmath, amsthm, mathrsfs}
\usepackage[titletoc]{appendix}
\usepackage{epsfig}
\usepackage{graphicx}
\usepackage{caption}
\usepackage{longtable}
\usepackage{dcolumn}
\usepackage{setspace}
\usepackage{breqn}
\usepackage{bbm}
\usepackage{mathtools}
\usepackage[dvipsnames,table,xcdraw]{xcolor}
\usepackage{booktabs}
\usepackage{lmodern} 

\usepackage{color}
\usepackage[pagebackref=true, colorlinks=true, citecolor=blue]{hyperref}
\usepackage{cite}
\numberwithin{equation}{section}

\newlength{\templength}
\newlength{\textparindent}
\makeatletter
\let \@makefntextorig \@makefntext
\newcommand{\@makefntextcustom}[1]{%
    \parindent 2\textparindent%
    \hspace{-\textparindent}%
    \settowidth{\templength}{0}%
    \ifnum\value{footnote}<10 \hspace{\templength}\else\fi%
    \thefootnote.\enskip #1%
}
\renewcommand{\@makefntext}[1]{\@makefntextcustom{#1}}
\makeatother

\usepackage[english]{babel}
\usepackage{blindtext}
\usepackage[pagebackref=true, colorlinks=true, citecolor=blue]{hyperref}

\definecolor{gr}{rgb}{0.0, 0.42, 0.24}

\numberwithin{equation}{section}

\date{\today}

\newtheorem{thm}{Theorem}[section]
\newtheorem{lem}[thm]{Lemma}
\newtheorem{prop}[thm]{Proposition}
\newtheorem{defn}[thm]{Definition}
\newtheorem{rem}[thm]{Remark}

\numberwithin{equation}{section}

\newcommand{\ud}{\mathrm{d}}
 
\newcommand{\ep}{\varepsilon}
\newcommand{\R}{\mathbb{R}}

\theoremstyle{definition}
\usepackage{cancel}

\usepackage{xcolor}
\colorlet{BLUE}{blue}

\usepackage{epsfig}
\usepackage{graphicx}
\usepackage{caption}
\usepackage{longtable}
\usepackage{dcolumn}
\usepackage{setspace}
\numberwithin{equation}{section}
\usepackage{thmtools}
\usepackage{verbatim}
\usepackage{enumerate}
\usepackage{mathrsfs}
\usepackage{color}
\usepackage[T1]{fontenc} 
\usepackage{tikz}
\usetikzlibrary{matrix,arrows,positioning}
\usepackage{ifthen}
\usepackage{appendix}

\usepackage[english]{babel}
\usepackage{blindtext}

\newcommand{\sB}{\mathcal B}

\newcommand{\sE}{\mathcal E}

\newcommand{\sC}{\mathcal C}
\newcommand{\sF}{\mathcal F}

\newcommand{\sO}{\mathcal O}
\newcommand{\sP}{\mathcal P}

\newcommand{\sX}{\mathcal Z}

\newcommand{\bE}{{\bf E}}

\newcommand{\bP}{{\bf P}}

\newcommand{\vv}{{\bf v}}

\newcommand{\N}{\mathbb{N}}

\newcommand{\dpiCoeff}{\EuScript{K}}

\newcommand{\norm}[1]{\left\| #1 \right\|}
\newcommand{\iprod}[2]{( #1, #2 )}

\newcommand{\Hiprod}[2]{( #1, #2 )}
\newcommand{\Viprod}[2]{(( #1, #2 ))}
\newcommand{\Hnorm}[1]{| #1 |}
\newcommand{\Vnorm}[1]{\| #1 \|}
\newcommand{\Vdualpair}[2]{\langle #1, #2 \rangle}

\newcommand{\levy}{L\'{e}vy~}
\newcommand{\cadlag}{c\`{a}dl\`{a}g~}

\newcommand{\ra}{\rightarrow}

\newcommand{\tn}[1]{\textnormal{#1}}
\renewcommand{\tilde}{\widetilde}

\renewcommand{\emph}{\textbf}
\renewcommand{\hat}{\widehat}

\newcommand{\cutoff}[1]{\theta(\Vnorm{#1})}
\newcommand{\oud}{\otimes \mathrm{d}}

\newtheoremstyle{myThm}
{8pt}
{8pt}
{\it}
{}
{\bf}
{.}
{6mm}
{}

\newtheoremstyle{myEx}
{8pt}
{8pt}
{\normalfont}
{}
{\bf}
{.}
{6mm}
{}

\theoremstyle{myThm}

\theoremstyle{definition}

\theoremstyle{myEx}

\usepackage[normalem]{ulem} 
\usepackage{cancel}

\usepackage[most]{tcolorbox}
\newenvironment{myenum}{\begin{enumerate}[\it i\tn{)}]}{\end{enumerate}}
\tcbset{
	frame code={}
	center title,
	left=0pt,
	right=0pt,
	top=0pt,
	bottom=0pt,
	colback=gray!10,
	colframe=white,
	width=\dimexpr\textwidth\relax,
	enlarge left by=0mm,
	boxsep=5pt,
	arc=0pt,outer arc=0pt,
}

\date{}

\begin{document}
	
	\title[Stochastic Ladyzenskaya-Smagorinsky equations with L\'evy noise]{ Nonlinear Stochastic parabolic partial differential equations with a monotone operator of the Ladyzenskaya-Smagorinsky type, driven by a L\'evy noise}
	\author[P. Nguyen, K. Tawri and R. Temam]{Phuong Nguyen$^{ 1,2}$, Krutika Tawri$^{1}$ and Roger Temam$^{1}$}

\maketitle
	{
	\centering{$^{1}$ Department of Mathematics and Institute for Scientific Computing and Applied Mathematics,}
	\centering{Indiana University,}
	\centering{Bloomington, IN 47405, USA.}\\
	\centering{$^{2}$ Department of Mathematics and Statistics,}
	\centering{Texas Tech University,}
	\centering{Lubbock, TX 79409, USA.}
}

\begin{abstract}
{The aim of this article is to show the global existence of both martingale and pathwise solutions of
{stochastic} equations with a monotone operator, of the Ladyzenskaya-Smagorinsky type, driven by a general Lévy noise. The classical approach based on using directly the Galerkin approximation is not valid. Instead, our approach is
based on using appropriate approximations for the monotone operator, Galerkin approximations and on the theory of martingale solutions.}
\end{abstract}
\tableofcontents

\section{Introduction} \label{sec:Introduction}
In this article we study the existence and uniqueness of global {martingale and} pathwise solutions to stochastic partial differential equations (SPDEs) with a nonlinear monotone operator driven by a \levy noise, not necessarily square-integrable. 
 As explained below from the physical point of view, the two types of equations that we consider are firstly the equations of Ladyzenskaya-Smagorinsky type, that are essentially the Navier Stokes equations with a viscosity depending on the gradient of velocity. The second type of equations that we consider are parabolic equations with a linear elliptic part and a nonlinearity of polynomial type. Such equations were already considered {in \cite{CNTT20}, \cite{DHI}} with restrictions on the degree of the polynomials but here polynomials of arbitrary degrees can be considered. Such equations appear in biology in e.g. population dynamics.\\
 The equation that we consider is written in the abstract form
\begin{equation} \label{spde}
\begin{cases}
\ud u(t) + [\gamma Au(t) + B(u(t),u(t)) + F(u(t))] \ud t \\
\hspace{0.5in}= G(u(t-)) \ud W + \int_{E_0} K(u(t-),\xi) \ud \hat{\pi}(t,\xi) + \int_{E \setminus E_0} \dpiCoeff(u(t-),\xi) \ud \pi(t,\xi),  \\
u(0) = u_0 ,
\end{cases}
\end{equation}
where $u$ takes values in a real, separable Hilbert space $H$. On the left-hand side of \eqref{spde}, $A$ is a linear operator on $H$, generally unbounded, $B$ is a bilinear form of the type often encountered in fluid dynamics and $F$ is a monotone operator on a reflexive Banach space $X$ which is continuously and densely embedded into $H$ (see e.g. \cite{Br73} {and below}). In Section \ref{sec:functional_framework} we introduce the precise assumptions on $F$. {The terms on the right-hand side of \eqref{spde} are stochastic integrals that represent different effects of a \levy noise, which is fitting in stochastic models of fluid dynamics as a way to represent interactions that occur at random times. The term $G(u)\ud W$ has effects that are continuous in time on the system, whereas the two terms $\int_{E_0} K(u(t-),\xi) \ud \hat{\pi}(t,\xi)$, $\int_{E \setminus E_0} \dpiCoeff(u(t-),\xi) \ud \pi(t,\xi),$ influence the system discretely in time.}

Interest in SPDEs with \levy noise has been growing in the recent years, see e.g., \cite{BLZ}, \cite{EulerPaper}, \cite{1LayerShallow}, \cite{MS16}, \cite{Motyl}. In \cite{MajCh16} and \cite{Thual16} the authors use processes with jumps to model phase transitions in bursts of wind that contribute to the dynamics of El Ni\~no. In \cite{SteNee14} the authors use processes with jumps to model phase transitions in precipitation dynamics.
Models based on SPDEs with \levy noise have been used to describe observational records in paleoclimatology in \cite{Dit99}, with the jump events of the \levy noise proposed as representing abrupt triggers for shifts between glacial and warm climate states.
See also \cite{PenEw08} for a comparison of Wiener noise and \levy noise in modeling climate phenomena.

{Our main results are Theorem \ref{existence_martingale} and Theorem \ref{existence} which establish the existence and uniqueness of martingale and pathwise solutions to the equation \eqref{spde} (see also Definitions \ref{definition1} and \ref{definition2}). We use the classical Galerkin approximation method to prove Theorem \ref{existence_martingale}. However a traditional application of the Galerkin method fails as {the projectors} $P_n$, defined in \eqref{e11.12}, may not necessarily be projectors in the Banach space $X$ associated with $F$ (see Section \ref{sec:functional_framework}). Hence we approximate the monotone operator $F$ in \eqref{spde} by a family of monotone operators $F^R$ acting in another Hilbert space, continuously and densely embedded in $H$, so as to recover certain useful properties when $R \rightarrow \infty$. The a priori estimates for \eqref{spde} also suggest the need to truncate the bilinear term $B$ in the original system into $B_{\tilde R}$ (see equations \eqref{approx_eqn}), then pass to the limit $\tilde R \rightarrow \infty$.

 The paper is organized as follows. {In Section \ref{sec:globally_monotonic_coefficients} we state the main results.} In Section \ref{sec:functional_framework} we layout the functional analytical and probabilistic framework and the assumptions on the terms appearing in \eqref{spde}. In Section \ref{galerkin} we introduce the Galerkin system for the approximation equations \eqref{approx_eqn}. We establish a priori estimates in Section \ref{ss11.4} and tightness of the laws of the Galerkin approximations in Section \ref{tightness}. Then we use the Skorohod convergence theorem to obtain a sequence, same in distribution, that converges a.s. on a
 new probability space and pass to the limit to obtain a global martingale solution to the truncated
 equation \eqref{approx_eqn} in Section \ref{passagetothelimit}. We then pass to the limit $R\rightarrow \infty$ and $\tilde R \rightarrow \infty$ in Sections \ref{sec:R} and \ref{sec:tildeR} respectively to obtain a global martingale solution to \eqref{spde}. In Section \ref{global} we establish pathwise uniqueness for the equation \eqref{spde} and then obtain a global pathwise solution to \eqref{spde} using a well-known argument {by Gy\"ongy and Krylov} that  generalizes {to infinite dimensions the classical} Yamada–Watanabe theorem. In Section \ref{Examples} we {describe} examples that fit the setting of our analysis.

 }
\section{Mathematical framework} \label{sec:functional_framework}
{The following is a general framework which we will adapt to the actual problem under consideration.}
We consider a separable, real Hilbert space $V \subset H$, such that the embedding is dense and compact. Similarly, $X$ is a reflexive Banach space continuously and densely embedded into $H$.  We may thus define the Gelfand inclusions
\begin{align}\label{V}
V\subset H\equiv H' \subset V',\\
X \subset H \equiv H' \subset X'.\label{eq11.1}
\end{align}
where $V'$, $X'$ and $H'$ are the dual spaces of $V$, $X$ and $H$ respectively.
 In fact, we will strengthen \eqref{V} and \eqref{eq11.1} by requiring that $V\cap X$ is dense in $V$ and $X$, so that, for $V'+X'$
denoting the dual of $V\cap X$, we have
\begin{subequations}
	\begin{equation}\label{assumptiononX}
	V\cap X\text{}\subset
	X\subset H\subset X'\subset V'+X',
	\end{equation}
	and
	\begin{equation}\label{assumptiononV}
	V\cap X\text{}\subset
	V\subset H\subset V' \subset V'+X'.
	\end{equation}
\end{subequations}
 We denote by $\Hiprod{\cdot}{\cdot}$, $\Hnorm{\cdot}$, $\Viprod{\cdot}{\cdot}$, and $\Vnorm{\cdot}$ the norms and inner products of $H$, and $V$ respectively.\\
  The duality pairing between $V'$ and $V$ will be denoted by $\Vdualpair{\cdot}{\cdot}$.\\
For the purpose of the abstract analysis in this paper, in continuation with \cite{TT20}, and keeping our primary examples in mind, we will first assume that, for some open set $\sO \subseteq \R^d$, $V$ and $H$ are closed subspaces of $(H^1(\mathcal{O}))^d$ and $(L^2(\mathcal{O}))^d$ respectively and that $X=(W^{k,p}(\mathcal{O}))^d \cap V$, for some $p > 2$ and an integer $k\geq 0$. However {slightly} different settings can be {introduced} depending on the properties of the monotone operator {(see Remark \ref{fw_explanation} and Section \ref{other}).} 

We now give the precise assumptions on each of the terms appearing in \eqref{spde}, beginning with the linear term $A$. In our analysis we shall assume that $\gamma>0$ in \eqref{spde}. However we {show in Section \ref{withoutA} (see Remark \ref{imp_equi}), how the case $\gamma=0$ is handled.} We assume that $A : D(A) \subset V \ra H$ is an unbounded, densely defined, bijective, linear operator such that $\Hiprod{Av}{w} = \Viprod{v}{w}$ for all $v,w \in D(A)$. Under these conditions on $A$, the space $D(A)$ is complete under the Hilbertian norm $|v|_{D(A)} := \Hnorm{Av}$. When $H$ is a subspace of $(L^2(\sO))^d$ and $V$ a subspace of $(H^1(\sO))^d$, the role of $A$ can be played by the negative Laplacian associated with suitable boundary conditions. It is clear that for suitable boundary conditions, the operator $A$ is symmetric. The operator $A$ can be viewed as a linear continuous operator from $V$ to $V'$ via
\begin{equation} \label{L2.1}
\Vdualpair{Av}{w} := \Viprod{v}{w} , \qquad \tn{for all } v,w \in V.
\end{equation}
Furthermore, since we assume that $V$ is compactly embedded in $H$, it follows that $A^{-1}$ is a compact operator on $H$. 
Since $A$ is assumed to be symmetric and surjective we easily see that $A^{-1}$ is self-adjoint. 
 By the spectral theor{y} for positive compact operators, there exists an orthonormal basis $\left(w_k\right)_{k=1}^\infty$ of $H$ {made of eigenvectors of $A$}. 
 The sequence $\left(w_k\right)_{k=1}^\infty$ lies in $D(A)$ and forms an orthogonal basis in $D(A),V,V'$. Furthermore, each $w_k$ is an eigenvector of $A$ whose corresponding eigenvalue $\lambda_k$ is positive and $\lambda_k \ra \infty$ as $k \ra \infty$. Note that
\[
D(A^{1/2}) := \Big\{ \sum_{k=1}^\infty \alpha_k w_k \in H : \sum_{k=1}^\infty \alpha_k^2 \lambda_k < \infty \Big\}
\]
is nothing but the space $V$. Under the conditions given above, one can show that the embedding $D(A) \subset D(A^{1/2}) = V$ is compact (see e.g.~\cite{LM72}, \cite{T_NSE}).

Next we assume that $B \colon V \times D(A) \to V'$ is bilinear and continuous. Furthermore, we assume that $B$ maps the subspace $D(A) \times D(A)$ continuously into $H$ and that $B$ possesses the cancellation property
\begin{equation} \label{eqn:B_cancellation_property}
\Vdualpair{B(v,w)}{w} = 0, \qquad \tn{for all } v \in V \tn{ and } w \in D(A).
\end{equation}
We also assume that 
\begin{equation}\begin{split}
\begin{cases}\label{eqn:B_bound2}
|{B(v,w)}|_{D(A)'} &\lesssim \Hnorm{v}  \|w\| \qquad \tn{for all } v \in H ,w\in V, 
 \\
 &\lesssim {\Vnorm{v}  |w|} \qquad \tn{for all } v \in V ,w \in H, 
 \end{cases}\end{split}\end{equation}and,
 \begin{align}
|{B(v,v)}|_{V'} & \lesssim |v|^{\frac12}\|v\|^{\frac32}  \qquad \tn{for all } v \in V. \label{eqn:B_bound3}
\end{align}
In Section \ref{uniqueness},
	we also use the following assumption:
	\begin{align}\label{rs}
	|B(v,v)|_{X'} \lesssim |v|\|v\| \qquad \tn{for all } v \in V.
	\end{align}
We will often write $B(v) := B(v,v)$.\\
 For certain subspaces $V$ and $H$ of $(H^1(\mathcal{O}))^d$ and $(L^2(\mathcal{O}))^d$ respectively, we have in mind for $B$ the bilinear function $B(v,w) := (v \cdot \nabla) w$, corresponding to the nonlinear terms in the Navier-Stokes equations (see e.g. \cite{T_NSE}). 
{  \begin{rem}\label{fw_explanation}
 	The framework above, which we will use till the end of Section \ref{ex_poly}, covers the main example that we have in view, namely the Ladyzenskaya-Smagorinsky equations and equations of the same type. We will see in Section \ref{other} a variation of this framework which covers other related equations.
 \end{rem}}
Before discussing the coefficients $G$ and $K$ of the noise terms in equation \eqref{spde} we discuss the assumptions on the noise. The noise processes will be defined on a fixed probability space $(\Omega,\sF,\bP)$. In this paper, we assume that $W$ is a Wiener process taking values in a real, separable Hilbert space $U$ (see \cite{Bil95}).
 We denote by $Q$ the covariance operator of $W$ and we set $U_0 := Q^{1/2}(U)$. We assume that $\pi$ is a Poisson random measure on $(0,\infty) \times E$ arising from a stationary Poisson point process $\Pi$ on a measurable space $(E,\mathscr{E})$. The intensity measure of $\pi$ has the form $\ud \nu \oud t$ for some $\sigma$-finite measure $\nu$ on $(E,\mathscr{E})$. In equation \eqref{spde} we fix a set $E_0 \in \mathscr{E}$ such that $\nu(E \setminus E_0) < \infty$.   Furthermore, we assume that $W$ is independent of $\pi$. Under this condition it is possible to construct a filtration $\left(\sF_t\right)_{t \geq 0}$ on $(\Omega,\sF,\bP)$ such that $W$ is an $\sF_t-$Wiener process, $\Pi$ is an $\sF_t-$ Poisson point process and such that the filtered probability space $(\Omega,\sF,(\sF_t)_{t \geq 0},\bP)$ satisfies the usual conditions (as explained in {e.g.} \cite{PZvsIWreview}). {The motivation behind these assumptions stems from the L\'evy-Khinchin decomposition of a \levy process $L$ which, informally, gives us the rule for the decomposition of integration with respect to $L$ as $\ud L=a \ud t+\ud W+\ud \hat \pi+\ud \pi$, where $\pi$ is the jump measure of $L$.  Hence as an example, we can set $E:=U\setminus \{0\}$, $E_0:=\{y\in U:0<|y|_U<1\}$ and endow the space $E$ with a metric such that it is separable, complete, such that $A \subset E$ is bounded if and only if $A$ is separated from 0, and $\nu$ is finite on bounded subsets of $E$.}\\
  Now, let $L_2(U_0;H)$ denote the space of Hilbert-Schmidt operators from ${U_0} $ to $H$ equipped with the
 Hilbert-Schmidt norm
$ \|S\|^2_{L_2(U_0;H)} :=
\sum_{n=1}^\infty |S{e_n}|^2
 $ where $\{e_n\}_{n=1}^\infty$ is an orthonormal basis of ${U_0}$. 
 
 Then for the coefficients $G$ and $K$ in equation \eqref{spde} we assume that,
\begin{itemize}
	\item $G \colon H \ra L_2(U_0,H)$ and $G$ maps $V$ into $L_2(U_0,V)$, and 
	\item $K \colon H \times E \ra H$ is a Borel measurable function and $K$ maps $V \times E$ into $V$.
	\item And {for some $\rho >0$} we have
\end{itemize}
\begin{align} \label{eqn:G,K_H_growth}
&\norm{G(v)}^2_{L_2(U_0,H)} + \int_{E_0} \Hnorm{K(v,\xi)}^2 \ud \nu(\xi) \leq \rho(1 + \Hnorm{v}^2) \qquad \tn{for all } v \in H,\\
\label{eqn:G,K_H_Lipschitz}
&\norm{G(v)-G(w)}^2_{L_2(U_0,H)} + \int_{E_0} \Hnorm{K(v,\xi) - K(w,\xi)}^2 \ud \nu(\xi) \leq \rho\Hnorm{v-w}^2 \quad \tn{for all } v,w \in H.
\end{align}
Note that the above conditions and Lemma 3.2 from \cite{CNTT20} imply that the stochastic integrals $\int_0^t G(u(s-)) \ud W(s)$ and  $\int_{(0,t]} \int_{E_0} K(u(s-),\xi) \ud \hat{\pi}(s,\xi)$ appearing on the right-hand side of equation \eqref{spde} are well-defined for every adapted, c\`adl\`ag, $H$-valued process $u$ that satisfies $\bE \int_0^T \Hnorm{u(s)}^2 \ud s < \infty$.\\
We assume that $\dpiCoeff \colon H \times E \ra V$ is a measurable function. In contrast to $K$, there will be no growth or Lipschitz assumptions on $\dpiCoeff$. We introduce the additional term  with coefficient $\dpiCoeff$ on the right-hand side of \eqref{spde} to represent the influence of large jumps from \levy noise. If the \levy noise is not square integrable, then the influence of large jumps cannot be represented without this additional term; see e.g., \cite{PZvsIWreview} or \cite{PZ}. 
Now we introduce the nonlinear term $F: X \ra X'$. 
{We} assume that $F$ satisfies the following conditions:
\begin{itemize}	
\item (F1) (Monotonicity): For all $ v,w\in X$ we have
\begin{align}\label{F1}
\iprod{F(v)-F(w)}{v-w}_{X,X'}\ge 0.
\end{align}
\item (F2) (Hemicontinuity)
For all $u,v,w \in X$, the map
\begin{align}\label{F2}
\mathbb{R}: \lambda \mapsto ({F(u+\lambda v)},{w})_{X,X'} \quad \text{is continuous.}
\end{align}
\end{itemize}
In view of using the results of \cite{TT20} we also assume the existence of a functional $J (v):  X \rightarrow \mathbb{R}$, such that its G\^{a}teaux derivative $J'(v)=F(v)$. Then, given \eqref{F1}, we know from \cite{Lio69}, \cite{ET99} that for given $v\in X$, the map $v \mapsto J(v)$ is convex. For $i \in \mathbb{N}$, we denote by $D^iv$, the ordered tuple of all partial derivatives of $v$ of order $i$. {We also assume, as in \cite{TT20} that $J$ possesses the integral representation:}
\begin{align}\label{J_form}
	J( v) = \sum_{i=0}^k \int_{\mathcal{O}} \mathscr{J}_i( D^iv) \ud x,
\end{align}
where for every $i=0,1,...,k$, $\mathscr{J}_i : \mathbb{R}^{d^{i+1}} \rightarrow \mathbb{R}$ is further assumed to have the following properties: \\
(i) $\mathscr{J}_i$  is convex and attains its minima at $0 \in \R^{d^{i+1}}$ and there exist $c_1>0,c_2\geq 0$ such  that
\begin{align}\label{gen_convex}
x D^2 \mathscr{J}_i (x) x^T \geq c_1|x|_{l^{p}(\mathbb{R}^{d^{i+1}})}^{p} -c_2 \quad \forall x \in \mathbb{R}^{d^{i+1}}.
\end{align}
(ii) $\mathscr{J}_i$  satisfies the following growth conditions: there exists $c_3 \geq 0$ 
 such that $\forall x \in \mathbb{R}^{d^{i+1}}$
\begin{equation}\label{gc}
\begin{split}
\mathscr{J}_i( x) & \geq c_1|x|_{l^{p}(\mathbb{R}^{d^{i+1}}) }^{p}  -c_3 , \\
|\nabla_x \mathscr{J}_i(x)|_{l^{\frac{{p}}{{p}-1}}(\mathbb{R}^{d^{i+1}})} &\leq c_1|x|_{l^{{p}}(\mathbb{R}^{d^{i+1}})}^{{p}-1}   + c_3, \\
|x \, D^2 \mathscr{J}_i(x)|_{l^{\frac{{p}}{{p}-1}}(\mathbb{R}^{d^{i+1}})} &\leq c_1|x|_{l^{{p}}(\mathbb{R}^{d^{i+1}})}^{{p}-1} + c_3, 
\end{split}
\end{equation}
where $D^2 \mathscr{J}_i$ is the Hessian {of $\mathscr{J}_i$}. 
Observe that by assuming the above conditions we have that the hypothesis \eqref{F4} {below} is satisfied. 
Adhering to the notation $f_i(x)  : \mathbb{R}^{d^{i+1}} \rightarrow \mathbb{R}^{d^{i+1}}$ given by $f_i=\nabla_x \mathscr{J}_i(x)$ from \cite{TT20} we have for $v,\varphi \in X$,
\begin{align}\label{gen_Ff}
(F(v),\varphi) &= \sum_{i=0}^k  \int_{\mathcal{O}} \nabla \mathscr{J}_i (D^i v) \cdot D^i\varphi \, \ud x \\
&= \sum_{i=0}^k  \int_{\mathcal{O}} f_i (D^i v) \cdot D^i\varphi \,\ud x. \nonumber
\end{align}
We also assume,
\begin{itemize}
	\item (F3) (Coercivity)
We have  $\forall v\in X$,
	\begin{align}\label{F3}
	&\iprod{F(v)}{v}_{X,X'}\ge c_1\norm{v}^p_{X}.
	\end{align}
	\item (F4) (Boundedness)
	There exist {$ c_4\ge 0$} such that  $\forall v\in X$,
	\begin{align}\label{F4}
	\norm{F(v)}_{X'}	\le c_1\norm{v}^{p-1}_X +c_4.
	\end{align}
\end{itemize}
Here we would like to point out that the assumptions (F1)-(F4) are not unusual and that they have physical significance. Examples of such an operator $F$ include the monotone operator appearing in the Ladyzenskaya equations (Smagorinsky's model of turbulence) and the $p$-Laplacian. These examples are discussed at length in Section \ref{Examples}.
\section{Existence results}\label{sec:globally_monotonic_coefficients}
\subsection{The main results}
In order to prove the existence of solutions to \eqref{spde} one can first prove the existence when $\mathcal{K}=0$ and then add the noise term using the so-called piecing out argument (see e.g.\cite{INW66},\cite{IW}).
So we will first consider the following stochastic differential equation:
\begin{align}\label{eqn}
\ud u(t)+[\gamma A(u(t))+&B(u(t),u(t))
+F(u(t))]\ud t\\
&=G(u(t-))\ud W(t)+
\int_{E_0}K(
u(t-),\xi)\ud \hat\pi(t,
\xi) \quad {\text{in }\mathcal{O}\times(0,T)}\nonumber\\
u(0)&=u_0\quad \text{in }\mathcal{O}. \label{bc}
\end{align}
The precise meaning of \eqref{eqn}-\eqref{bc} appears below where we define (probabilistically) "weak" and "strong" solutions.
\begin{defn}\label{definition1}
	A pair $(\tilde{\mathscr{S}},\tilde u)$ is called a martingale solution to \eqref{eqn}-\eqref{bc} if
for a stochastic basis		$\tilde{\mathscr{S}}=(\tilde\Omega,\tilde{\mathcal{F}},(\tilde{\mathcal{F}}_t)_{t\ge 0},
		\tilde\bP,\tilde W,\tilde \pi) $, the process $\tilde u\in L^2(\tilde\Omega ,L^\infty(0,T;H)) \cap L^2([0,T]\times \tilde\Omega,\ud t\otimes \ud \tilde\bP;V)
		\cap L^p([0,T]\times \tilde\Omega,\ud t\otimes \ud \tilde\bP;X)$
	 is $\tilde{\mathcal{F}}$-adapted with c\`adl\`ag sample paths in $V'+X'$, $\tilde\bP$-a.s., and $\tilde u(0)$ is $\tilde{\mathcal{F}}_0$-measurable
		with law $\mu_0$, and the equation		
		\begin{align}\label{eqn_1}
		\tilde{u}(t)+\int_0^t[\gamma A(\tilde{u}(s))+&B(\tilde{u}(t),\tilde{u}(t))+ F(\tilde{u}(s))]\ud s\\
		&= \tilde{u}(0)
		+\int_0^tG(\tilde{u}(s-))\ud \tilde{W}(s)+
		\int_{(0,t]}\int_{E_0}
		K(\tilde{u}(s-),\xi)\ud \hat{\tilde{\pi}}(s,\xi),\nonumber
		\end{align}

 holds 
 $\tilde{\bP}$-a.s., for all $t\in [0,T]$.
\end{defn}
\begin{defn}\label{definition2}
{We are given} a stochastic basis
	$(\Omega,\mathcal{F}, (\mathcal{F}_t)_{t\ge 0},
	\bP,W,\pi) $ and a random variable\\ $u_0\in L^{2}(\Omega,\mathcal{F}_0,\bP;H)$ with law $\mu_0$. We say that an $\mathcal{F}_t$-adapted process 
	$u \in L^2(\Omega,L^\infty(0,T;H))\cap L^2([0,T]\times \Omega,\ud t\otimes \ud \bP;V)
	\cap L^p([0,T]\times \Omega,\ud t \otimes \ud \bP;X)$
	is a {(global)} pathwise solution to \eqref{eqn}-\eqref{bc} if $u$ has c\`adl\`ag sample paths in {$V'+X'$}, $\bP$-a.s., and if the equation
	\begin{align}\label{eqn2:pathwise}
	u(t)+\int_0^t[\gamma A(u(s))+&B(u(t),u(t))+ F(u(s))]\ud s\\
	&=u_0
	+\int_0^tG(u(s-))\ud W(s)+
	\int_{(0,t]}\int_{E_0}
	K(u(s-),\xi)\ud \hat\pi(s,\xi),\nonumber
	\end{align}
holds 
$\bP$-a.s., for all $t\in [0,T]$.
\end{defn}

Now we state the main results of this paper:
\begin{thm}\label{existence_martingale}
We assume that we are given an $ \mathcal{F}_0$-measurable random variable $u_0: \Omega \rightarrow {H}$ 
with law $\mu_0$ that satisfies		\begin{align}			\int_{H}|v|^2 d\mu_0(v) < \infty.	\end{align}
	We also assume that $G,K$ satisfy the hypotheses \eqref{eqn:G,K_H_growth}-\eqref{eqn:G,K_H_Lipschitz}, the term $B$ satisfies \eqref{eqn:B_bound2}-\eqref{eqn:B_bound3}, {$\gamma>0$ and the term $A$ satisfies \eqref{L2.1},}
	and that the nonlinear term  $F$ satisfies the conditions \eqref{F1}$-$\eqref{F4}.
	Then for any $p>2$ there exists a martingale solution $\tilde u$ to \eqref{eqn}-\eqref{bc} in the sense of Definition \ref{definition1}.
\end{thm}

\begin{thm}\label{existence}
We assume that the hypotheses of Theorem \ref{existence_martingale} hold. We additionally assume that the term $B$ satisfies \eqref{rs}.
Then for any $p>2$, there exists a unique global pathwise solution $u$ to \eqref{eqn}-\eqref{bc}, in the sense of Definition \ref{definition2}, {with initial condition $u_0$}.
\end{thm}
\section{Galerkin scheme}\label{galerkin}
	We describe here the Galerkin approximation of the problem, then derive moment estimates, estimates in fractional Sobolev spaces, tightness of the Galerkin approximation and then pass to the limit to obtain the existence of solutions in Theorem \ref{existence_martingale}. \\
	To construct the Galerkin scheme for equation \eqref{eqn}-\eqref{bc} we use the orthonormal basis {$(w_k)^\infty_{k=1}$} of $H$ which consists of the eigenvectors of $A$ as described in Section \ref{sec:functional_framework}. Consider the finite dimensional subspace of $H$:
\begin{equation*}
H_n:
=\text{span}\{w_1,...,w_n\}, n\ge 1.
\end{equation*}
Let $P_n: H\rightarrow H_n$ be the finite dimensional projection of $H$ onto $H_n$, i.e.,
\begin{equation}\label{e11.12}
P_nv:=\sum_{i=1}^n
\iprod{v}{w_i}w_i, \text{   }v\in H.
\end{equation}
Note that since {the} $w_k$ are eigenvectors of $A$, $P_n$ is also an orthogonal projector in $V$, $D(A)$, $V'$ onto  $H_n$. \textit{However this property is not true for the space $X$ and this is precisely what prevents us from applying {directly} the Galerkin scheme to \eqref{eqn}-\eqref{bc} in the traditional way.} Hence instead we consider the following truncated stochastic equations:
	\begin{equation}\label{approx_eqn}\begin{split}
	\begin{cases}
	&\ud u(t)+[\gamma A(u(t))+\theta_{\tilde{R}}(|u|)B(u(t),u(t)) +F^R(u(t))]\ud t \\
	&\hspace{1.5in}=G(u(t-))\ud W(t)+ \int_{E_0}K(u(t-),\xi)\ud \hat\pi(t,\xi)\\
	& \hspace{1.2in} u(0)=u_0.
\end{cases}	
\end{split}\end{equation}
Here, for the parameter $\tilde{R}>0$, $\theta_{\tilde{R}}:[0,\infty) \rightarrow [0,1]$ is a Lipschitz function with constant $1$ such that $\theta_{\tilde R} \equiv 1$ on $[0,\tilde{R}]$ and $\theta_{\tilde{R}}\equiv 0$ on $[2\tilde{R},\infty)$. We will frequently denote $\theta_{\tilde{R}}(|u|)B(u,u)$ by $B_{\tilde{R}}(u,u)$. Given another parameter $R>0$, $F^R: V \rightarrow V'$ is a monotone, hemicontinuous and coercive operator on $V$ defined later and approximating $F$ as $R \rightarrow \infty$  {as described in \cite{TT20} and below}.\\ Like in  \eqref{gen_Ff}, for well-chosen {functions} $f^R_i :\mathbb{R}^{d^{i+1}} \rightarrow \mathbb{R}^{d^{i+1}}$ we consider $F^R:V \rightarrow V'$ of the form
	\begin{align}\label{gen_def_FR}
	(F^R(v),\varphi) &= \sum_{i=0}^k \int_{\mathcal{O}} f^R_i(D^iv) \cdot D^i\varphi \, \ud x \quad \text{for }v,\varphi \in V.
	\end{align}
	
The choice of the functions $f^R_i$ is made by the following construction from \cite{TT20}. \\
For each $i=0,..,k$ we choose $f_i^{R}:\R^{d^{i+1}} \rightarrow \R^{d^{i+1}}$ as follows: For $x \in \R^{d^{i+1}}$
	\begin{align}\label{gen_def_fR}
	f_i^{R}(x)=f_i(x)\mathbbm{1}_{\{|x|_{p}<R\}} +  \left[ f_i \left(\frac{Rx}{|x|_{p}}\right) + \left( 1-\frac{R}{|x|_{p}}\right) x  D^2 \mathscr{J}_i\left(\frac{Rx}{|x|_{p}}\right) \right] \mathbbm{1}_{\{|x|_{p}>R\}},
	\end{align}	
where for brevity we use the notation $|\cdot|_{p} \, :=|\cdot|_{l^{p}(\R^{d^{i+1}})}$ given $x \in R^{{d^{i+1}}}$ and $i=0,1,...,k$. 
	We also present here in brief a few results proved in \cite{TT20} that we will make use of:
	\begin{itemize}	
		\item (FR1) (Monotonicity): We have $\forall u,v\in V$,
		\begin{align}\label{FR1}
		\Vdualpair{F^R(u)-F^R(v)}{u-v}\ge 0.
		\end{align} 
	{\item (FR2) (Hemicontinuity)
		For all $u,v,w \in V$, the map 
		\begin{align}\label{FR2}
		\mathbb{R}: \lambda \mapsto \Vdualpair{F^R(u+\lambda v)}{w} \quad \text{is continuous.}
		\end{align}}
			\item (FR3) (Equicoercivity)
			With an adaptation of the proof of Theorem 2.2 in \cite{TT20} we can show that there exist $c_5=c_5(\mathcal{O},d) \ge 0$ independent of $R$ and some $R_0>1$ such that for any $R>R_0$ and any $ v\in V$,
			\begin{align}\label{FR3}
			&\Vdualpair{F^R(v)}{v}\ge c_1\norm{v}^2-c_5.
			\end{align}
			The importance of \eqref{FR3} is mentioned in Remark \ref{imp_equi}.
			\item (FR4) (Boundedness)
		We have for any $ v\in V$, 
			\begin{align}\label{FR4}
			|{F^R(v)}|_{V'}	\le c_1R^{p-2}\norm{v}.
			\end{align}
		\end{itemize}		
	Now, for fixed $R>R_0$ and $\tilde{R}>0$ we can introduce the $n^{th}$ Galerkin approximation {of} \eqref{approx_eqn} $u^n:=u^{n,R,\tilde{R}}$ that satisfies		
	\begin{equation}\label{approximationsolution}
	\begin{cases}
	&\ud u^n(t)+\left[\gamma A(u^n(t))+P_n B_{\tilde{R}}(u^n(t),u^n(t)) + {P_n F^R(u^n(t))}\right]\ud t \\
	&\hspace{2in}=P_nG(u^n(t-))\ud W(t)+\int_{E_0} P_nK(u^n(t-),\xi)\ud \hat\pi(t,\xi),\\
	& u^n(0)=P_nu_0=:u^n_0.
	\end{cases}
	\end{equation}
{Since $B_{\tilde{R}}$ is globally Lipschitz on $H_n$, $F^R$ satisfies \eqref{FR1}-\eqref{FR4} and $G$ and $K$ satisfy the Lipschitz condition \eqref{eqn:G,K_H_Lipschitz}, the existence of a unique pathwise solution $u^n$ to the SDE \eqref{approximationsolution} is guaranteed thanks to Theorem 1.2 in \cite{KRZ} and its generalization in Theorem 3.1 in \cite{ABW}. {(See also Theorem 3.5 in \cite{CNTT20} for a local existence result.)}
\subsection{Moment estimates}\label{ss11.4}
In order to construct a solution to the system \eqref{approx_eqn}, we need to perform a priori estimates for the approximating sequence $u^n:=u^{n,R,\tilde{R}}$ {and then pass to the limit $n \rightarrow \infty$, $R \rightarrow \infty$, $\tilde R \rightarrow \infty$ (in that order).}
\begin{lem}\label{uniformboundL2}
	\textnormal{($L^2$-regularity)}
	Suppose that F satisfies the conditions \eqref{J_form}-\eqref{gc}, so that {the} family of operators $F^R$ satisfies (FR1)-(FR4). Then there exists a constant $C(T,\rho,c_1,c_5)> 0$, independent of $n$ and in fact independent of $R$ and $\tilde{R}$ such that for every $n \geq 1$:	
	\begin{align}\label{e11.14}
	&
	\bE\sup_{s\in[0,T]}
	|u^n(s)| ^{2}+
	(\gamma+c_1)\bE\int_0^T \norm{u^n(s)}^2  \ud s \le C\bE|u_0|^2=: K_1(u_0,T,\rho,c_1,c_5).
	\end{align}
\end{lem}
\begin{proof}
	{We begin by applying the It\^{o} formula to the process in \eqref{approximationsolution}. To be precise, here we use the It\^{o} formula in the case of solutions to SDEs with \levy noise as stated in Theorem 2.19 in \cite{CNTT20} and the special case in which $\psi(u)=|u|^2$ as stated in Corollary 1 {in \cite{CNTT20}}.  See also Theorem I.3.1 in \cite{KR07} and \cite{Metivier82} for the  It\^{o} formula in the general context of semi-martingales. We obtain}
	\begin{equation}\label{ito}
	\begin{split}
	&|{u^n}|^2+2\int_0^t
	\gamma \iprod{Au^n(s)}{u^n(s)}\ud s
	+2\int_0^t
	\iprod{P_n B_{\tilde{R}}(u^n(s))}{u^n(s)}\ud s +2\int_0^t	\iprod{P_n F^R(u^n(s))}{u^n(s)}_{}\ud s \\
	&=|u^n_0|^2 +\int_{(0,t]}\int_{E_0}	2
	\iprod{u^n(s)}{P_nK(u^n(s-),\xi)\ud \hat\pi(s,\xi)}+\int_0^t2\iprod{u^n(s)}{P_nG(u^n(s-))\ud W(s)}\\
	&+\int_0^t
	\norm{P_nG(u^n(s))}_{L_2(U_0,H)}^2\ud s  +\int_{(0,t]}\int_{E_0}
	|{P_nK(u^n(s-),\xi)}|^2 \ud \pi(s,\xi).
		\end{split}
	\end{equation}
Here we {made} use of the cancellation property $\iprod{B_{\tilde{R}}(u,u)}{u}=0$. {Now we use} the equicoercivity result \eqref{FR3} along with the equation \eqref{L2.1},{ we take} the supremum in time over $[0,r]$ for a fixed $r \in[0,T]$ and then take expectation of both sides; {we find} 
	\begin{align}
	\bE[\sup_{0\le t\le r}\Hnorm{u^n(t)}^2] +& 2\bE\int_0^r {\left((\gamma +c_1)\norm{u^n(s)}^2-c_5 \right)} \ud s \label{e11.16} \\
	&	\le \bE[\Hnorm{u^n_0}^2] + \bE\sup_{0\le t\le r}[I(t)+J(t)+K(t)+L(t)],\nonumber
	\end{align}
	where
	\begin{equation*}
	\begin{aligned}
	 I(t) &=
	\int_0^t  \norm{P_nG(u^n(s))}_{L_2(U_0,H)}^2\ud s, \\
	 J(t)&=\int_{(0,t]}\int_{E_0}	|{P_nK(u^n(s-),\xi)}|^2 \ud \pi(s,\xi), \\
	 K(t)&=\int_0^t\iprod{u^n(s)}{P_nG(u^n(s-))\ud W(s)}, \\
	L(t)&=\int_{(0,t]}\int_{E_0}
		\iprod{u^n(s)}{P_nK(u^n(s-),\xi)\ud \hat\pi(s,\xi)}.
		\end{aligned}
	\end{equation*}	
	We estimate the terms $I(t)$ and $J(t)$
	based on the assumption \eqref{eqn:G,K_H_growth} along with Young's inequality as follows
	\begin{align}
	\nonumber\bE\sup_{0\le t\le r}
	\int_0^t 
	\norm{P_nG(u^n(s))
	}_{L_2(U_0,H)}^2 \ud s
	&\lesssim \bE \sup_{0\leq t \leq r}\int_0^t
\norm{G(u^n(s))}_{L_2(U_0,H)}^2 \ud s \nonumber\\
	&= \bE\int_0^r
		\norm{G(u^n(s))}_{L_2(U_0,H)}^2 \ud s \nonumber\\
	&\lesssim \bE \int_0^r
	\rho(|u^n(s)|^2+1) \ud s. \label{I(t)}
	\end{align}
	Similarly for $J(t)$ we use the isometry property {(see e.g. Theorem 2.16 in \cite{CNTT20}
		) }to obtain,
	\begin{align}
	\bE\sup_{0\le t\le r}\int_{(0,t]}\int_{E_0}
	|P_nK(u^n(s-),\xi)|^2 \ud\pi(s,\xi)  &\leq \bE\sup_{0\le t\le r}\int_{(0,t]}\int_{E_0}
	|K(u^n(s-),\xi)|^2 \ud\pi(s,\xi)   \nonumber\\
	&= \bE
	\int_{(0,r]}\int_{E_0} 
	|K(u^n(s-),\xi)|^2\ud\nu(\xi)\ud s \nonumber\\
	&\lesssim \bE \int_{(0,r]}
	\rho(|u^n(s)|^2+1) \ud s. 
	\label{K(t)}
	\end{align}
{Note that in \eqref{I(t)} and \eqref{K(t)} we used the fact that $P_n$ is a projector on $H$.}	The next two martingale terms  $K(t)$ and $L(t)$ are treated by using the Burkholder-Davis-Gundy (BDG) inequality (see e.g. \cite{PZ}) and Young's inequality.	Observe that for some constant $C(\rho)>0$ independent of $n$ we obtain	
	\begin{align}
	 \bE\sup_{0\le t\le r}\left|
	\int_0^t\iprod{u^n(s)}{P_nG(u^n(s-))\ud W(s)}\right|
	&\lesssim \bE\left(\int_0^r
	\iprod{u^n(t)}{P_nG(u^n(t-))}^2 \ud t\right)^{\frac12} \nonumber\\
	&  \lesssim \bE\left(\int_0^r \norm{P_nG(u^n(t))}^2_{L_2(U_0,H)}|u^n(t)|^2\ud t\right)^{\frac12} \nonumber\\
	&  \lesssim \bE\sup_{0\le t\le r}|u^n(t)|\left( \int_0^r\norm{G(u^n(t))}^2_{L_2(U_0,H)}\ud t \right)^{\frac12}\nonumber,
	\end{align}which implies for $K(t)$ that,\begin{align*} \bE\sup_{0\le t\le r}\left|
	\int_0^t\iprod{u^n(s)}{P_nG(u^n(s-))\ud W(s)}\right|&\le \frac14 \bE\sup_{0\leq t \leq r}|u^n(t)|^2+C(\rho)\bE\int_0^r\left(1+|u^n(s)|^2\right) \ud s.
	\end{align*}
	Similarly using the Burkholder-Davis-Gundy inequality, (see Proposition 2 pg. 5650 in \cite{CNTT20}) we obtain for the term $L(t)$: 
	\begin{align}	
	&\bE  \sup_{0\le t\le r}\left|\int_{(0,t]}\int_{E_0} 
	\iprod{u^n(s)}{P_nK(u^n(s-),\xi)} \ud \hat\pi(s,\xi)\right| \nonumber\\
	&\qquad \qquad \qquad\lesssim  \bE \left(\int_{(0,r]}\int_{E_0}
	|u^n(s)|^{2}
	\Hnorm{P_nK(u^n(s-),\xi)}^2\ud \pi(s,\xi)\right)^{\frac 12}\nonumber\\
	&\qquad \qquad \qquad\le \frac14 \bE\sup_{0\le t\le r}
	|{u^n(t)}|^2+ C\bE\bigg[\int_{(0,r]}\int_{E_0}\Hnorm{K(u^n(s-),\xi)}^2\ud \pi(s,\xi)\bigg] \nonumber\\
		&\qquad \qquad \qquad= \frac14 \bE\sup_{0\le t\le r}
	|{u^n(t)}|^2+ C\bE\bigg[\int_{(0,r]}\int_{E_0}\Hnorm{K(u^n(s-),\xi)}^2\ud	\nu(\xi) ds \bigg] \nonumber\\
		&\qquad \qquad \qquad \le \frac14 \bE\sup_{0\leq t \leq r}|u^n(t)|^2+C(\rho)\bE\int_0^r\left(1+|u^n(s)|^2\right) \ud s. \label{N(t)}
	\end{align}
Collecting all relations from \eqref{I(t)} through \eqref{N(t)} and applying in  \eqref{e11.16}, we obtain
	\begin{align*}
	&\bE \sup_{0\le s\le r}|u^n(s)|^2+{(\gamma+c_1)}\bE\int_0^r\norm{u^n(s)}^2\ud s	\\
	&\hspace{2in}\le  ~\bE|u^n_0|^2
	+\frac12 \bE\sup_{0\le s\le r}|u^n(s)|^2+C(\rho)\bE\int_0^r(1+|u^n(s)|^2)\ud s. 
	\end{align*}
	Multiplying both sides by 2 and using the fact that $\bE|u_0^n|^2 \leq \bE|u_0|^2$, {we obtain} for some constant $C(T,\rho,c_1,c_5)>0$ independent of $n, R, \tilde R$
	\begin{equation}
	\begin{aligned}
	& \bE\sup_{0\le s\le r}|u^n(s)|^2+{2(\gamma +c_1)}\int_0^r \bE \norm{u^n(s)}^2  \ud s
	  \le ~ C\left(\bE\Hnorm{u_0}^2 + {\int_0^r \sup_{0 \leq t\leq s}\bE|u^n(t)|^2 \ud s}\right)+C.
	\end{aligned}
	\end{equation}
	The proof is complete after applying the  Gr\"onwall lemma 
	 to $\bE\sup_{0\le s\le r}\Hnorm{u^n(s)}^2$.
\end{proof} 
\begin{rem}
	Observe that the absence of the linear term $A$, {that is if $\gamma=0$}, does not change this result thanks to the equicoercivity result stated in \eqref{FR3} and derived in \cite{TT20} { for the operators $F^R$}. This fact will be used in Section \ref{Examples}. See also Remark \ref{imp_equi} for an explanation.
\end{rem}

\subsection{Estimates in fractional Sobolev spaces}\label{ss11.5}
To obtain compactness, we will use a particular characterization of the fractional Sobolev space ${W^{\alpha,m}(0,T;Z)}$ where $Z$ is a Banach space.\\ Given $m>1$ and $\alpha \in (0,1)$, $W^{\alpha,m}(0,T;Z)$ is the Sobolev space consisting of {the} $v \in L^m(0,T;Z)$ such that the (Gagliardo) seminorm $\int_0^T \int_0^T \frac{|v(t)-v(s)|^m}{|t-s|^{1+\alpha m}}\ud t\ud s$ is finite. This space is then endowed with the norm
\begin{align}
\|v\|^m_{W^{\alpha,m}(0,T;Z)}=\|v\|_{L^m(0,T;Z)}^m + \int_0^T \int_0^T \frac{|v(t)-v(s)|^m}{|t-s|^{1+\alpha m}}\ud t\ud s.
\end{align}
We denote $W^{\alpha,2}=:H^\alpha$. Now we state the following lemma:


\begin{lem}\label{Sobolev1}
	Let $\mathcal{Y}_0\subset\mathcal{Y}\subset\mathcal{Y}_1$ be Banach spaces, $\mathcal{Y}_0$ and $\mathcal{Y}_1$ reflexive with compact embedding of $\mathcal{Y}_0$ in $\mathcal{Y}$. Let $s, m >1$ and $\alpha \in (0,1)$ be given. Let
	$\mathcal{G}$ be the space
	\begin{equation*}
	\mathcal{G}=L^\infty(0,T;\mathcal{Y})\cap L^s(0,T;\mathcal{Y}_0)
	\cap W^{\alpha,m}(0,T;\mathcal{Y}_1),
	\end{equation*}
	endowed with the natural norm. Then the embedding of $\mathcal{G}$ in $L^r(0,T;\mathcal{Y})$ is compact, $\forall 1 \leq r<\infty$.
\end{lem}
\begin{proof}
	Let $q=\min(m,s)>1$ and define $\mathscr{G}_q=L^q(0,T;\mathcal{Y}_0) \cap W^{\alpha,q}(0,T;\mathcal{Y}_1)$. Then it is shown in \cite{Tem95} (Section 13.3 Theorem 13.2) and \cite{FG95} (Theorem 2.1) that the injection of $\mathscr{G}_q$ in $L^q(0,T;\mathcal{Y})$ is compact. Now if $\{v_k\}_{k\geq 1}$ is a bounded sequence in {$\mathcal{G}$} that converges to $v$ weakly in $ \mathscr{G}_q$ and strongly in $L^q(0,T;\mathcal{Y})$, then up to the extraction of a subsequence, $v_k(t) \rightarrow v(t)$ strongly in $\mathcal{Y}$ {for }a.e. $t \in (0,T)$. Since $v_k$ is also bounded in $L^\infty(0,T;\mathcal{Y})$,  using the Vitali convergence theorem, we have that $v_m \rightarrow v$ strongly in $L^r(0,T;\mathcal{Y})$ for every $r\in [1,\infty)$.
\end{proof}
In what follows, we will apply the above Lemma with $\mathcal{Y}_0=V$, $ \mathcal{Y}_1=D(A)'$, $\mathcal{Y}=H$, $\alpha\in (0,\frac 12)$, $s=m=2$. We will use the above lemma again in Sections \ref{sec:R} and \ref{sec:tildeR} as well.   

We are now in position to show that the {approximating} sequence $\{u^n\}_{n=1}^\infty$ {is} bounded in the space $L^{2}(\Omega;H^{\alpha}(0,T;D(A)'))$
for $\alpha\in (0,\frac 12)$.  For that purpose, we {recall} the following proposition { essentially borrowed from \cite{FG95}:}
\begin{prop}\label{noisefraction}
	For every $\alpha \in (0,\frac 12)$ and $1\leq \beta \leq 2$ there exists a constant $C(\alpha,\beta,\rho,K_1)$, independent of $n, R, \tilde R$, such that
	\begin{equation}\label{IW}
	\bE\left\|\int_0^tP_n G(u^n(s))\ud W(s)\right\|^{\beta}_{W^{\alpha,\beta}
		(0,T;H)}\le C,
	\end{equation}
	and
	\begin{equation}\label{Ipi}
	\bE\left\|\int_{(0,t]}\int_{E_0}P_n K(u^n(s-),\xi)\ud \hat \pi(s,\xi)\right\|^{\beta}_{W^{\alpha,\beta}(0,T;H)}\le C.
	\end{equation}
\end{prop}
\begin{proof}
Using Lemma 2.1 in \cite{FG95}, the assumption \eqref{eqn:G,K_H_growth} and Lemma \ref{uniformboundL2} we obtain
	\begin{equation} 
	\begin{aligned}
	~ \bE\left\|\int_0^tP_nG(u^n(s))\ud W(s)\right\|^{\beta}_{W^{\alpha,\beta}
		(0,T;H)} &\lesssim \bE\int_0^T \norm{P_nG(u^n)}^{\beta}_{L_2(U_0,H)} \ud s \\
	&  \lesssim \left( \bE \int_0^T(1+|u^n(s)|^{2}) \ud s\right)^{\frac{\beta}{2}} \leq C.
	\end{aligned}
	\end{equation}
A continuity result similar to Lemma 2.1 in \cite{FG95} holds for the stochastic integral with respect to a compensated Poisson random measure. For more details of the case $\beta=2$, we refer the reader to Proposition A.4 in \cite{EulerPaper}. Using the assumption \eqref{eqn:G,K_H_growth} and Lemma \ref{uniformboundL2} we thus obtain 
	\begin{align}
\bE\left\|\int_{(0,t]}\int_{E_0}P_nK(u^n(s-),\xi)\ud \hat \pi(s,\xi)\right\|^{\beta}_{W^{\alpha,\beta}(0,T;H)} &\lesssim \bE\int_0^T \int_{E_0}|P_nK(u^n(s-),\xi)|^{\beta} \ud\nu(\xi) \ud s \nonumber \\
&  \lesssim \left( \bE \int_0^T(1+|u^n(s)|^{2}) \ud s\right)^{\frac{\beta}{2}} \leq C.
\end{align}

\end{proof}

\begin{prop}\label{boundinW}
	Let $u^n$ be the solution to \eqref{approximationsolution} with an inital data $u_0\in L^2(\Omega,\mathcal{F}_0,\bP;H)$.
	Then for every $\alpha\in (0,\frac 12)$ there exists a constant $C=C(\alpha,\gamma,T,R,\tilde{R},K_1)>0$ such that
	{\begin{equation}\label{frac_u}
	\sup_{n\ge 1}\bE\norm{u^n}^{2}_{H^{\alpha}(0,T;{{{}D(A)'}})}\le C.
	\end{equation}}
\end{prop}

\begin{proof}
	In light of Proposition \ref{noisefraction}, it is sufficient to prove that
	\begin{equation} \label{eqn:eqn:fractional_Sobolev_estimate.3}
	\sup_{n \geq 1} \bE \left\|  u^n - \int_0^{\cdot}P_n G(u^n(s))\ud W(s) - \int_{(0,{\cdot }]}\int_{E_0}P_n K(u^n(s-) ,\xi)\ud \hat \pi(s,\xi)\right\| ^{2}_{H^{\alpha}(0,T;{D(A)'})} < \infty.
	\end{equation}
	Let $g^n := u^n - \int_0^tP_n G(u^n)\ud W(s) - \int_{(0,t]}\int_{E_0}P_n K(u^n,\xi)\ud \hat \pi(s,\xi)$. Since $\alpha<1$ we use the fact that the embedding $H^{1}\subset H^{\alpha}$ is continuous 
	 to obtain,	
	\begin{align}
	 ~\norm{g^n}^{2}_{H^{\alpha}(0,T;{D(A)'})}
	&\lesssim  \norm{u_0^n -\int_0^{\cdot} \gamma Au^n(s) \ud s  - \int_0^{\cdot} P_n B_{\tilde{R}}(u^n(s)) \ud s - \int_0^{\cdot} P_nF^R(u^n(s)) \ud s}^{2}_{H^{1}(0,T;{D(A)'})} \nonumber \\\nonumber
	&\lesssim  ~ \norm{u_0^n -\int_0^{\cdot}  \gamma Au^n(s) \ud s  - \int_0^{\cdot}  P_n B_{\tilde{R}}(u^n(s)) \ud s - \int_0^{\cdot} P_nF^R(u^n(s)) \ud s}^{2}_{L^{2}(0,T;{{}D(A)'})}  \\
	& \qquad  + \int_0^T \gamma\Hnorm{Au^n(s)}_{{D(A)'}}^2\ud s + \int_0^T  \Hnorm{B_{\tilde{R}}(u^n(s))}_{{D(A)'}}^2 \ud s + \int_0^T \Hnorm{F^R(u^n(s))}_{{{}D(A)'}}^2\ud s \nonumber\\
	\lesssim & ~ \Hnorm{u_0}^2+ \int_0^T \gamma\Hnorm{Au^n(s)}^2_{{{}D(A)'}} \ud s + \int_0^T \Hnorm{B_{\tilde{R}}(u^n(s))}_{{D(A)'}}^2 \ud s + \int_0^T \Hnorm{F^R(u^n(s))}^2_{{D(A)'}} \ud s \nonumber  \\
	=: & ~\Hnorm{u_0}^2 + \mathcal{K}_1 + \mathcal{K}_2 + \mathcal{K}_3,\label{gn}
	\end{align}
	where the hidden constants depend on $T$ but not on $n$. Next, using Lemma \ref{uniformboundL2} we obtain the existence of a constant $C = C(T,\gamma,K_1) > 0$, independent of $n, R, \tilde {R}$, such that
	\begin{equation}\label{eqn:fractional_Sobolev_A_est}
	\bE[\mathcal{K}_1]\lesssim \bE\int_0^T
	\gamma\Hnorm{Au^n(s)}_{V'}^2\ud s \leq \gamma\bE\int_0^T\|u^n(s)\|^2 \ud s \leq C.
	\end{equation}
	
	We estimate $\mathcal{K}_2$ using the available estimates for $B$, \eqref{eqn:B_bound2}, as follows.
	\begin{align}
	\bE\int_0^T|B_{\tilde{R}}(u^n(s))|_{D(A)'}^2 \ud s \lesssim \bE \int_0^T \big(\theta_{\tilde{R}}(|u^n|)\big)^2|u^n(s)|^{2} \|u^n(s)\|^2 \ud s \lesssim \tilde{R}^2\bE \int_0^T\|u^n(s)\|^2 \ud s.\label{frac_B}
	\end{align}
	 {The above follows from an application of H\"older's inequality and the fact that $\theta_{\tilde R}(x)$ vanishes when $x \geq 2\tilde R$. Then we use Lemma \ref{uniformboundL2} which shows that the right-hand side of \eqref{frac_B}} is bounded by a constant $C=C(\tilde{R},T,K_1)>0$ independent of $n$ and $R$.\\
	 	We estimate $\mathcal{K}_3$ using the property \eqref{FR4} and Lemma \ref{uniformboundL2} and obtain for some $C=C(R,T,p,K_1)>0$ independent of $n$ and $\tilde R$,
	\begin{equation}
	\begin{aligned}\label{frac_F}
	\bE[\mathcal{K}_3] \lesssim & ~ \bE\int_0^T
	|F^R(u^n(s))|^2_{V'}\ud s\lesssim R^{2p-4}\bE\int_0^T\norm{u^n(s)}^2 \ud s\le C.
	\end{aligned}
	\end{equation}
	Thus using Proposition \ref{noisefraction} with $\beta=2$ and \eqref{eqn:fractional_Sobolev_A_est}-\eqref{frac_F} we obtain the desired result \eqref{frac_u}. {Although there are sharper estimates available for the term $\mathcal{K}_3$, \eqref{frac_F} suffices at this point.	}
\end{proof}
\subsection{Tightness}\label{tightness}
\begin{prop}\label{tightnessinV}
	Let $(u^n)_{n=1}^\infty$ be the solution to the equations \eqref{approximationsolution} with initial data\\ $u_0\in L^2(\Omega,\mathcal{F}_0,\bP;H)$. Then the laws of
	$(u^n)_{n=1}^\infty$ are tight in $L^2(0,T;H).$
\end{prop}
\begin{proof}
The proof follows standard procedure (see e.g. \cite{DGHT}). For $M>0$, let
	\begin{equation}
	{\mathcal{B}}_M:=\{u\in L^2(0,T;V)\cap H^{\alpha}(0,T;D(A)'):\Hnorm{u}^2_{L^2(0,T;V)}+\Hnorm{u}^2_
	{H^{\alpha}(0,T;D(A)')}\le M\}.
	\end{equation}
Thanks to Lemma \ref{Sobolev1}, we know that ${\mathcal{B}}_M$ is a compact subset of $L^2(0,T;H)$. By using Chebyshev's inequality, we see that	
	\begin{equation}
	\begin{aligned}
	\bP\left[u^n\not\in {\mathcal{B}}_M\right]\le & \bP\left[
	\Hnorm{u^n}^2_{L^2(0,T;V)}\ge \frac{M}{2}\right]+
	\bP\left[
	\Hnorm{u^n}^2_{H^{\alpha}(0,T;D(A)')}\ge \frac{M}{2}\right]	\\
	\le & \frac{4}{M^2}\bE\left[
	\Hnorm{u^n}^2_{L^2(0,T;V)}+\Hnorm{u^n}^2_
	{H^{\alpha}(0,T;D(A)')}\right].
	\end{aligned}
	\end{equation}
	By Lemma \ref{uniformboundL2} and Proposition \ref{boundinW}, there exists a constant $C=C(T,M,R,\tilde R,K_1)>0$ independent of $n$ such that for every positive integer $n\in \mathbb{N}$,
	\begin{equation}
	\bP\bigg[u^n\not\in {\mathcal{B}}_M \bigg]\le \frac {C}{M^2}.
	\end{equation}
This  proves that the laws of $u^n$ are tight in $L^2(0,T;H).$\end{proof}
However in order to obtain a candidate solution that is \cadlag we will need to establish tightness of the laws of $u^n$ on a certain space of functions that are \cadlag in time endowed with the Skorohod topology;  {see below and in \cite{Billingsley99}}.\\
 For that purpose we will show that the Aldous condition, [cf.\cite{Aldous}, \cite{Aldous2}, \cite{Metivier88} (Theorem 3.2, page 29)] is satisfied. {The following is an easy way to guarantee that the Aldous condition is satisfied.} 
\begin{lem}
	Let {$(\Xi,||\cdot||_\Xi)$} be a separable Banach space and let $(X_l)_{l \in \N}$ be a sequence of $\, \Xi$-valued random variables. Assume that for every sequence $(\tau_l)_{l\in\mathbb{N}}$
	of $\mathcal{F}$-stopping times with $\tau_l\le T$ and {for every} $\delta \ge 0$ the following condition holds
	\begin{equation}\label{Aldous}
	\sup_{l \geq 1}\bE(||X_l(\tau_l+\delta)-
	X_l(\tau_l)||^\alpha_{\Xi}\le C\delta^{\epsilon},
	\end{equation}
	for some $\alpha, {\epsilon}>0$ and a constant $C>0$.
Then the sequence $(X_l)_{l\in \mathbb{N}}$
	satisfies the Aldous condition in the space $\Xi$ and thus the laws of $X_l$ form a tight sequence on $\mathcal{D}([0,T];\Xi)$ endowed with the Skorohod topology.\end{lem}
 See also Theorem 13.2 of \cite{Tem95} for a compactness result in the deterministic setting {using a condition} analogous to the condition \eqref{Aldous}.

We will show that the laws of the approximate solutions $u^n$ denoted by $\mu_u^n$
are tight as probability measures on the phase space {$\sX$} where
\begin{equation}\label{chiu}
\sX = L^2(0,T;H)\cap \mathcal{D}([0, T];V'),
\end{equation}
and
\begin{equation}\label{laws_n}
\mu_u^n(\cdot)=\bP(u^n\in\cdot) \in Pr\left(L^2(0,T;H)\cap
\mathcal{D}([0, T];V')\right).
\end{equation}
Here $Pr(\mathcal{S})$ denotes the set of probability measures on the metric space $\mathcal{S}$. For a complete, separable metric space $\mathcal{S}$ we denote
by $\mathcal{D}([0, T]; \mathcal{S})$ the space of functions $u: [0, T] \rightarrow \mathcal{S}$ that are {right-continuous on $[0, T)$ and have left-limits at every point in $(0, T]$}. The space {$\mathcal{D}([0, T]; \mathcal{S})$} is endowed with the Skorokhod topology, which makes $\mathcal{D}([0, T]; \mathcal{S})$ separable and metrizable by a complete metric (see e.g. Chapter 3 in \cite{Billingsley99}).
\begin{prop}\label{tightnessinD}
	Suppose $\bE|u_0|^2<\infty$. Then the laws $\{\mu_u^n\}_{n\geq 1}$ of the Galerkin approximations form a tight sequence of probability measures on $\mathcal{D}([0,T];V')$
	endowed with the Skorohod topology.
\end{prop}
\begin{proof}
	Let $(\tau_n)_{n\geq 1}$ be a sequence of stopping times where $0\le \tau_n\le T$ and {let us integrate} the equation \eqref{approximationsolution} from 0 to t for $t \in[0,T]$; we obtain
	\begin{equation}
	\begin{aligned}\label{Mn}
	u^n(t)& =  u^n_0 - \int_0^t \gamma A(u^n(s))\ud s - \int_0^t P_nB_{\tilde{R}}(u^n(s),u^n(s))\ud s - \int_0^t P_n F^R(u^n(s))\ud s
	\\&\quad +\int_0^tG(u^n(s-))\ud W(s) + \int_{(0,t]}\int_{E_0}
	K(u^n(s-),\xi)\ud\hat\pi(s,\xi)
	\\&=: u^n_0+\mathcal{M}_1(t)+
	\mathcal{M}_2(t)+
	\mathcal{M}_3(t)+
	\mathcal{M}_4(t)+\mathcal{M}_5(t).
	\end{aligned}
	\end{equation}
	
	\noindent Let $\delta >0$ {be given}. We need to show that {a suitable Aldous condition is satisfied and for that purpose we need to show that }all the terms in the above equation satisfy
	\eqref{Aldous} {with suitable $\alpha, \epsilon$ and $\Xi=V'$}. 
	We first treat $\mathcal{M}_1$ by using the H\"{o}lder inequality
	\begin{equation}
	\begin{aligned}\label{M1}
	 \bE\norm{\mathcal{M}_1(\tau_n+\delta)-\mathcal{M}_1(\tau_n)}_{V'}
	&\leq \gamma\bE\int_{\tau_n}^{\tau_n+\delta}|A(u^n(s))|_{V'}\ud s \\
	& \le \gamma\bE\int_{\tau_n}^{\tau_n+\delta}\norm{u^n(s)}\ud s
	\le \gamma \delta^{\frac 12}\left(\bE\int_{0}^{T}\norm{u^n(s)}^2\ud s\right)^{\frac12}. 
	\end{aligned}
	\end{equation}
	The last term is bounded independently of $n$ thanks to Lemma \ref{uniformboundL2}. Hence, $\mathcal{M}_1(t)$ satisfies \eqref{Aldous} with $\alpha=1,\epsilon=\frac 12$.	
	Next, we deal with $\mathcal{M}_2(t)$. By applying H\"older's inequality and the assumption \eqref{eqn:B_bound3}, we obtain for some $C=C(T,K_1)>0$ independent of $n$	
	\begin{align}
	\bE\norm{\mathcal{M}_2(\tau_n+\delta)-\mathcal{M}_2(\tau_n)}_{V'} 
	\le& C\bE\int_{\tau_n}^{\tau_n+\delta}|B_{\tilde{R}}(u^n(s),u^n(s))|_{V'}\ud s  \nonumber	\\
	\le& C\bE\int_{\tau_n}^{\tau_n+\delta}{\theta_{\tilde{R}}(|u^n(s)|)}|u^n(s)|^{\frac12}\norm{u^n(s)}^{\frac32}\ud s \nonumber\\
	\le& C
	\left(\bE\int_{\tau_n}^{\tau_n+\delta}{\theta_{\tilde{R}}(|u^n(s)|)}|u^n(s)|^2\ud s\right)^{\frac 14}
	\left(\bE\int_{\tau_n}^{\tau_n+\delta}||u^n(s)||^2\ud s\right)^{\frac34}\nonumber\\
	\nonumber
	\le& C\delta^{\frac14}\left(\bE \sup_{s\in[0,T]}{|u^n(s)|}^{2}\right)^{\frac14} 
\left(\bE\int_{0}^T||u^n(s)||^2\ud s\right)^{\frac34} \\\leq& C\delta^{\frac14}.\label{M2}
	\end{align}	
	The last inequality above comes from Lemma \ref{uniformboundL2}.
	Hence, $\mathcal{M}_2(t)$
	satisfies the condition
	\eqref{Aldous} for $\alpha=1,\epsilon=\frac 14$.
	Next, by making use of \eqref{FR4}, we readily obtain {that} for some $C=C({R},T,p,K_1)>0$ independent of $n$
	\begin{equation}
	\begin{aligned}\label{aldousF}
	\bE\norm{\mathcal{M}_3(\tau_n+\delta)-\mathcal{M}_3(\tau_n)}_{V'}	& \lesssim \bE\int_{\tau_n}^{\tau_n+\delta}|F^R(u^n(s))|_{V'}\ud s	\\
	& \lesssim \bE\int_{\tau_n}^{\tau_n+\delta}R^{p-2}\|u^n(s)\|\ud s	\\
	& \lesssim R^{p-2}\delta^{\frac12} \left(\bE\int_0^T \norm{u^n(s)}^2\ud s\right)^{\frac12}\\
	&\leq C \delta^{\frac12}.
	\end{aligned}
	\end{equation}
	Thus, $\mathcal{M}_3(t)$ satisfies \eqref{Aldous} with $\alpha=1,\epsilon=\frac12$. Note that for \eqref{aldousF} (and earlier for \eqref{gn}) we used the fact that $P_n$ is a projector operator in $V'$ (as well as in $H$). This is an essential point motivating the introduction of the approximation $F^R$ of $F$ (see \cite{TT20}).\\
	Next by using the It\^o isometry and \eqref{eqn:G,K_H_growth} and noticing that $V\subset H$, we find
	\begin{equation}
	\begin{aligned}
	&\bE\left[\norm{\mathcal{M}_4(\tau_n+\delta)-\mathcal{M}_4(\tau_n)}^{2}_{V'}\right]
	\lesssim \bE \left[\Hnorm{\mathcal{M}_4(\tau_n+\delta)-\mathcal{M}_4(\tau_n)}^{2}\right]
	\\
	= & \bE\left|\int_{\tau_n}^{\tau_n+\delta}
	P_nG(u^n(s))\ud W(s)\right|^2= \bE\int_{\tau_n}^{\tau_n+\delta}
	|P_nG(u^n(s))|^2_{L_2(U_0,H)}\ud s
	\\
	\lesssim &~ \bE\int_{\tau_n}^{\tau_n+\delta}\rho(|u^n(s)|^2+1)\ud s\lesssim \rho\delta\bE\left[\sup_{s\in [0,T]}(1+|u^n(s)|^2)\right] \leq C(\rho,K_1)\delta.
	\end{aligned}
	\end{equation}
	Hence, $\mathcal{M}_4(t)$ satisfies \eqref{Aldous} with $\alpha=2, \epsilon=1$.
	
	\noindent Finally, we consider the $\mathcal{M}_5(t)$ term. Using the It\^o isometry and \eqref{eqn:G,K_H_growth}, we write
	\begin{equation}
	\begin{aligned}\label{M5}
	&~ \bE\left[\norm{\mathcal{M}_5(\tau_n+\delta)-\mathcal{M}_5(\tau_n)}^{2}_{V'}\right]
	\lesssim \bE
	\left[\Hnorm{\mathcal{M}_5(\tau_n+\delta)-\mathcal{M}_5(\tau_n)}^{2}\right]
	\\
	= &~ \bE\left|\int_{(\tau_n,\tau_n+\delta]}\int_{E_0}K(u^n(s-),\xi)\ud\hat\pi(s,\xi)\right|^2
	= \bE\int_{\tau_n}^{\tau_n+\delta}\int_{E_0}|K(u^n(s-),\xi)|^2\ud\nu(\xi)\ud s
	\\
	\lesssim &~ \bE\int_{\tau_n}^{\tau_n+\delta}\rho(|u^n(s)|^2+1)\ud s\lesssim \rho\delta\bE\left[\sup_{s\in [0,T]}(1+|u^n(s)|^2)\right]\le C(\rho,K_1)\delta.
	\end{aligned}
	\end{equation}
	Thanks to Lemma \ref{uniformboundL2}, we see that $\mathcal{M}_5(t)$ satisfies the relation \eqref{Aldous} with $\alpha=2, \epsilon=1$ {and $\Xi=V'$}. Collecting all {the} estimates from \eqref{M1} to \eqref{M5}, we conclude that the laws of $(u^n)_{n=1}^\infty$ are tight in $\mathcal{D}([0,T];V')$ endowed with the Skorohod topology.
\end{proof}
\begin{prop}\label{tightnessinchi}
	Let $(u^n)_{n=1}^\infty$
	be the solution to the equations \eqref{approximationsolution} with $u_0\in L^2(\Omega,\sF_0,\bP;H)$. Then the laws of $(u^n)_{n=1}^\infty$ are tight in the space $\sX$, which was defined in \eqref{chiu}.
\end{prop}
\begin{proof}
	For every $\ep >0$, by Propositions \ref{tightnessinV} and \ref{tightnessinD} and the definition of tightness, there exist compact sets $K_1\subset \mathcal{D}([0,T];V')$
	and $K_2\subset L^2(0,T;H)$ such that	
	\begin{equation}
	\sup_{n \geq 1} \bP \left[u^n\not\in K_1\right]\le \frac{\ep}{2}
	\quad \text { and } \quad
	\sup_{n\geq 1} \bP \left[u^n\not\in K_2\right]\le \frac{\ep}{2}.
	\end{equation}
	Since $K_1\cap K_2$ is sequentially compact in $ \sX$ and
	\begin{equation}
	\sup_{n\geq 1}
	\bP\left[u^n\in (K_1\cap K_2)^C\right]\le
	\sup_{n\geq 1} \bP \left[u^n\not\in K_1\right]+\sup_{n\geq 1} \bP  \left[u^n\not\in K_2\right]\le \ep,
	\end{equation}
	we conclude that the laws of $(u^n)_{n=1}^\infty$ are tight on $ \sX$. 	
\end{proof}
\subsection{Passage to the limit}\label{passagetothelimit}
\subsubsection{Convergence results}
 Let $\mathcal{N}^{\# *}_{[0,\infty)\times E}$ denote the space of counting measures on $[0,\infty) \times E$ that are finite on bounded sets and let $\Upsilon:= H \times \sX \times C([0,T];U) \times \mathcal{N}^{\# *}_{[0,\infty)\times E}$. {It is an established fact that the spaces $\mathcal{N}^{\# *}_{[0,\infty)\times E}$ and $C([0,T];U)$ (cf. \cite{DVJ}) are separable and metrizable by a complete metric. Thus we see that the sequence of probability measures, $\mu_W^n(\cdot):=\bP(W\in\cdot)$, being constantly equal to one element, is tight on $C([0,T];U)$. Similarly we find out that the law of $\pi$ is tight on $\mathcal{N}^{\# *}_{[0,\infty)\times E}$. It is also clear that $u^n_0\rightarrow u_0$ $\bP$-a.s. in $H$, that is the laws of $(u^n_0)_{n = 1}^{\infty}$ are tight on $H$ and that their weak limit is the measure $\mu_0=\bP(u_0 \in \cdot)$. Finally, recalling Proposition \ref{tightnessinchi} and using Tychonoff's theorem it can be shown that the sequence of the joint probability measures $\mu^n$ associated with the approximating sequence ${(u^n_0,u^n,W,\pi)_{n=1}^{\infty}}$ is tight on the space $\Upsilon$.}
Hence, {by} virtue of the Prohorov theorem, $({\mu^n})_{n=1}^\infty$ is weakly compact over $\Upsilon$. This implies that there exists a probability measure $\mu^\infty$ {on $\Upsilon$} such that for some subsequence $\mu^n$ converges weakly to $\mu^\infty$. Since {we} want to show the existence of a solution to the truncated equations \eqref{approx_eqn}, the weak
convergence {of} $u^n$ {to} $ u $ in law and the weak convergence along a further subsequence {invoking} Lemma \ref{uniformboundL2} will not be enough. Instead,
we will use the Skorohod representation theorem {as stated in Theorem \ref{skorohodtheorem}} to upgrade the weak convergence to a.s. convergence in $\Upsilon$ for random variables defined on a new probability space. 
\begin{thm}\label{skorohod1}
There exists a probability space $(\tilde \Omega,\tilde{\mathcal{F}},\tilde \bP)$
, with the associated expectation denoted by $\tilde \bE$, and {$\Upsilon-$}valued random variables $(\bar u_0^{n},
\bar u^{n},\bar W^{n},\bar\pi^{n})$
and $(\bar u_0,
\bar u,\bar W, \bar\pi)$, such that
\begin{myenum}
	\item The probability laws of $(\bar u_0^{n}, \bar u^{n},\bar W^{n}, \bar\pi^{n})$ and  $(\bar u_0, \bar u,\bar W, \bar\pi)$ are $\mu^{n}$ and $\mu^\infty$, respectively,
	\item $(\bar W^{n},\bar\pi^{n})=(\bar W,\bar\pi) $ everywhere on $\tilde \Omega$,
	\item $(\bar{u}_0^{n},\bar{u}^{n},\bar W^{n},\bar\pi^{n})$ converges almost surely in {$\Upsilon $} to $(\bar u_0,\bar u,\bar W^{},\bar\pi^{})$, i.e. since $\bar W^{n}=W$ and $\bar {\pi}^{n}=\bar{\pi}$:
	\begin{align}
	&\bar u_0^{n}\rightarrow \bar u_0 \text{ in } H, \,\tilde \bP\tn{-a.s.},\label{uonk}	\\
	&\bar u^{n}\rightarrow
	\bar u \text { in } L^2(0,T;H)\cap \mathcal{D}([0,T];V'),
	\tilde \bP\tn{-a.s.},\label{un_convH}
	\end{align}
	\item  $\bar{\pi} $, which is the Poisson random measure over $(\tilde \Omega,\tilde{\mathcal{F}}, \tilde \bP)$ with intensity measure {$\ud t \otimes \ud \nu$}, is time invariant.
	\item $\bar u^n(0)=\bar u^n_0$ $\tilde \bP$-a.s. and
	$(\bar u^{n}_0,\bar u^{n},\bar W^{},\hat{\bar\pi})$ satisfies
	the following equation {in $H_n$}, $\tilde \bP$- a.s. for every $t\in[0,T]$
	\begin{equation}\label{barunk}
	\begin{aligned}
	&\bar u^{n}(t)+\int_0^t\left[\gamma A(\bar u^{n}(s))+P_{n} B_{\tilde R}(\bar u^{n}(s),\bar u^{n}(s))+ P_{n} F^R(\bar u^{n}(s))\right] \ud s
	\\&\qquad \qquad =\bar u^{n}_0+\int_0^tP_{n}G(\bar u^{n}(s
	))\ud \bar W(s)+\int_{(0,t]}\int_{E_0}P_{n}K(\bar u^{n}(s-),\xi)\ud \hat{\bar\pi}(s,\xi),
	\end{aligned}
	\end{equation}
	with respect to some filtration $\left(\bar{\mathcal{F}}^{n}_t \right)_{t \geq 0}$.
\end{myenum} \end{thm}
	 Let $\sF_t'$ be the $\sigma$-field generated by the random variables $\bar{W}(s), \bar{u}(s)$, $\bar{\pi}((0,s] \times \Gamma)$ for all $s \leq t$ and all Borel subsets $\Gamma$ of $E_0$. Then we define
\begin{align}\label{Ft}
\mathcal{N} &:=\{\mathcal{A}\in \tilde{\mathcal{F}} \ | \ \tilde \bP(\mathcal{A})=0\},\qquad
\bar{\mathcal{F}}^0_t:=\sigma(\mathcal{F}_t' \cup \mathcal{N}),\qquad
\bar{\mathcal{F}}_t :=\bigcap_{s\ge t}\bar{\mathcal{F}}^0_s.
\end{align}
This gives a complete, right-continuous filtration $(\bar{\mathcal{F}}_t)_{t \geq 0}$, {dependent on $R$ and $\tilde R$}, on the new probability space $(\tilde \Omega,\tilde{\mathcal{F}},\tilde \bP)$, to which the noise processes and solutions are adapted. For each $n$ we define a filtration $\left(\bar{\mathcal{F}}^{n}_t \right)_{t \geq 0}$ {in $t$} on  $(\tilde \Omega,\tilde{\mathcal{F}},\tilde \bP)$ the same way as above but using $\bar{u}^{n}$ in place of $\bar{u}$. Each $\bar{u}^{n}$ is adapted to $\left(\bar{\mathcal{F}}^{n}_t \right)_{t \geq 0}$.
\begin{rem}{We note here the point appearing in the proof of the Skorohod representation theorem, {as it appears in {Theorem} \ref{skorohodtheorem},} that the space $(\tilde \Omega,\tilde{\mathcal{F}},\tilde \bP)$, unlike the variables $(\bar u_0^{n},
\bar u^{n},\bar W^{n},\bar\pi^{n})$, does not depend on $R$ and $\tilde R$. Even though the analysis that follows does not require the new probability space to be independent of $R$ and $\tilde R$, as tightness and weak convergence are properties of the law, the preceding observation keeps the notation concise.}\end{rem} 
	{Since $u^n_0$ converges to $u_0$ as $n\rightarrow \infty$, ($P_n$ converges to $I$) $\bP$-a.s. in $H$, $\bar u_0$ has the same law as that of $u_0$, which is independent of $R$ and $\tilde R$. However $\bar u_0$ itself might depend on $R$ and $\tilde R$, but this is not relevant in the analysis below.} 

The proof of $v)$ can be obtained by adapting the argument of \cite{Ben} from the Wiener case to the \levy noise case. For more details concerning the utilization of \cite{Ben}, the reader is referred to e.g. Proposition B.1 in \cite{EulerPaper}.\\
 Now we define the new stochastic basis $\bar{\mathscr{S}}:=(\tilde \Omega,\tilde{\mathcal{F}},(\bar{\mathcal{F}}_t)_{t \geq 0},\tilde \bP, \bar{W},\bar\pi)$.  We will show in Theorem  \ref{prop:passage_to_limit} that  $(\bar{\mathscr{S}},\bar u)$ is a global martingale solution of the system \eqref{approx_eqn}.
Observe that $\bar u^{n}$ also satisfies the estimates in Lemma \ref{uniformboundL2} and thus the following lemma holds true.
	\begin{lem}
	For $\bar u^n$ and $\bar u$ as in \eqref{un_convH}, Lemma \ref{uniformboundL2} implies that $ \bar u \in  L^\infty([0,T];L^2(\tilde{ \Omega};H))\cap
		L^2([0,T] \times\tilde \Omega,\ud t\otimes \ud \tilde\bP;V)$ and that there exists a subsequence, still denoted by $n \rightarrow \infty$ such that
		\begin{align}\label{weakstar}
		\bar u^{n} &\rightharpoonup  \bar u \text{ weak* in } L^\infty([0,T];L^2(\tilde \Omega;H)),\\	\label{weakV}
		\bar u^{n}&\rightharpoonup \bar u \text{ weakly in }L^2([0,T]\times\tilde \Omega,\ud t\otimes \ud \tilde\bP;V),
		\shortintertext{and thanks to \eqref{FR4}}
		Y^{n}:= P_{n}F^R(\cdot,\bar u^{n})&\rightharpoonup  Y^{R,\tilde{R}} \text{ weakly in }
		L^{2}([0,T]\times \tilde \Omega,\ud t\otimes \ud \tilde\bP;V').\label{Y}
		\end{align}		
	\end{lem}
Now we state the main result of this section.
\begin{thm}\label{prop:passage_to_limit} 
		Suppose that $\bE[|u_0|^2] <\infty$. Then $\bar u$ solves \eqref{approx_eqn}  with respect to the stochastic basis $(\tilde \Omega,\tilde{\mathcal{F}}, (\bar{\mathcal{F}}_t)_{t\ge 0},\tilde \bP,\bar{W},\bar\pi)$ with initial condition $\bar u_0$, where $\bar u_0$ has the same law as $u_0$.
\end{thm}
\begin{proof}
We are going to prove this theorem in two steps.\\
\noindent {{\underline{Step 1}}}:
We will first show that $\bar u $ solves the following equation
\begin{align}\label{step1}
\bar u(t)+\int_0^t[\gamma A(\bar u(s))+B_{\tilde{R}}(\bar u(s),\bar u(s))+ Y^{R,\tilde{R}}]&\ud s\\
&\hspace{-1.5in}=\bar u_0+\nonumber\int_0^tG(\bar u(s))\ud\bar{W}(s)+\int_{(0,t]}\int_{E_0}K(\bar u(s-),\xi)\ud\hat{\bar{\pi}}(s,\xi),
\end{align}
where $Y^{R,\tilde{R}}$ was defined in  \eqref{Y}. 
We fix $\phi\in D(A)$ and take the inner product of equation \eqref{barunk} in $H$ with $\phi$ on both sides and let $n\rightarrow \infty$. \\
{First, observe that from \eqref{un_convH} we have {that} $\bar u^n \rightarrow \bar u$, $\tilde \bP$-a.s. in $L^2(0,T;H)$ and thus in $L^1(0,T;H)$. 
Thanks to Lemma \ref{uniformboundL2} we also know that the sequence $\{\bar u^n\}_{n=1}^\infty$ is bounded in $L^2(\tilde\Omega; L^2(0,T;H))$ {independently} of $n$. 

Thus an application of the Vitali convergence {theorem} gives us that \begin{align}\label{conv_un}\bar u^n \rightarrow \bar u \quad \text{in } L^1(\tilde\Omega;L^1(0,T;H)).\end{align}}
We also know from \eqref{uonk} that $	(\bar u^{n}_0,\phi) \rightarrow (\bar u_0,\phi)$ in $L^1(\tilde{ \Omega}).$\\
We will next show that, as $n \rightarrow \infty$,
\begin{equation}\label{conv_A}
\int_0^t\iprod{A\bar{u}^{n}(s)}{\phi} \ud s \rightarrow \int_0^t\iprod{A\bar u(s)}{\phi} \ud s \quad \text { in }
L^1(\tilde \Omega\times [0,T]).
\end{equation}
First, we note that for almost every $\omega \in \tilde{ \Omega}$ {and for every $t\in [0,T] $}, 
\begin{equation}
\begin{aligned}\label{A1}
 \left|\int_0^t\iprod{A\bar{u}^{n}(s)-A\bar u(s)}{\phi}
\ud s\right|
&=\left|\int_0^t\iprod{\bar{u}^{n}(s)-\bar u(s)}{A\phi}
\ud s\right| \\
&\lesssim \Hnorm{A\phi}\left(\int_0^T
\Hnorm{\bar{u}^{n}(s)-\bar{u}(s)}^2 \ud s\right)^{\frac12}\\
&\rightarrow 0 \text{ as } n\rightarrow \infty \text{ from \eqref{un_convH}}. 
\end{aligned}
\end{equation}
Furthermore, we infer from Lemma \ref{uniformboundL2} the existence of a constant $C=C(T,\phi,K_1)>0$ independent of $n$ such that, 
\begin{align}
\sup_{n\geq 1}
\tilde \bE \int_0^T\left|\int_0^t
\iprod{A\bar{u}^{n}(s)}{\phi}\ud s\right|^2 \ud t
&\leq T\sup_{n\geq 1}
\tilde \bE\int_0^T\left|\iprod{\bar{u}^{n}(s)}{A\phi}\right|^2\ud s\nonumber\\
&\le C(T)\sup_{n\geq 1}
\tilde \bE\int_0^T|\bar{u}^{n}(s)|^2|A\phi|^2\ud s
\nonumber\\&\le C(T) \|\phi\|_{D(A)}^2\,\, \sup_{n\geq 1}\tilde \bE\sup_{s\in [0,T]}|\bar{u}^{n}(s)|^2\nonumber\\&\le C.\label{A2}
\end{align}
From \eqref{A1}-\eqref{A2} and by means of the Vitali convergence theorem, we obtain \eqref{conv_A}.

We now consider the term $B_{\tilde{R}}$ as introduced in Section \ref{sec:functional_framework}. We will show that as $ n \rightarrow \infty$,
\begin{equation}\label{conv_B}
\int_0^t\iprod{P_{n}B_{\tilde{R}}(\bar{u}^{n},\bar{u}^{n})}{\phi}\ud s
\rightarrow \int_0^t\iprod{B_{\tilde{R}}(\bar u,\bar u)}{\phi} \ud s \quad \text{ in $L^1(\tilde \Omega\times [0,T])$.}
\end{equation}
For this purpose, we write
\begin{align}
\left|\int_0^t\iprod{P_{n}B_{\tilde{R}}(\bar{u}^{n},\bar{u}^{n})-B_{\tilde{R}}(\bar u,\bar u)}{\phi}\ud s\right|
& \le \int_0^t|(\theta_{\tilde{R}}(|\bar u^{n}|)-\theta_{\tilde{R}}(|\bar u^{}|))\iprod{B_{}(\bar u,\bar u)}{P_{n}\phi}|\ud s \nonumber\\
& +\int_0^t\theta_{\tilde{R}}(|\bar u^{n}|)|\iprod{(I-P_{n})B(\bar u,\bar u)}{\phi}|\ud s \nonumber\\
&\hspace{-0.7in} + \int_0^t \theta_{\tilde{R}}(|u^{n}|)|\iprod{B(\bar{u}^{n}-\bar u,\bar{u}^{n})+B(\bar u,\bar u^{n}-\bar u)}{P_{n}\phi}|\ud s \nonumber \\
&:=I_{B1}^n(t)+I_{B2}^n(t)+I_{B3}^n(t).
\end{align}
Using the continuity of $\theta_{\tilde{R}}$ and \eqref{un_convH} together with the Lebesgue dominated convergence theorem we obtain that $\tilde \bE\int_0^T I_{B1}^{n} (t)\ud t \rightarrow 0 $ as $n\rightarrow \infty$. For a subsequence this implies,
\begin{align*}
I_{B1}^n\rightarrow 0 \text{ a.e. }(\omega,t) \in \tilde{ \Omega} \times [0,T] \text{ as } n\rightarrow \infty.
\end{align*}
Next observe that $\|(I-P_{n})\phi\| \leq \frac{1}{\sqrt{\lambda_{n}}}\|\phi\|_{D(A)}$ where {{}the} $\lambda_n$ are the eigenvalues of the operator $A$. Thus using \eqref{eqn:B_bound3} and that $V \subset H$ is continuous we see that
\begin{equation}\label{IB2}
\begin{aligned}
|I_{B2}^n(t)|= \int_0^t\theta_{\tilde{R}}(|\bar u^{n}|)|\iprod{(I-P_{n})B(\bar u,\bar u)}{\phi}|\ud s \nonumber& \le \int_0^t |\iprod{B(\bar u,\bar u)}{(I-P_{n})\phi}|\ud s \\
& \le \frac{1}{\sqrt{{\lambda_{n}}}}\|\phi\|_{D(A)} \int_0^t \|\bar u(s)\|^2\ud s.
\end{aligned}
\end{equation}
Since we know that $\lambda_n\rightarrow \infty$ as $n\rightarrow \infty$ we have $\tilde \bE \int_0^T I_{B2}^n \ud t \rightarrow 0$ as $n\rightarrow \infty$. For a subsequence this implies,
$$
I_{B2}^n\rightarrow 0 \text{ a.e. }(\omega,t) \in \tilde{ \Omega} \times [0,T] \text{ as } n\rightarrow \infty.
$$ 
 {Next using {the fact that $\theta_{\tilde R}\leq 1$} and the assumption \eqref{eqn:B_bound2} we obtain 
	\begin{align*}
	|I_{B3}^n(t)| & \leq \int_0^t \theta_{\tilde{R}}(|\bar u^{n}|)|\bar u^{n}-\bar u|\left(\|\bar u^{n}\|+\|\bar u\|\right)\|\phi\|_{D(A)} \ud s\\
	&\leq \|\phi\|_{D(A)}\left(\|\bar u^{n}\|_{L^2(0,T;V)}+ \|\bar u\|_{L^2(0,T;V)}\right)\,|\bar u^{n}-\bar u|_{L^2(0,T;H)}.
	\end{align*}	
We know that $\tilde \bE \int_0^T\|\bar u^{n}\|^2 \ud s $ is bounded by a constant independent of $n$, 
thus thanks to \eqref{un_convH} we obtain that $I_{B3}^n\rightarrow 0$ for a subsequence $n\rightarrow \infty$ for almost every $(\omega,t) \in \tilde{ \Omega} \times [0,T]$.}\\ Next, recall that $\theta_{\tilde R}(x)$ vanishes for $x \geq 2\tilde R$. Thus using the assumption \eqref{eqn:B_bound2} and Lemma \ref{uniformboundL2} we obtain,
\begin{align}
\sup_{n\geq 1}
\tilde \bE\int_0^T\bigg|\int_0^t\iprod{P_{n}B_{\tilde{R}}(\bar{u}^{n},\bar{u}^{n})}{\phi} \ud s\bigg|^2 \ud t
&= \sup_{n\geq 1}
\tilde \bE \int_0^T\bigg|\int_0^t \theta_{\tilde{R}}(|u^{n}|)\iprod{B(\bar{u}^{n},\bar{u}^{n})}{P_{n}\phi}\ud s\bigg|^2 \ud t \nonumber\\
&\le \tilde{R}^2 T \|\phi\|_{D(A)}^2\,\, \sup_{n\geq 1}
\tilde \bE\int_0^T\Vnorm{\bar u^{n}(s)}^2 \ud s \nonumber\\
& \le C(\phi,T, \tilde{R},K_1). \end{align}
Thus using {once more} the Vitali convergence theorem we {obtain} \eqref{conv_B}.\\
The noise terms are treated differently.
\begin{lem}\label{conv_G}
	The processes $
	(\int_0^tP_{n}G(\bar{u}^{n}(s))\ud\bar{W}(s))_{t \in [0,T]}$ converge to $(\int_0^tG(\bar{u}(s))\ud\bar{W}(s))_{t \in [0,T]}$  in $L^1(\tilde{ \Omega};L^1([0,T];H))$ as $n\rightarrow \infty$.
\end{lem}
\begin{proof}
We will first show that along a subsequence, {$\tilde \bP$-a.s.}
\begin{align}\label{g1}
P_{n}G(\bar{u}^{n})\rightarrow G(\bar u) \qquad \text{in $L^2([0,T];L_2(
U_0,H))$}.\end{align}
 Indeed, utilizing the Poincar\'e inequality, the hypothesis \eqref{eqn:G,K_H_Lipschitz} {and the fact that $P_n$ is a projector in $H$ }we obtain {$\tilde \bP$-a.s.} 
\begin{align}
\nonumber
\int_0^T\Vnorm{P_{n}G(\bar{u}^{n})- G(\bar u)}^2_{L_2(
	U_0,H)}\ud s
&\le \int_0^T(\Vnorm{P_{n}G(\bar{u}^{n})-P_{n} G(\bar u)}^2_{L_2(
	U_0,H)}\\
&\hspace{2in}+\Vnorm{(I-P_{n})G(\bar u)}^2_{L_2(
	U_0,H)})\ud s\nonumber\end{align}\begin{align}\label{G1}
&\le \rho\int_0^T\Hnorm{\bar{u}^{n}-\bar u}^2+\Vnorm{(I-P_{n})G(\bar u)}^2_{L_2(
	U_0,H)}\ud s
\rightarrow 0.
\end{align}
{The convergence in }\eqref{G1} follows from \eqref{un_convH} and the Lebesgue Dominated convergence theorem since we have that $\Vnorm{(I-P_{n})G(\bar u)}^2_{L_2(
	U_0,H)}$ converges to $0$ a.e. on $ [0,T]$ and is dominated by $\Vnorm{G(\bar u)}^2_{L_2(
	U_0,H)} \in L^1[0,T]$.\\
Due to Lemma 2.1 of \cite{DGHT} (see also \cite{Ben}), the convergence \eqref{g1} implies that
 \begin{align}\int_0^tP_{n}(G(\bar u^{n}))\ud\bar W(s) \rightarrow \int_0^tG(\bar u(s))\ud\bar W(s) \text{ in probability in $L^2(0,T;H)$. } \end{align}
  Additionally, for some $C=C(T,\rho,K_1)>0$ independent of $n$, thanks to the  It\^{o} isometry we obtain
\begin{align}
\tilde \bE\int_0^T \Hnorm{\int_0^t P_{n}G(\bar{u}^{n}(s))\ud\bar W(s)}^2\ud t
&=\int_0^T\tilde \bE\int_0^t\Vnorm{G(\bar{u}^{n})}^2_{L_2(
	U_0,H)}\ud s \ud t \nonumber\\
 & \le \rho T
\sup_{n\geq 1}
\tilde \bE\left(\int_0^T(\Hnorm{\bar{u}^{n}}^2+1)\ud s\right)\nonumber\\
&\le \rho T\sup_{n\geq 1}\tilde \bE\sup_{t\in [0,T]}
(\Hnorm{\bar{u}^{n}}^2+1)\le C.\label{G2}
\end{align}
Making use of the Vitali convergence theorem again, {we conclude the proof of Lemma \ref{conv_G}}.
\end{proof}
\begin{lem}\label{conv_K}
	The processes $
(\int_{(0,t]} \int_{E_0}{P_{n}K(\bar{u}^{n}(s-),\xi)}{}\ud\hat{\bar{\pi}}(s,\xi)
)_{t\in[0,T]}$ converge to\\
 $ (\int_{(0,t]} \int_{E_0} {K(\bar u(s-),\xi)}{}\ud\hat{\bar{\pi}}(s,\xi))_{t \in [0,T]}$
in $L^1(\tilde \Omega; L^1([0,T];H))$ as $n\rightarrow \infty$.\end{lem}
\begin{proof}

First we have {$\tilde \bP$-a.s.}
\begin{align}
 \int_0^T\int_0^t\int_{E_0}&
\left|{P_{n}K(\bar{u}^{n}(s-),\xi)-K(\bar u(s-),\xi)}{}\right|^2\ud\nu(\xi)\ud s \ud t
\nonumber\\\nonumber
\le &~ \int_0^T\int_0^t\int_{E_0}\left|{P_{n}K(\bar{u}^{n}(s-),\xi)-P_{n}K(\bar u(s-),\xi)}{}\right|^2\ud\nu(\xi)\ud s \ud t \nonumber\\
&\hspace{2in} +\int_0^T\int_0^t\int_{E_0}\left|{(I-P_{n})K(\bar{u}^{}(s-),\xi)}{}\right|^2\ud\nu(\xi)\ud s \ud t \nonumber \end{align}\begin{align}
\le &~ \int_0^T\int_0^t\int_{E_0}\left|{K(\bar{u}^{n}(s-),\xi)-K(\bar u(s-),\xi)}{}\right|^2\ud\nu(\xi)\ud s \ud t \nonumber\\
&\hspace{2in}+\int_0^T\int_0^t\int_{E_0}\left|{(I-P_{n})K(\bar{u}^{}(s-),\xi)}{}\right|^2\ud\nu(\xi) \ud s \ud t \nonumber\\
\le &~ T\int_0^T
\Hnorm{ \bar{u}^{n}(s)-\bar u(s)}^2 \ud s+ T\int_0^T\int_{E_0}\left|{(I-P_{n})K(\bar{u}^{}(s-),\xi)}{}\right|^2\ud\nu(\xi)\ud s
\nonumber\\
\rightarrow &\, 0\, \text{ as }n\rightarrow \infty.
\label{eqn:estimate1}
\end{align}

The last line follows thanks to the fact that
$\bar{u}^{n}\rightarrow \bar u$
in $L^2(0,T;H)$ {$\tilde \bP$-a.s.} and using the Lebesgue dominated convergence theorem since we also see that as $n\rightarrow \infty$ the sequence $|(I-P_{n})K(\bar u^{}(s-),\cdot)|^2_{L^2(E_0,\nu;H)}$ converges to $0$ a.e. on $[0,T]$ and is dominated by $|K(\bar u(s-),\cdot)|^2_{L^2(E_0,\nu;H)} \in L^1([0,T]).$
\\ This implies that
${P_{n}K(\bar u^{n}(s-),\xi)}{}$ converges to
$ {K(\bar u(s-),\xi}{})$
in $L^2([0,T]\times E_0,\ud t\otimes \ud \nu;H)$ $\tilde \bP-$a.s.
It is then easy to see that the above almost sure convergence implies the {following} convergence
\begin{equation}
\int_{(0,t]} \int_{E_0}{P_{n}K(\bar{u}^{n}(s-),\xi)}\ud\hat{\bar{\pi}}(s,\xi)
\rightarrow \int_{(0,t]} \int_{E_0} {K(\bar u(s-),\xi)}{}\ud\hat{\bar{\pi}}(s,\xi)
\end{equation}
in probability as $L^2(0,T;H)$-valued random variables as $n\rightarrow \infty$. 

Also observe that, using the isometry property {(e.g. \cite{IW} or Theorem 2.16 in \cite{CNTT20})}, we obtain

\begin{align}\label{eqn:estimate2}
\tilde \bE \int_0^T\bigg|{\int_{(0,t]}\int_{E_0}
	P_{n}K(\bar u^{n}(s-),\xi)}\ud\hat{\bar{\pi}}(s,\xi)\bigg|^2\ud t &=
\int_0^T\left( \tilde \bE\int_0^t\int_{E_0}|{P_{n}K(\bar u^{n}(s-),\xi)}{}|^2
\ud\nu(\xi)\ud s\right)\ud t\nonumber\\
&\le \rho T\, \sup_{n \geq 1}
\tilde \bE \int_0^T(1+|\bar u^{n}(s)|^2) \ud s \nonumber\\
& \le C(T,\rho,K_1).
\end{align}
From \eqref{eqn:estimate1}-\eqref{eqn:estimate2}, and an application of the Vitali convergence Theorem {we finish the proof of Lemma \ref{conv_K}}.
\end{proof}
{Gathering} all the convergence results \eqref{conv_un}, \eqref{conv_A}, \eqref{conv_B}, Lemma \ref{conv_G} and Lemma \ref{conv_K}, along with {the weak convergence obtained in} \eqref{Y}, we see that for any $\phi\in D(A)$ and {any} measurable set $\mathcal{S}\subset \tilde \Omega\times[0,T]$, we have
\begin{align}
\tilde \bE\int_0^T\chi_{\mathcal{S}}
\bigg[\iprod{\bar u(s)}{\phi}\ud s
+\int_0^t\gamma\iprod{A\bar u(s)}{\phi}\ud s
+\int_0^t\iprod{B_{\tilde R}(\bar u(s),\bar u(s))}{\phi} \ud s+\int_0^t
\iprod{ Y^{R,\tilde{R}}}{\phi}\ud s\bigg]\ud t
\nonumber \\
=\tilde \bE\int_0^T\chi_{\mathcal{S}}
\bigg[\iprod{\bar u_0}{\phi}\ud s+\int_0^t\iprod{G(\bar u(s))}{\phi}\ud\bar W(s)+\int_{(0,t]}\int_{E_0}\iprod{K(\bar u(s-),\xi)}{\phi}\ud\hat{\bar{\pi}}(s,\xi)\bigg]\ud t.\label{convergenceresult}
\end{align}
Since $D(A)$ is dense in {$V$}, and $\mathcal{S}$ is arbitrary it follows that $ \ud  \tilde\bP \otimes \ud t$-a.e. we have
\begin{align} \label{eqn:bar_u_soln_w/Y} 
&\bar u(t)
+\int_0^t \gamma A\bar u(s)\ud s
+\int_0^tB_{\tilde R}(\bar u(s),\bar u(s))\ud s+\int_0^t
 Y^{R,\tilde{R}}\ud s \nonumber
 \\&\hspace{1in}=
\bar u_0+\int_0^tG(\bar u(s))\ud \bar W(s)+\int_{(0,t]}\int_{E_0}K(\bar u(s-),\xi)\ud\hat{\bar{\pi}}(s,\xi) \text{ in }{V'}.
\end{align}
See Lemma 5.13 in \cite{EulerPaper} for details. In particular we have $\bar u_0=\bar u(0)$ $\tilde \bP$-a.s.\\
\noindent {\underline{Step 2}}:
In order to show that $\bar u$ is a solution to the equations \eqref{approx_eqn}, it remains to verify that
\begin{align}\label{YFn}
&Y^{R,\tilde{R}}=F^R(\bar{u}), \quad  \ud \tilde\bP \otimes \ud t \tn{-a.e.}
\end{align}
This follows from a classical monotonicity argument due to G. Minty \cite{M62}, \cite{M63} (see also \cite{Br73}, \cite{B77}, \cite{BLZ},\cite{LL65}, \cite{Prevot} \cite{T_NSE}). First using the product rule we have
	\begin{align}\label{mult_exp}
	e^{-\rho t}|\bar u(t)|^2= |\bar u_0|^2 +\int_0^t  e^{-\rho s}\left(-\rho|\bar u(s)|^2 \ud s + \ud |\bar u(s)|^2\right).
	\end{align}	
	Now we apply It\^o's formula to the process in \eqref{eqn:bar_u_soln_w/Y} using the function $\psi(u)=|\bar u|^2$ {as in \eqref{ito}}:
\begin{align}
\nonumber& \Hnorm{\bar u(t)}^2-\Hnorm{\bar u_0}^2
=\int_0^t \Bigg[-2 \gamma \iprod{A\bar u(s)}{\bar u(s)}-2\iprod{Y^{R,\tilde{R}}(s)}{\bar u(s)} + \Vnorm{G(\bar{u}(s))}^2_{L_2(U_0,H)}\Bigg]\ud s\\&+\int_{(0,t]}\int_{E_0}
|K(\bar{u}(s-),\xi)|^2\ud{\bar{\pi}}(\xi,s)\nonumber+\int_0^t\iprod{\bar u(s)}{G(\bar u(s))\ud W(s)} + \int_{(0,t]}\int_{E_0}\iprod{K(\bar u(s-),\xi)}{\bar u(s)}\ud\hat {\bar \pi}(s,\xi).
\end{align}
Combining the above two equations we obtain
\begin{align}
\nonumber \Hnorm{\bar u(t)}^2-\Hnorm{\bar u_0}^2
&=\int_0^t e^{-\rho s}
\bigg[-\rho|\bar u(s)|^2 -2\gamma\iprod{A\bar u(s)}{\bar u(s)}-2\iprod{Y^{R,\tilde{R}}(s)}{\bar u(s)}
 + \Vnorm{G(\bar{u}(s))}^2_{L_2(U_0,H)}\bigg]\ud s\\&+\int_{(0,t]}\int_{E_0}e^{-\rho s}
|K(\bar{u}(s-),\xi)|^2\ud{\bar{\pi}}(\xi,s) + \int_{(0,t]}\int_{E_0}e^{-\rho s}\iprod{K(\bar u(s-),\xi)}{\bar u(s)}\ud\hat {\bar \pi}(s,\xi)\nonumber\\
&+\int_0^t e^{-\rho s}\iprod{\bar u(s)}{G(\bar u(s))\ud W(s)}.\label{ito_exp}
\end{align}
We take expectation of both sides of \eqref{ito_exp}. Observing that the last two terms are martingales {and their expectations vanish and we obtain}
\begin{align}
\nonumber\tilde \bE\left(e^{-\rho t}\Hnorm{\bar u(t)}^2-\Hnorm{\bar u_0}^2\right)
&=\tilde \bE\Big(\int_0^t e^{-\rho s}
\bigg[-2\gamma\iprod{A\bar u(s)}{\bar u(s)}-2\iprod{Y^{R,\tilde{R}}(s)}{\bar u(s)}
\\
& + \Vnorm{G(\bar{u}(s))}^2_{L_2(U_0,H)}-\rho|\bar u(s)|^2\bigg]\ud s+ \int_{(0,t]} \int_{E_0} e^{-\rho s}
|K(\bar{u}(s-),\xi)|^2\ud{\bar{\pi}}(\xi,s)\Big)\label{ito_exp2}.
\end{align}
In a similar fashion, we carry out \eqref{mult_exp} for $\bar u^{n}$, then apply It\^o's formula to the process in \eqref{barunk} and take expectation of both sides. Using that {$P_{n}(\bar u^{n}(t))=\bar u^{n}(t)$}, we obtain:
\begin{align}\label{eqn2}
\tilde \bE\left(e^{-\rho t}\Hnorm{\bar{u}^{n}(t)}^2-\Hnorm{\bar{u}^{n}_0}^2\right)
&=\tilde \bE\big(\int_{0}^t e^{-\rho s}
\bigg[-2\gamma\iprod{A\bar{u}^{n}(s)}{\bar{u}^{n}(s)} - 2\iprod{F^R(\bar{u}^{n}(s))}{\bar{u}^{n}(s)}
\nonumber\\
&\hspace{-0.4in}+\Vnorm{G(\bar{u}^{n}(s))}^2_{L_2(U_0,H)}-\rho|\bar u^{n}(s)|^2 \bigg]\ud s+ \int_{(0,t]} \int_{E_0} e^{-\rho s}
|K(\bar{u}^n(s-),\xi)|^2\ud{\bar{\pi}}(\xi,s)\big).
\end{align}
Next, we will use the monotonicity of $A$ that comes from its linearity and that of $F^R$  as specified in \eqref{FR1}. That is, for an arbitrary $ \Phi\in  L^2([0,T]\times \tilde{ \Omega},\ud t\otimes  \ud\tilde \bP;V)$, we write $\iprod{F^R(\bar u^{n}(s))}{\bar u ^{n}} \geq \iprod{F^R(\bar{u}^{n}(s))}{\Phi(s)} + \iprod{F^R(\Phi(s))}{\bar{u}^{n}(s)-\Phi(s)} $. Treating the terms involving $A$ similarly we obtain,
\begin{align}
\tilde \bE\left(e^{-\rho t}\Hnorm{\bar{u}^{n}(t)}^2-\Hnorm{\bar{u}^{n}_0}^2\right)&
\leq \tilde \bE\int_0^t e^{-\rho s}\bigg(
-2\gamma\iprod{A\bar{u}^{n}(s)}{\Phi(s)} - 2\gamma\iprod{A\Phi(s)}{\bar{u}^{n}(s)-\Phi(s)}\ud s
\nonumber\\&\hspace{-0.2in}-2\iprod{F^R(\bar{u}^{n}(s))}{\Phi(s)}\ud s  -2\iprod{F^R(\Phi(s))}{\bar{u}^{n}(s)-\Phi(s)}\ud s
\nonumber\\
&\hspace{-0.2in}+\Vnorm{G(\bar{u}^{n}(s))}^2_{L_2(U_0,H)}\ud s+\int_{E_0}
|K(\bar{u}^{n}(s-),\xi) |^2\ud \nu(\xi)\ud s-\rho|\bar u^{n}|^2 \ud s\bigg).
\end{align}{Furthermore, using the equation $|a|^2=|a-b|^2 + 2\iprod{a}{b} -|b|^2$ we {see that the right-hand side of the equation above is equal to}}
\begin{equation}\begin{split}\label{nk}
&\tilde \bE\int_0^t e^{-\rho s}
\bigg[-2\gamma\iprod{A\bar{u}^{n}(s)}{\Phi(s)}\ud s+2\gamma\iprod{A\Phi(s)}{\Phi(s)}\ud s-2\gamma\iprod{A\Phi(s)}{\bar{u}^{n}(s)}\ud s\\
&+2\iprod{F^R(\Phi(s))}{\bar{u}^{n}(s)}\ud s
-2\iprod{F^R(\Phi(s))}{\Phi(s)}\ud s-2\iprod{F^R(\bar{u}^{n}(s))}{\Phi(s)}\ud s\\
&-\rho\left(|\bar u^{n}(s)-\Phi(s)|^2-|\Phi(s)|^2 +2\iprod{\bar u^{n}}{\Phi}\right)\ud s
-\int_{E_0}|K(\Phi(s),\xi)|^2\ud\nu(\xi)\ud s
-\Vnorm{G(\Phi(s))}^2_{L_2(U_0,H)}\ud s
\\&+2\int_{E_0}
\iprod{K(\Phi(s),\xi)}{K(\bar{u}^{n}(s),\xi)}\ud\nu(\xi)\ud s +2\iprod{ G(\Phi(s))}{G(\bar u^{n}(s))}_{L_2(U_0,H)}\ud s\\
&+\Vnorm{G(\bar{u}^{n}(s))-G(\Phi(s))}^2_{L_2(U_0,H)}\ud s +\int_{E_0}
|K(\bar{u}^{n}(s-),\xi)-  K(\Phi(s),\xi) |^2\ud\nu(\xi)\ud s
\bigg].\nonumber
\end{split}
\end{equation}
We next use \eqref{eqn:G,K_H_Lipschitz} on the last two terms of the above equation {and bound them from above by $\rho|\bar u^{n}(s)-\Phi(s)|^2$}. This gives us the following upper bounds
\begin{align}
&\hspace{-0.7in}\tilde \bE\left(e^{-\rho  t}\Hnorm{\bar u^{n}(t)}^2-\Hnorm{\bar u^{n}_0}^2\right)\nonumber\\
&\leq \tilde \bE\bigg[\int_0^t e^{-\rho s}
\bigg[-2\gamma\iprod{A\bar{u}^{n}(s)}{\Phi(s)}\ud s+2\gamma\iprod{A\Phi(s)}{\Phi(s)}\ud s-2\gamma\iprod{A\Phi(s)}{\bar{u}^{n}(s)}\ud s\nonumber\\
&+2\iprod{F^R(\Phi(s))}{\bar{u}^{n}(s)}\ud s
-2\iprod{F^R(\Phi(s))}{\Phi(s)}\ud s-2\iprod{F^R(\bar{u}^{n}(s))}{\Phi(s)}\ud s\nonumber\\
&-\rho\left(|\bar u^{n}(s)-\Phi(s)|^2-|\Phi(s)|^2 +2\iprod{\bar u^{n}(s)}{\Phi}\right)\ud s+{{}\rho |\bar u^{n}(s)-\Phi(s)|^2\ud s}\nonumber\\
&-\Vnorm{G(\Phi(s))}^2_{L_2(U_0,H)}\ud s +2\iprod{ G(\Phi(s))}{G(\bar u^{n}(s))}_{L_2(U_0,H)}\ud s\nonumber\\
&-\int_{E_0}|K(\Phi(s),\xi)|^2\ud\nu(\xi)\ud s+2\int_{E_0}
\iprod{K(\Phi(s),\xi)}{K(\bar{u}^{n}(s),\xi)}\ud\nu(\xi)\ud s\bigg].\label{imp_exp}
\end{align}

 {Now we multiply each side of the above equation by $\varphi\in L^\infty([0,T],\ud t;\mathbb{R}),$ $ \varphi \geq 0$ and integrate in time. Finally we let $n\rightarrow \infty$} using the convergence results \eqref{un_convH}-\eqref{weakV} and obtain
\begin{align*}
\tilde \bE\int_0^T\varphi(t)&\left(e^{-\rho  t}\Hnorm{\bar u(t)}^2-\Hnorm{\bar u_0}^2\right)\ud t\\
&\le
\tilde \bE\int_0^T\varphi(t)\bigg(\int_0^t e^{-\rho s}
\bigg[-2\gamma\iprod{A\bar u(s)}{\Phi(s)}\ud s
+2\gamma\iprod{A\Phi(s)}{\Phi(s)}\ud s-2\gamma\iprod{A\Phi(s)}{\bar u(s)}\ud s\\
&+2\iprod{F^R(\Phi(s))}{\bar u(s)}\ud s
-2\iprod{F^R(\Phi(s))}{\Phi(s)}\ud s-2\iprod{Y^{R,\tilde{R}}(s)}{\Phi(s)}\ud s\\
&-\rho(-|\Phi(s)|^2 +2\iprod{\bar u}{\Phi})\ud s
-\norm{G(\Phi(s))}^2_{L_2(U_0,H)}\ud s+ 2\iprod{G(\bar u(s))}{G(\Phi(s))}_{L_2(U_0,H)}\ud s\\
&-\int_{E_0}|K(\Phi(s),\xi)|^2 \ud\nu(\xi)\ud s+2\int_{E_0}
\iprod{K(\Phi(s),\xi)}{K(\bar u(s-),\xi)}\ud\nu(\xi)\ud s\bigg]\bigg)\ud t.
\end{align*}
Combining this with \eqref{ito_exp2}, we obtain
\begin{align*}
&\tilde \bE \int_0^T  \varphi(t)\bigg(\int_{0}^t e^{-\rho  s}
\bigg[-2\gamma\iprod{A\bar u(s)}{\bar u(s)} -2\iprod{Y^{R,\tilde{R}}(s)}{\bar u(s)}
+\Vnorm{G(\bar{u}(s))}^2_{L_2(U_0,H)} {-\rho |\bar u(s)|^2}\bigg]\ud s
\\&\hspace{2in}+\int_{(0,t]}\int_{E_0}
|K(\bar{u}(s-),\xi)|^2\ud{\bar{\pi}}(s,\xi)\bigg)\ud t\\
&\le
\tilde \bE\int_0^T\varphi(t)\bigg(\int_0^t e^{-\rho s}
\bigg[-2\gamma\iprod{A\bar u(s)}{\Phi(s)}\ud s
+2\gamma\iprod{A\Phi(s)}{\Phi(s)}\ud s-2\gamma\iprod{A\Phi(s)}{\bar u(s)}\ud s\\
&\qquad+2\iprod{F^R(\Phi(s))}{\bar u(s)}\ud s
-2\iprod{F^R(\Phi(s))}{\Phi(s)}\ud s-2\iprod{Y^{R,\tilde{R}}(s)}{\Phi(s)}\ud s\\
&{+\rho\left(|\Phi(s)|^2 -2\iprod{\bar u(s)}{\Phi(s)}\right)\ud s}-\norm{G(\Phi(s))}^2_{L_2(U_0,H)} \ud s+ 2\iprod{G(\bar u(s))}{G(\Phi(s))}_{L_2(U_0,H)}\ud s\\
&-\int_{E_0}|K(\Phi(s),\xi)|^2 \ud\nu(\xi)\ud s+2\int_{E_0}
\iprod{K(\Phi(s),\xi)}{K(\bar u(s-),\xi)}\ud\nu(\xi)\ud s\bigg]\bigg)\ud t.
\end{align*}
Using again $|a|^2=|a-b|^2+2\iprod{a}{b}-|b|^2$ we obtain
\begin{align*}
 \tilde \bE\int_0^T\varphi(t)\int_0^t &e^{-\rho s}\bigg(
2\gamma\iprod{A\bar u(s)-A\Phi(s)}{\bar u(s)-\Phi(s)}\ud s+
2\iprod{F^R(\Phi(s))-Y^{R,\tilde{R}}(s)}{\Phi(s)-\bar u(s)}\ud s
\\&\qquad \qquad \qquad+ 2|\bar{u}(s)-\Phi(s)|^2\ud s-\|G(\bar u(s))-G(\Phi(s))\|^2_{L_2(U_0,H)}\ud s  \\
&\qquad \qquad \qquad -\int_{E_0}|K(\bar u(s),\xi)-K(\Phi(s),\xi)|^2\ud \nu(\xi)\ud s
 \bigg)\ud t\geq 0.
\end{align*}
Since the last two terms in the inequality above are negative we have,
\begin{align}
\tilde \bE\int_0^T\varphi(t)\int_0^t e^{-\rho s}\bigg(
2\gamma\iprod{A\bar u(s)-A\Phi(s)}{\bar u(s)-\Phi(s)}+
2\iprod{F^R(\Phi(s))-Y^{R,\tilde{R}}(s)}{\Phi(s)-\bar u(s)}\nonumber\\
+ \rho|\bar{u}(s)-\Phi(s)|^2  \bigg)\ud s\ud t\geq 0.\label{minty}
\end{align}
We now let $\Phi=\bar u\pm \lambda\zeta\vv$ for $\lambda \ge 0$, $\vv\in V$
and $\zeta \in L^\infty([0,T]\times \tilde{ \Omega},\ud t\otimes \ud \tilde\bP;\mathbb{R})$ and use the linearity of the operator $A$ . Then dividing both sides by $\lambda$, letting $\lambda\rightarrow 0$ and using the hemicontinuity property of the operator $F^R$, we obtain
\begin{equation}
\tilde \bE\left(\int_0^T\varphi(t)
\left(\int_0^t e^{-\rho s}\zeta(s)\Vdualpair{F^R(\bar u(s))-Y^{R,\tilde{R}}(s)}{\vv}\ud s\right)\ud t\right)=0.
\end{equation}
Since $\varphi$, $\zeta$ and $\vv$ {are all} arbitrary, we conclude that 
\begin{equation}
Y^{R,\tilde{R}}= F^R(\bar u(s)) \quad  \ud \tilde\bP \otimes \ud t-\text{a.e.}
\end{equation}
Therefore, we have now proven that $\bar u$ solves \eqref{approx_eqn} with respect to the basis $(\tilde \Omega, \tilde{\mathcal{F}},(\bar{\mathcal{F}}_t)_{t\geq 0},\tilde \bP, \bar W,\bar \pi  )$ which marks the end of the passage to the limit $n\rightarrow \infty$ (with $R$ and $\tilde R$ fixed) and thus {concludes} the proof of Theorem \ref{prop:passage_to_limit}. 
\end{proof}
\begin{rem}
	Note that in Lemma \ref{conv_G} and in Lemma \ref{conv_K} we do not obtain convergence of the mentioned processes in $L^2(\tilde \Omega;L^2([0,T];H))$ {as in e.g. \cite{1LayerShallow} or \cite{DGHT}}. 
	The lack of such a result could pose a problem while passing to the limit $n\rightarrow \infty$ in Step 2. However this is taken care of by first carrying out \eqref{mult_exp} and then applying the It\^{o} formula using the function $\psi(u)=|u(t)|^2$, as can be observed in \eqref{imp_exp}.	 
\end{rem}

\section{Passage to the limit $R \rightarrow \infty$} \label{sec:R}

For this section, to emphasize the dependence on the parameter $R$, we will use the notation {$(\bar u_0^R, \bar u^R, \bar W^R,\bar \pi^R):=(\bar u_0, \bar u, \bar W,\bar \pi)$} to denote the solution found in the previous section {and thus omiting for the moment the dependence on $\tilde R$, e.g. $\bar u^R =\bar u^{R, \tilde R}$, etc. 
 Now,} for fixed $\tilde R$ we will pass to the limit $R \rightarrow \infty$ and show the existence of a martingale solution to the stochastic equation
{\begin{equation}\label{approx_eqn2}\begin{split}
	 u(t)+ \int_0^t[\gamma Au(s)+& \theta_{\tilde{R}}(|u|)B(u(s),u(s)) +F(u(s))]\ud s \\&=u(0)+\int_0^tG(u(s))\ud W(s)+ \int_{(0,t]}\int_{E_0}K(u(s-),\xi)\ud \hat\pi(s,\xi) .
\end{split}	\end{equation}} 
Recall that $X \subset (W^{k,p}(\mathcal{O}))^d$ for some $p>2$ and integer $k\geq 0$. Recall also that in Lemma \ref{uniformboundL2} we used the equicoercivity result \eqref{FR3} for the family $\{F^R\}_{R>R_0}$, derived in \cite{TT20}, to obtain bounds that were independent of the parameter $R$. This readily gives us the following energy estimates.
\begin{lem}\label{gencase_uniformboundL2}	Let $\bar u^R$ denote a martingale solution to \eqref{approx_eqn} with respect to the basis \\ $(\tilde \Omega, \tilde{\mathcal{F}},{(\bar{\mathcal{F}}^R_t)_{t\geq 0}},\tilde \bP, {\bar W^R,\bar \pi^R } )$ 
 as found in Theorem \ref{prop:passage_to_limit} in the previous section. Assuming all the hypotheses in Lemma \ref{uniformboundL2}, 
 for large enough $R>R_0$ we have
	\begin{align}\label{apriori_gencase}
	&\tilde \bE	\sup_{s\in[0,T]}
	|\bar u^R(s)| ^{2}+ (\gamma+c_1)
	\tilde \bE\int_0^T\|{\bar u^R(s)}\|^2\ud s
	\le K_1.
	\end{align}
	Furthermore, for some $K_2(T,\gamma,\rho,k,p,K_1)>0$ independent of $R$, we obtain the following bounds,
	\begin{align}\label{apriori2}
	\sum_{i=0}^k  \left( \tilde \bE \int_0^T\int_{\mathcal{O}}|D^i \bar u^R|_p^p \mathbbm{1}_{\{|D^i \bar u^R|_p<R\}} \ud x\ud t + R^{p-2} \tilde \bE \int_0^T\int_{\mathcal{O}}|D^i \bar u^R|_p^2 \mathbbm{1}_{\{|D^i \bar u^R|_p>R\}} \, \ud x\ud t \right) &\leq K_2.
	\end{align}
\end{lem}
\begin{rem}
	Note that the bound {for the} second term on the left hand side of \eqref{apriori2} further implies {for the same $K_2$ as in \eqref{apriori2}},
	\begin{align}\label{apriori3}
	\sum_{i=0}^k    \tilde \bE \int_0^T \int_{\mathcal{O}} \mathbbm{1}_{\{|D^i\bar u^R|_p>R\}}\ud x\ud t \leq \frac{K_2}{R^{p}}.
	\end{align}
\end{rem}
\begin{proof}
The estimate \eqref{apriori_gencase} is readily obtained by using the convergence results \eqref{weakstar} and \eqref{weakV} and Lemma \ref{uniformboundL2} along with the lower semicontinuity property of the 
norms involved.\\
We will now very briefly discuss the results from \cite{TT20} that lead to \eqref{apriori2}; {this discussion will be useful below}. 
First it can be seen that the assumptions \eqref{F1}-\eqref{gc} give us, for our construction \eqref{gen_def_fR}, the 
 existence of some constant $c_6 >0$ independent of $R$  such that
\begin{align}
f_i (x ) \cdot {x} \geq c_1 |x|_p - c_6 \qquad \forall x \in \R^{d^{i+1}}, i=0,1,...,k.\label{gen_ineqR1}
\end{align}
These assumptions also give us that for $x\in \R^{d^{i+1}}$, if $|x|_{p}>R>1$, then for ${c_7 \geq\max\{c_2,c_6\}}>0$ (independent of $R$) we have
\begin{align}\label{gen_ineqR4}
f_i \left({\frac{Rx}{|x|_p}}\right) \cdot x + \left( x-{\frac{Rx}{|x|_p}}  \right)  \,{D}^2\mathscr{J}_i \left({{\frac{Rx}{|x|_p}} }\right) \cdot x \geq (c_1 R^{p-2}-c_5)|x|_{p}^2 .  
\end{align}
Then the definitions \eqref{gen_def_FR}-\eqref{gen_def_fR} along with a direct application of \eqref{gen_ineqR1} and \eqref{gen_ineqR4}, give us the following lower bound for any $v \in V$,
\begin{align}\label{FRlower}
(F^R(v),v) \geq \sum_{i=0}^k   \int_{\mathcal{O}}\left(|D^i v|_p^p \mathbbm{1}_{\{|D^i  v|_p<R\}} +
R^{p-2}|D^i v|_p^2 \mathbbm{1}_{\{|D^i v|_p>R\}} \, \right)\ud x  -C,
\end{align}
where the constant $C=C(p,k)>0$ is independent of $R$ and $v$.\\
For detailed proofs of the inequalities \eqref{gen_ineqR1} and \eqref{gen_ineqR4}  and thus {of} \eqref{FRlower}, we refer the reader to \cite{TT20} wherein these inequalities are derived in the deterministic setting.
 {We} now apply It\^o's formula to \eqref{approx_eqn} with respect to the basis $(\tilde \Omega, \tilde{\mathcal{F}},{(\bar{\mathcal{F}}^R_t)_{t\geq 0}},\tilde \bP, {\bar W^R,\bar \pi^R } )$ using the function $|\bar{u}^R|^2$, then take supremum in time and expectations of both sides, as in \eqref{ito}. However unlike in \eqref{e11.16} we use \eqref{FRlower} instead of \eqref{FR3}. Note that  the noise terms are adapted to the filtration $(\bar{\mathcal{F}}^R_t)_{t\geq 0}$ and we obtain bounds like in \eqref{I(t)} -\eqref{N(t)}. 

 This results in the following bounds, 
\begin{align*}
&\tilde \bE \sup_{t\in[0,T]}\sum_{i=0}^k   \int_0^t\int_{\mathcal{O}}\left(|D^i \bar{u}^R|_p^p \mathbbm{1}_{\{|D^i  \bar{u}^R|_p<R\}} +
R^{p-2}|D^i \bar{u}^R|_p^2 \mathbbm{1}_{\{|D^i \bar{u}^R|_p>R\}} \, \right)\ud x\ud t  -C(T,p,k)  \\
& \qquad \qquad \qquad   \leq {\bE}| u_0|^2 + \tilde \bE\sup_{t\in[0,T]}
|{\bar u^R(t)}|^2+ C(\rho)\tilde \bE\int_0^T(1+\Hnorm{\bar u^R(s)}^2)\ud s \\
&\qquad \qquad \qquad \le K_2 \quad \text{ thanks to \eqref{apriori_gencase}}.
\end{align*}
\end{proof}
\begin{prop}\label{ur_boundinH}
Let $\bar u^R$ be a martingale solution to \eqref{approx_eqn}  with respect to the stochastic basis $(\tilde \Omega, \tilde{\mathcal{F}},{(\bar{\mathcal{F}}^R_t)_{t\geq 0}},\tilde \bP, {\bar W^R,\bar \pi^R } )$ {as given by Theorem \ref{prop:passage_to_limit}}, 
	then for every $\alpha\in (0,\frac 12)$ there exists a constant $C=C(\alpha,T,\tilde{R},K_1)>0$ such that
	\begin{equation}\label{frac_utilR}
	\sup_{R > R_0}	\tilde \bE\norm{\bar u^R}^{\frac{p}{p-1}}_{W^{\alpha,\frac{p}{p-1}}(0,T;{{D(A)'}+X'})}\le C.
		\end{equation}
	Furthermore, due to Lemma \ref{Sobolev1}  and an argument like in Proposition \ref{tightnessinV}, the estimate \eqref{frac_utilR} implies that the laws of $\bar u^{R}$ are tight in $L^2(0,T;H)$.
\end{prop}

\begin{proof}
First observe that Proposition \ref{noisefraction} holds for $\beta=\frac{p}{p-1}$ and {$P_n(G(u^n))$ and $P_n(K(u^n))$ replaced by $G(\bar u^R)$ and $K(\bar u^R)$} respectively.\\
	Thus the rest of the proof follows the proof of Proposition \ref{boundinW} with slight modifications to the estimates found in \eqref{frac_B} and \eqref{frac_F} as described below. First since $\alpha<1$, $W^{\alpha,\frac{p}{p-1}}(0,T;D(A)'+X')$ is continuously embedded in $W^{1,\frac{p}{p-1}}(0,T;D(A)'+X')$ and
		we have, {as in \eqref{gn}},		
	\begin{equation} 
	\begin{aligned}
&	~\norm{g^R}^{\frac{p}{p-1}}_{W^{\alpha, \frac{p}{p-1}}(0,T;{D(A)'+X'})}\\
	&:=  \norm{\bar u_0^R -\int_0^{\cdot} \gamma A\bar u^R(s) \ud s  - \int_0^{\cdot} B_{\tilde R}(\bar u^R(s)) \ud s - \int_0^{\cdot}F^R(\bar u^R(s)) \ud s}^{\frac{p}{p-1}}_{W^{\alpha,\frac{p}{p-1}}(0,T;D(A)'+X')} \label{gR} \\
	&\lesssim  ~ \Hnorm{\bar u^R_0}^{\frac{p}{p-1}}+ \int_0^T \gamma\Hnorm{A\bar u^R(s)}^{\frac{p}{p-1}}_{{D(A)'}} \ud s + \int_0^T \Hnorm{B_{\tilde{R}}(\bar u^R(s))}_{{D(A)'}}^{\frac{p}{p-1}} \ud s + \int_0^T \Hnorm{F^R(\bar u^R(s))}^{\frac{p}{p-1}}_{{X'}} \ud s  ,
	\end{aligned}
	\end{equation}
	where the hidden constants depend on $T,p$ but not on $R$.\\
	Recall that the law of $u^R_0$ is the same as that of $u_0$. Thus, using the fact that $\frac{p}{p-1}<2$, the expected values of the first two terms {above} are treated using the H\"older inequality. For example,
	\begin{align}
	\tilde \bE\int_0^T
	\gamma\Hnorm{A\bar u^R(s)}_{D(A)'}^{\frac{p}{p-1}}\ud s& \lesssim \gamma T^{\frac{p+1}{2p}}\left(\tilde \bE\int_0^T
	\Hnorm{A\bar u^R(s)}_{V'}^2\ud s\right)^{\frac{p-1}{2p}}\nonumber \\&
	\le C(T,p,\gamma)\left(\tilde \bE\int_0^T
	\Vnorm{\bar u^R(s)}^2\ud s\right)^{\frac{p-1}{2p}} \nonumber\\&	\leq C(T,p,\gamma,K_1),\label{fracAR}
	\end{align}	
where the constant $C(T,p,\gamma,K_1)>0$ is independent of $R$.\\
Next, {using} \eqref{eqn:B_bound2} and remembering that $\theta_{\tilde R}(x)$ vanishes for $x \geq 2\tilde R$, we find:
	\begin{align}
	\tilde \bE\int_0^T|B_{\tilde{R}}(\bar u^R(s),\bar u^R(s))|_{D(A)'}^{\frac{p}{p-1}} \ud s &\lesssim \tilde \bE \int_0^T \big(\theta_{\tilde{R}}(|\bar u^R|)\big)^{\frac{p}{p-1}}|\bar u^R(s)|^{^{\frac{p}{p-1}}}\|\bar u^R(s)\|^{\frac{p}{p-1}} \ud s\nonumber\\ & \lesssim \tilde{R}^{\frac{p}{p-1}}T^{\frac{p-2}{2p-2}}\left(\tilde \bE \int_0^T\|\bar u^R(s)\|^2\ud s\right)^{\frac{p}{2p-2}}  .	\label{BRbound}
	\end{align}
	Thanks to Lemma \ref{gencase_uniformboundL2} the last line in \eqref{BRbound} is bounded by a constant for $C=C(p, T,\tilde{R},K_1)>0$ independent of $R$.\\
	{Next, as proven in Theorem 4.1 in \cite{TT20}, using the H\"older inequality on the expressions in \eqref{gen_def_FR}-\eqref{gen_def_fR} we obtain
		\begin{align*}
		|F^R(\bar u^R)|_{L^{\frac{p}{p-1}}(0,T;X')}^{\frac{p}{p-1}} &\leq \||D^i\bar u^R|_p\mathbbm{1}_{\{|D^i \bar u^R_{}|_p \leq R\}}\|^{p}_{L^{p}(0,T;L^p(\mathcal{O}))}
\\	&{\hspace{-0.1in}}+ \|R^{\frac{p-2}{2}} |D^i\bar u^R|_p\mathbbm{1}_{\{|D^i \bar u^R_{}|_p \geq R\}}\|^{{\frac{p}{p-1}}}_{L^{2}(0,T;L^2(\mathcal{O}))}\| R^p \mathbbm{1}_{\{|D^i\bar u^R|_p\geq R\}}\|^{{\frac{p-2}{2(p-1)}}}_{L^1(0,T;L^{1}(\mathcal{O}))}.
		\end{align*}
		Then we take expectation of both sides and apply the H\"older inequality observing that $\frac{p}{2(p-1)}+\frac{p-2}{2(p-1)}=1$. Using \eqref{apriori2} and \eqref{apriori3}, for a constant $C=C(T,p,K_2)>0$ independent of $R$, we then obtain
		\begin{align}\label{FRbound}
		\tilde \bE\int_0^T|F^{{R}}(\bar u^R(s))|_{X'}^{\frac{p}{p-1}} \ud s \leq C(T,p,K_2).
		\end{align}}
	 Thus we obtain the {desired result \eqref{frac_utilR}}.
\end{proof}
Next define, (compare to \eqref{chiu}-\eqref{laws_n}), 
\begin{align}
\sX_1&=L^2(0,T;H)\cap \mathcal{D}([0,T];V'+X'),\label{chi1}\\
\mu_u^R(\cdot)&=\tilde \bP(\bar{u}^R\in\cdot) \in Pr\left(L^2(0,T;H)\cap
\mathcal{D}([0,T];V'+X')\right).\label{laws_R}
\end{align}

\begin{prop}\label{tightnessinchi1}
Let $\bar u^R$ be a solution to \eqref{approx_eqn}  with respect to the basis $(\tilde \Omega, \tilde{\mathcal{F}},{(\bar{\mathcal{F}}^R_t)_{t\geq 0}},\tilde \bP, {\bar W^R,\bar \pi^R } )$ {as given by Theorem \ref{prop:passage_to_limit}}. 
	Then the laws $\{\mu_u^R\}_{R>R_0}$ of $\{\bar u^R\}_{R>R_0}$ are tight in the space $\sX_1$.	
\end{prop}
\begin{proof}
First we will show that the sequence of laws of $\bar u^R$ is tight in the space $\mathcal{D}([0,T];V'+X')$ by showing that {a suitable version of the} Aldous condition \eqref{Aldous} is satisfied.	 Observe that $\bar u^R$ satisfies:
	\begin{equation}
		\begin{aligned}\label{MR}
			\bar u^R(t) = &~ \bar u^R_0 - \int_0^t \gamma A \bar u^R(s)\ud s - \int_0^t B_{\tilde{R}}(\bar u^R(s),\bar u^R(s))\ud s - \int_0^t  F^R(\bar u^R(s))\ud s
			\\&~ +\int_0^tG(\bar u^R(s-))\ud  {\bar W^R}(s) + \int_{(0,t]}\int_{E_0}
			K(\bar u^R(s-),\xi){\ud\hat{\bar \pi}^R}(s,\xi)
			\\=: &~
			\bar u^R_0+\mathcal{M}_1(t)+
			\mathcal{M}_2(t)+
			\mathcal{M}_3(t)+
			\mathcal{M}_4(t)+\mathcal{M}_5(t).
		\end{aligned}
	\end{equation}
	Let $(\tau_R)_{R>R_0}$ be a sequence of stopping times where $0\le \tau_R \le T$.
	 Thanks to the similarity of the bounds in Lemma \ref{gencase_uniformboundL2} to that in Lemma \ref{uniformboundL2}, the proof that the terms $\mathcal{M}_1,\mathcal{M}_2, \mathcal{M}_4$ and $\mathcal{M}_5$ satisfy the Aldous condition follows from the proof of Proposition \ref{tightnessinD} identically. And so we obtain, for some $C=C(\tilde{R},T,\gamma,K_1)>0$, independent of $R$
	 \begin{align}
	\tilde \bE\norm{\mathcal{M}_i(\tau_R+\delta)-\mathcal{M}_i(\tau_R)}_{V'} \leq C\delta^{\frac12} \quad \text{ for } i=1,2,4,5.
	 \end{align} We will discuss the term $\mathcal{M}_3$ that is treated slightly differently here than in the previous section. 

	Using \eqref{FRbound} we obtain for some $C(T,p,K_1)>0$ independent of $R$,
		\begin{equation}
	\begin{aligned}
	\tilde \bE\left[\norm{\mathcal{M}_3(\tau_R +\delta)-\mathcal{M}_3(\tau_R)}_{X'}\right]	& \le C\tilde \bE\int_{\tau_R}^{\tau_R +\delta}|F^R(\bar u^R(s))|_{X'}\ud s
	\\
	& \le C\delta^{\frac1p}\left(\tilde \bE\int_{\tau_R}^{\tau_R+\delta}|F^R(\bar u^R(s))|^{\frac{p}{p-1}}_{X'}\ud s \right)^{\frac{p-1}{p}}
	\\
	& \le C\delta^{\frac1p} .
	\end{aligned}
	\end{equation}
This proves that the Aldous condition \eqref{Aldous} is satisfied by {$\bar u^R$  with $\Xi=V'+X'$, $\alpha=1$ and $\epsilon=\frac{1}{p}$. This implies that the laws of $(\bar u^R)_{\{R>R_0\}}$ are tight in the space $\mathcal{D}([0,T];V'+X')$}. Thus combining this result with Proposition \ref{ur_boundinH} and {proceeding as} in Proposition \ref{tightnessinchi} we conclude the proof of Proposition \ref{tightnessinchi1}.
\end{proof}
Let {$\Upsilon_1:= H \times \sX_1 \times C([0,T];U) \times \mathcal{N}^{\# *}_{[0,\infty)\times E}$}. Recall that due to Theorem \ref{skorohod1} we know that the laws of $\bar W^R$ (and similarly the laws of $\bar \pi^R$) are equal to one element.
Thus the tightness result in Proposition \ref{tightnessinchi1} together with an application of Prohorov's theorem gives us the existence of a measure $\mu^\infty_1$ on $\Upsilon_1$ such that the sequence $\mu^R$, of the joint probability laws of $({\bar u}_0^{R},
{\bar{ u}}^{R},{\bar{ W}}^{R},{\bar \pi}^{R})$ indexed by some infinite set $\Lambda \subseteq \mathbb{N}$, converges weakly to $\mu^\infty_1$.  Next we apply {Theorem \ref{skorohodtheorem}} to obtain the following strong convergence result. 
\begin{thm}\label{skorohod2}
For any $R \in \Lambda$ there exist {$\Upsilon_1$}-valued random variables $(\bar{\bar u}_0^{R},
\bar{\bar{ u}}^{R},\bar{\bar{ W}}^{R},\bar {\bar\pi}^{R})$
and $(\bar{\bar u}_0,\bar{\bar u} ,\bar {\bar W}, \bar {\bar \pi})$, {depending on $\tilde R$}, such that
\begin{myenum}
	\item The probability laws of $(\bar{\bar u}_0^{R},
	\bar{\bar{ u}}^{R},\bar{\bar{ W}}^{R},\bar {\bar \pi}^{R})$ and  $(\bar{\bar u}_0,\bar{\bar u} ,\bar {\bar W}, \bar {\bar \pi})$ are $\mu^{R}$ and $\mu_1^\infty$, respectively,
	\item $(\bar {\bar W}^{R},\bar {\bar \pi}^{R})=(\bar {\bar{W}},\bar {\bar{\pi}}) $ everywhere on $\tilde{ \Omega}$,
	\item $(\bar{\bar u}_0^{R},
	\bar{\bar{ u}}^{R},\bar{\bar{ W}}^{R},\bar {\bar \pi}^{R})$ converges almost surely to $(\bar{\bar u}_0,\bar{\bar u} ,\bar {\bar W}, \bar {\bar\pi})$ in the topology of { $\Upsilon_1$}, i.e.,
	\begin{align}
	&{\bar{\bar u}_0^{R}\rightarrow \bar{\bar{u}}_0 \text{ in } H, \quad \tilde \bP\tn{-a.s.},\label{uork}}	\\
	&\bar{ \bar u^{}}^R\rightarrow
	\bar{\bar u} \text { in } L^2(0,T;H)\cap \mathcal{D}([0,T];V'+X'),
	\quad \tilde \bP\tn{-a.s.},\label{ur_convH}
	\end{align}
	\item  $\bar {\bar\pi} $, is a time homogeneous Poisson random measure over $(\tilde{ \Omega},\tilde { {\mathcal{F}}}, \tilde \bP)$ with intensity measure $\ud t \otimes \ud \nu$.
	
	\item
	$\bar{\bar u}^R(0)=\bar{\bar u}_0^R$ $\tilde \bP$-a.s. and 
	$(\bar {\bar u}^{R},\bar {\bar W},\hat{\bar {\bar \pi}})$ satisfies
	the following equation $\tilde \bP$-a.s. in $V'$ for every $t\in[0,T]$
	\begin{align}\label{eq_R}
	&\bar{\bar u}^{R}(t)+\int_0^t\left[\gamma A(\bar {\bar u}^{R}(s))+ B_{\tilde{R}}(\bar {\bar u}^{R}(s),\bar {\bar u}^{R}(s))+ F^R(\bar {\bar u}^{R}(s))\right]\ud s
	\nonumber\\
	&\qquad \qquad= \bar{\bar u}^{R}_0+{\int_0^tG(\bar {\bar u }^{R}(s-))\ud \bar {\bar W}(s)+\int_{(0,t]}\int_{E_0}K(\bar {\bar u}^{R}(s-),\xi)\ud \hat{\bar {\bar \pi}}(s,\xi)},
	\end{align}
with respect to some filtration $(\bar{\bar{\mathcal{F}}}^R_t)_{t\geq 0}$.
\end{myenum}
\end{thm}
The filtrations $(\bar{\bar{\mathcal{F}}}^R_t)_{t\geq 0}$ and $(\bar{\bar{\mathcal{F}}}_t)_{t\geq0}$ are defined like in \eqref{Ft} and depend on $\tilde R$. 
{We note that $\bar{\bar u}_0$ might depend on $\tilde R$ but the only fact relevant to our analysis is that it has the same law as $u_0$.}
{\begin{rem}\label{diffrv}
		The random variable $(\bar{\bar u}_0^{R},
		\bar{\bar{ u}}^{R},\bar{\bar{ W}}^{R},\bar {\bar\pi}^{R})$ may or may not be the same as the random variable $({\bar u}_0^{R},
		{\bar{ u}}^{R},{\bar{ W}}^{R}, {\bar\pi}^{R})$. This solely depends on the measure $\mu_1^\infty$ which decides the refined cover of the space $\Upsilon_1$, 
		in the proof of the Skorohod representation theorem {(see Theorem \ref{skorohodtheorem})}.
\end{rem}}	
{ In what follows, we will consider $R$ belonging to the denumerable set $\Lambda$.} Comparing equations \eqref{approx_eqn} and \eqref{eq_R} we can see that $\bar{\bar u}^R$ also satisfies the uniform bounds found in equation \eqref{apriori_gencase} in Lemma \ref{gencase_uniformboundL2}. 
	Thus we obtain that $\bar {\bar u} \in L^\infty([0,T];L^2(\tilde \Omega,H)) \cap L^2([0,T]\times \tilde \Omega,\ud t \otimes  \ud \tilde \bP;V)$ and {there exists} some subsequence $R \rightarrow \infty$, {such that}
\begin{align}\label{conv_ur}
\bar {\bar u}^R &\rightharpoonup \bar {\bar u} \text{ weakly in } L^2([0,T]\times \tilde \Omega,\ud t \otimes \ud \tilde \bP;V),\\
\bar {\bar u}^R &\rightharpoonup \bar {\bar u} \text{  weak$^*$  in } L^\infty([0,T];L^2(\tilde \Omega,H)).
\end{align}
Bounds in \eqref{FRbound} also give us, {for a subsequence},
\begin{equation}\label{YR}
Y^{R,\tilde{R}}= F^R(\bar {\bar u}^{R})\rightharpoonup  Y^{\tilde{R}} \text{ weakly in }
L^{\frac{p}{p-1}}([0,T] \times\tilde \Omega,\ud t\otimes \ud \tilde\bP;X').
\end{equation}

{We denote $\bar{\bar {\mathscr{S}}}=({\tilde \Omega},\tilde {\mathcal{F}}, (\bar{ \bar{\mathcal{F}}}_t)_{t\ge 0},\tilde \bP,\bar {\bar W},\bar {\bar\pi})$. Note that our aim now is to show that $(\bar{\bar {\mathscr{S}}},\bar{\bar u})$ is a martingale solution of \eqref{approx_eqn2} which includes showing that $ Y^{\tilde{R}}=F({\bar {\bar u}})$}.
For that purpose we begin by defining the truncated functions $\xi^R_i:\tilde{ \Omega}\times [0,\infty)\times \mathcal{O} \rightarrow \mathbb{R}^{d^{i+1}} $ for $i=0,...,k$ as follows,
\begin{align}\label{def_xi}
\xi_i^R(\omega,t,x)= D^i\bar {\bar u}^R \mathbbm{1}_{\{|D^i \bar {\bar u}^R|_p \leq R\}}.
\end{align}
Here we recall the notation $|\cdot|_{p} \, :=|\cdot|_{l^{p}(\R^{d^{i+1}})}$ 
given $x \in \R^{d^{i+1}}$ and $i=0,1,...,k$.
Due to the availability of bounds independent of $R$ for the first left hand side term in \eqref{apriori2} we can see that for some subsequence $R\rightarrow \infty$ and some $\eta_i$, $i=0,...,k$, we have that
\begin{align}\label{conv_xir}
\xi^R_i \rightharpoonup \eta_i \text{ weakly in } L^p([0,T]\times \tilde \Omega, \ud t \otimes \ud \tilde\bP;(L^p(\mathcal{O}))^{d^{i+1}}).
\end{align}

Observe also that \eqref{apriori3} implies,
\begin{align}\label{conv_1r}
\mathbbm{1}_{\{|D^i \bar {\bar u}^R(x,t)|_p<R\}} \rightarrow 1 \text{ in probability for almost every $t\in [0,T]$ and $x \in \mathcal{O}$ as } R \rightarrow \infty.
\end{align}
An application of the Lebesgue dominated convergence theorem and \eqref{conv_ur}, \eqref{conv_xir} and \eqref{conv_1r} tells us that $\xi^R_i \rightharpoonup D^i\bar {\bar u} \text{ weakly in } L^2([0,T]\times \tilde \Omega, \ud t \otimes \ud \tilde\bP;(L^2(\mathcal{O}))^{d^{i+1}})$. This, due to \eqref{conv_xir}, further implies that in fact $\eta_i=D^i\bar {\bar u}$ for $i=0,...,k$ and thus
\begin{align}
\bar {\bar u} &\in L^p([0,T]\times \tilde \Omega, \ud t \otimes \ud \tilde\bP;X) \cap  L^2([0,T]\times \tilde \Omega,\ud t \otimes  \ud \tilde \bP;V),\label{uinlp}\\
\xi^R_i &\rightharpoonup D^i \bar {\bar u}\text{ weakly in } L^p([0,T]\times \bar { \Omega}, \ud t \otimes \ud \tilde\bP;(L^p(\mathcal{O}))^{d^{i+1}}) \text{ as }R \rightarrow \infty.\label{conv_xir_u}
\end{align}
Now we can prove the main result of this section:
\begin{thm}\label{gencasetheorem} 
		Suppose that $\bE|u_0|^2 <\infty$. Then $\bar{\bar u}$ (denoted by $\bar{\bar u} ^{\tilde R}$ in the later section) solves \eqref{approx_eqn2}, with respect to the stochastic basis $({\tilde \Omega},\tilde {\mathcal{F}}, (\bar{ \bar{\mathcal{F}}}_t)_{t\ge 0},\tilde \bP,\bar {\bar W},\bar {\bar\pi})$ with initial condition $\bar {\bar {u}}_0$ that has the same law as $u_0$.
\end{thm}
\begin{proof}
	Here again {as for Theorem \ref{prop:passage_to_limit}} we will pass to the limit $R \rightarrow \infty$ and prove Theorem \ref{gencasetheorem} in two steps. {We will skip some details very similar to the proof of Theorem \ref{prop:passage_to_limit}.}\\	
\underline{Step 1}:\\
	We first show, {as done before in \eqref{step1}}, that $\bar {\bar u} $ solves the following equation $\tilde \bP$-a.s. and for a.e. $t \in [0,T]$
	\begin{align}\label{step1_R}
	&\bar {\bar u}(t)+\int_0^t[\gamma A(\bar {\bar u}(s))+B_{\tilde{R}}(\bar {\bar u}(s),\bar {\bar u}(s))+Y^{\tilde{R}}(s)]\ud s\nonumber\\
	&\hspace{1.5in}=\bar {\bar u}_0+\int_0^tG(\bar {\bar u}(s))\ud\bar {\bar W}(s)+\int_{(0,t]} \int_{E_0}K(\bar {\bar u}(s-),\xi)\ud\hat{\bar {\bar \pi}}(s,\xi),
	\end{align}
	where $Y^{\tilde{R}}$ was defined in \eqref{YR}. The proof of \eqref{step1_R} follows that of \eqref{step1} in Theorem \ref{prop:passage_to_limit}.\\ 
  Note also that, thanks to Theorem \ref{skorohod1} and Theorem \ref{skorohod2}, we have that $\bar{\bar u}_0$ has the same law as $u_0$.\\ 
   
\underline{Step 2}:  \\In this step we will show that
\begin{align*}
&Y^{\tilde{R}}=F(\bar{\bar{u}}), \quad  \ud   \tilde\bP \otimes \ud t \tn{-a.e.}
\end{align*}
For that purpose we prove the following lemma for the truncated functions defined in \eqref{def_xi}:
\begin{lem}\label{4claims}
The following four convergences hold true for any $ \phi\in  L^{p}([0,T]\times \tilde \Omega,\ud t\otimes \ud \tilde\bP;X)$, for any $\varphi \in L^\infty([0,T],\ud t;\mathbb{R})$ such that $\varphi \geq 0$ and {for $\rho>0$ appearing in \eqref{eqn:G,K_H_growth}-\eqref{eqn:G,K_H_Lipschitz}},
\begin{subequations}
	\begin{align}
		&\limsup_{R \rightarrow \infty}\sum_{i=0}^k \tilde \bE \int_0^T\varphi(t)\int_0^t e^{-\rho s} \, ({f^R_i(\xi_i^R(s))},{\xi^R(s)}) \ud s \ud t \leq \tilde \bE \int_0^T\varphi(t)\int_0^t e^{-\rho s}   \iprod{Y^{\tilde{R}}}{ \bar{\bar u}(s)} \ud s \ud t, \label{evol_claim1}\\
	&\limsup_{R \rightarrow \infty} \sum_{i=0}^k\tilde \bE \int_0^T\varphi(t)\int_0^t e^{-\rho s}  \, \iprod{f^R_i(D^i\phi(s))}{\xi_i^R(s)} \ud s \ud t = \tilde \bE\int_0^T\varphi(t)\int_0^t e^{-\rho s}   \iprod{F(\phi(s))}{\bar{\bar u}(s)}\ud s \ud t, \label{evol_claim2}\\
	&\limsup_{R \rightarrow \infty} \sum_{i=0}^k\tilde \bE\int_0^T\varphi(t)\int_0^t e^{-\rho s}   \, \iprod{f^R_i(\xi_i^R(s))}{D^i\phi(s)} \ud s \ud t= \tilde \bE\int_0^T\varphi(t)\int_0^t e^{-\rho s}  \iprod{Y^{\tilde{R}}}{\phi(s)} \ud s \ud t,\label{evol_claim3}\\
		&\limsup_{R \rightarrow \infty} \sum_{i=0}^k\tilde \bE \int_0^T\varphi(t)\int_0^t e^{-\rho s}   \, \iprod{f^R_i(D^i\phi(s))}{D^i\phi(s)}\ud s \ud t=\tilde \bE \int_0^T\varphi(t)\int_0^t e^{-\rho s}   \iprod{F(\phi(s))}{ \phi(s)} \ud s \ud t.\label{evol_claim4}
	\end{align}
\end{subequations}
\end{lem}
\begin{proof} Here we describe the outline of the proofs of the above four claims. As mentioned earlier, the reader is referred to {\cite{TT20} for an} in depth description in the deterministic setting.	\\
	
\textit{Proof of claim (\ref{evol_claim1})}: First we can observe that for some $K_3>0$ independent of $R$ we have
	\begin{align*}
| \sum_{i=0}^k \tilde \bE\int_0^T\varphi(t)\int_0^te^{-\rho s}(f^R_i(\xi_i^R) , \xi_i^R) \ud s \ud t| \leq C(T,\|\varphi\|_{L^\infty})\tilde \bE\sum_{i=0}^k \|\xi^R_i\|^p_{L^p([0,T]\times \tilde \Omega;\, \ud t \otimes \ud \tilde\bP;\, (L^p(\mathcal{O}))^{d^{i+1}})} \leq K_3.
	\end{align*}
	This implies that, up to a subsequence if necessary, the limit\\ $\lim_{R \rightarrow \infty}\sum_{i=0}^k \tilde \bE\int_0^T\varphi(t)\int_0^te^{-\rho s}(f^R_i(\xi_i^R) , \xi_i^R) \ud s \ud t$ exists. Next we can see that for any $R>1$ we have
\begin{align}
\int_0^te^{-\rho s}(f^R_i(\xi_i^R) , \xi_i^R) \ud s = \int_0^te^{-\rho s}(f^R_i(\xi_i^R),D^i \bar{\bar u}^R) \ud s &\leq \int_0^te^{-\rho s} (f^R_i(D^i \bar{\bar u}^R),D^i \bar{\bar u}^R) \ud s.\label{evol_claim1_1}
\end{align}
Summing both sides in $i$ we obtain that,
\begin{align}\label{evol_claim1_2}
\sum_{i=0}^k\int_0^te^{-\rho s}(f^R_i(\xi_i^R(s)) , \xi_i^R(s)) \ud s & \leq \int_0^t e^{-\rho s}(F^R(\bar{\bar u}^R(s)),\bar{\bar u}^R(s))\ud s.
\end{align} 
Now {we form an equation similar to \eqref{mult_exp}}, then apply It\^o's formula to the process in \eqref{eq_R}, take expectation of both sides and finally use \eqref{evol_claim1_1} to obtain, 
	\begin{align*}
& \sum_{i=0}^k  \tilde \bE\int_0^te^{-\rho s}(f^R_i(\xi_i^R(s)),\xi_i^R(s)) \ud s  \leq \tilde \bE \int_0^te^{-\rho s}(F^R(\bar {\bar u}^R(s)),\bar {\bar u}^R(s))\ud s\\
&\qquad \qquad=  \tilde \bE\Big[\int_0^te^{-\rho s}\left({\frac12\|G(\bar {\bar u }^{R}(s))\|^2_{L_2(U_0,H)}}\ud s- \gamma({A(\bar {\bar u}^{R}(s))},{\bar {\bar u}^{R}(s)})\ud s-	\frac12{\rho} | {\bar{\bar u}^{R}(s)}|^2\ud s\right)\\
& \qquad \qquad +\frac12 \int_{(0,t]}{\int_{E_0} e^{-\rho s}|K(\bar {\bar u}^{R}(s),\xi)|^2}\ud {\bar {\bar\pi}}(s,\xi)\Big]-\frac12 \tilde \bE \left(e^{-\rho t}|\bar{\bar u}^R(t)|-|\bar{\bar u}^R_0|^2\right).
\end{align*}Using $|a|^2=|a-b|^2 +2\iprod{a}{b} -|b|^2$ we obtain that the following {is valid} for any $0\leq \varphi\in L^\infty([0,T],\ud t;\mathbb{R})$
\begin{align*}
\tilde \bE\int_0^T \varphi(t) &\int_0^te^{-\rho s}(F^R(\bar {\bar u}^R(s)),\bar {\bar u}^R(s))\ud s\ud t\\
&=  \tilde \bE\int_0^T\varphi(t)\int_0^te^{-\rho s}\Big[{\frac12\|G(\bar {\bar u }^{R}(s))-G(\bar {\bar u }(s))\|^2_{L_2(U_0,H)}}\ud s-\frac12{\|G(\bar {\bar u }(s))\|^2_{L_2(U_0,H)}}\ud s\\
&+ \frac12{\int_{E_0}|K(\bar {\bar u}^{R}(s),\xi)-K(\bar {\bar u}(s),\xi)|^2}\ud \nu(\xi)\ud s-\frac12{\int_{E_0}|K(\bar {\bar u}(s),\xi)|^2}\ud \nu(\xi)\ud s\\
&+ \iprod{G(\bar {\bar u }(s))}{G(\bar {\bar u }^R(s))}_{L_2(U_0,H)}\ud s+\iprod{K(\bar {\bar u}(s),\xi)}{K(\bar {\bar u}^R(s),\xi)} \ud\nu(\xi)\ud s\\
&+\gamma({A(\bar {\bar u}(s))},\bar {\bar u}(s))\ud s-\gamma\iprod{{A(\bar {\bar u}^{R}(s))- A(\bar {\bar u}(s))}}{{\bar {\bar u}^{R}(s)}-\bar {\bar u}(s)}\ud s-\gamma\iprod{A(\bar {\bar u}^{R}(s))}{\bar {\bar u}(s)}\ud s\\
&- \gamma\iprod{A(\bar {\bar u}(s))}{\bar {\bar u}^R(s)}\ud s
-\frac12 \rho \left(|\bar {\bar u}^{R}(s)-\bar {\bar u}(s)|^2-|\bar {\bar u}(s)|^2 +2\iprod{\bar {\bar u}^{R}(s)}{\bar {\bar u}(s)}\right)\ud s\Big]\ud t \\
&-\frac12 \tilde \bE \int_0^T\varphi(t)(e^{-\rho t}|\bar{\bar u}^R(t)|-|\bar{\bar u}^R_0|^2)\ud t.
\end{align*}
Using the fact that $-({A(\bar {\bar u}^{R}(s))-A(\bar {\bar u}(s))},{\bar {\bar u}^{R}(s)}-\bar {\bar u}(s)) \leq 0$ and applying an argument that follows the equation \eqref{nk} we obtain the following as $R \rightarrow \infty$
\begin{align}
&{\limsup_{R\rightarrow \infty}\sum_{i=0}^k  \tilde \bE\int_0^T\varphi(t)\int_0^te^{-\rho s}(f^R_i(\xi_i^R),\xi_i^R) \ud s \ud t}\nonumber\\
&\qquad \qquad \leq  \tilde \bE\int_0^T\varphi(t)\int_0^te^{-\rho s}\Big[-\frac12{\|G(\bar {\bar u }(s))\|^2_{L_2(U_0,H)}}\ud s-\frac12{\int_{E_0}|K(\bar {\bar u}(s),\xi)|^2}\ud\nu(\xi)\ud s\nonumber\\
&\qquad \qquad + \iprod{G(\bar {\bar u }(s))}{G(\bar {\bar u }(s))}_{L_2(U_0,H)} \ud s+\iprod{K(\bar {\bar u}(s),\xi)}{K(\bar {\bar u}(s),\xi)}\ud\nu(\xi)\ud s\nonumber\\
&\qquad \qquad -\gamma({A(\bar {\bar u}(s))},\bar {\bar u}(s))\ud s
-\frac12\rho (-|\bar {\bar u}(s)|^2 +2\iprod{\bar {\bar u}(s)}{\bar {\bar u}(s)})\ud s\Big]\ud t\nonumber\\
&\qquad \qquad-\frac12 \tilde \bE \int_0^T\varphi(t)(e^{-\rho t}|\bar{\bar u}(t)|-|\bar{\bar u}_0|^2)\ud t.
\end{align}
That is,
\begin{align}
&\hspace{-0.5in}{\limsup_{R\rightarrow \infty}\sum_{i=0}^k  \tilde \bE\int_0^T\varphi(t)\int_0^te^{-\rho s}(f^R_i(\xi_i^R),\xi_i^R) \ud s \ud t}
\nonumber\\&\leq  \tilde \bE\int_0^T\varphi(t)\int_0^te^{-\rho s}\Big[\frac12{\|G(\bar {\bar u }(s))\|^2_{L_2(U_0,H)}}\ud s+\frac12{\int_{E_0}|K(\bar {\bar u}(s),\xi)|^2}\ud \nu(\xi)\ud s\nonumber\\
&-\gamma({A(\bar {\bar u}(s))},\bar {\bar u}(s))\ud s
-\frac12\rho |\bar {\bar u}(s)|^2 \ud s\Big]\ud t -\frac12 \tilde \bE \int_0^T\varphi(t)(e^{-\rho t}|\bar{\bar u}(t)|-|\bar{\bar u}_0(t)|^2)\ud t\nonumber\\
&=  \tilde \bE\int_0^T\varphi(t)\int_0^te^{-\rho s}(Y^{\tilde R},\bar{\bar u})\ud s \ud t.\label{gR2}
\end{align}
The last step follows from similarly applying, {to \eqref{step1_R}}, the argument in \eqref{mult_exp} and the It\^o formula using the function $|\bar{\bar u}|^2$.\\
Thus, as claimed in \eqref{evol_claim1}, we have
\begin{align} 0 \leq \limsup_{R \rightarrow \infty} \sum_{i=0}^k  \tilde \bE\int_0^T\varphi(t)\int_0^te^{-\rho s}(f^R_i(\xi_i^R),\xi_i^R) \ud s \ud t \leq  \tilde \bE\int_0^T\varphi(t)\int_0^te^{-\rho s}(Y^{\tilde R},\bar{\bar u})\ud s \ud t.\end{align}
\textit{Proof of claim (\ref{evol_claim2})}:
Going back to the definition of $f^R_i$ and $\xi^R_i$ in \eqref{gen_def_fR} and \eqref{def_xi} we see that
\begin{align*}
 &\sum_{i=0}^k \tilde \bE\int_0^T\varphi(t)\int_0^te^{-\rho s} |  \iprod{f^R_i(D^i\phi)}{\xi_i^R} - \iprod{f_i(D^i\phi)}{D^i\bar{\bar u}}|\ud s\ud t\\
  &\qquad \leq \sum_{i=0}^k   T|\varphi|_{L^\infty}\tilde \bE \int_0^T|  \iprod{f^R_i(D^i\phi)}{\xi_i^R} - \iprod{f_i(D^i\phi)}{D^i\bar{\bar u}}|\ud s \\
 & \qquad \leq  \sum_{i=0}^k  T|\varphi|_{L^\infty}\tilde \bE \int_0^T|\left(f_i(D^i \phi), \xi_i^R -D^i \bar{\bar u}\right)|+ |\left(f_i^R(D^i \phi)-f_i(D^i \phi),\xi_i^R\right)| \ud s\\&\qquad =:I^b_1+I^b_2.
\end{align*}
{We can see from \eqref{gc}$_2$ that $f_i(D^i\phi) \in L^{p'}([0,T]\times \tilde \Omega, \ud t \otimes \ud \tilde\bP;(L^{p'}(\mathcal{O}))^{d^{i+1}}) $ since $ \phi\in  L^{p}([0,T]\times \tilde \Omega,\ud t\otimes \ud \tilde\bP;X)$.} And thus $I^b_1\rightarrow 0$ as $R \rightarrow \infty$ thanks to \eqref{conv_xir_u}. \\
Next using \eqref{gc} we observe that
\begin{align*}
I^b_2 &\leq T|\varphi|_{L^\infty}\sum_{i=0}^k \tilde \bE \int_0^T\int_{\mathcal{O}} |D^i\phi|_p^{p-1}|\xi_i^R|_p\mathbbm{1}_{\{|D^i\phi|_p>R\}}\ud x\ud s.
\shortintertext{Since $|\xi^R_i|_p<R$ by definition, we obtain the following bounds on the sets ${\{|D^i\phi|_p>R\}}$, }
\lim_{R \rightarrow \infty} I^b_2 &\leq  T|\varphi|_{L^\infty}\limsup_{R \rightarrow \infty}\sum_{i=0}^k\tilde \bE \int_0^T\int_{\mathcal{O}} |D^i\phi|_p^p\mathbbm{1}_{\{|D^i\phi|_p>R\}}\ud x \ud s.
\end{align*}
Observe that $ |D^i\phi|_p^p\mathbbm{1}_{\{|D^i\phi|_p>R\}}$ is dominated by $|D^i\phi|_p^p \, \in L^1( \tilde \Omega\times [0,T] \times \mathcal{O})$ and $\mathbbm{1}_{\{|D^i\phi|_p>R\}} \rightarrow 0$ a.e. on $\tilde \Omega\times [0,T] \times \mathcal{O}$. { Therefore,} using the Lebesgue dominated convergence theorem we obtain $I^b_2 \rightarrow 0$. Hence we have,
\begin{align}
&\limsup_{R \rightarrow \infty}\sum_{i=0}^k \tilde \bE\int_0^T\varphi(t)\int_0^te^{-\rho s} |  \iprod{f^R_i(D^i\phi)}{\xi_i^R} - \iprod{f_i(D^i\phi)}{D^i\bar{\bar u}}|\ud s\ud t=0.
\end{align}
Recall from \eqref{gen_Ff} that since $\phi, \bar{\bar u} \in  L^{{p}{}}([0,T]\times \tilde \Omega,\ud t\otimes \ud \tilde\bP;X)$ 
we have $ \iprod{F(\phi)}{\bar{\bar u}}=\sum_{i=0}^k \iprod{f_i(D^i\phi)}{D^i\bar{\bar u}} $ and thus \eqref{evol_claim2} follows.\\
\textit{Proof of claim (\ref{evol_claim3})}:
	We write
	\begin{align}\label{claimc}
\sum_{i=0}^k  \iprod{f^R_i(\xi_i^R)}{D^i\phi} &=  \iprod{F^R(\bar{\bar u}^R)}{\phi} -I^c_1(s) 
	\end{align}	
	where $I^c_1(s)$ is given by,
	\begin{align*}
\sum_{i=0}^k \int_{\mathcal{O}}\Big(f_i\left({\frac{R D^i \bar{\bar u}^R}{|D^i \bar{\bar u}^R|_p}}\right) \cdot D^i \phi+ \left(D^i \bar{\bar u}^R- {\frac{R D^i \bar{\bar u}^R}{|D^i \bar{\bar u}^R|_p}}\right)  D^2 \mathscr{J}_i\left({\frac{R D^i \bar{\bar u}^R}{|D^i \bar{\bar u}^R|_p}}\right) \cdot D^i \phi \Big)  \mathbbm{1}_{\{|D^i \bar{\bar u}^R|_p>R\}}\ud x. 
\end{align*}
	Using the H{\"o}lder inequality and \eqref{gc} we obtain
	\begin{align*}
	\tilde \bE\int_0^T\varphi(t)\int_0^te^{-\rho s}|I^c_1(s)|\ud s\ud t \leq& T|\varphi|_{L^\infty}\sum_{i=0}^k R^{\frac{p-2}{2}}\Bigg(\tilde \bE\int_0^T \int_{\mathcal{O}}R^{p-2}|D^i\bar{\bar u}^R|_p^2\mathbbm{1}_{\{|D^i\bar{\bar u}^R|_p>R\}}\ud x\ud s\Bigg)^{\frac12}\\
	&\hspace{-0.7in}\Bigg(\tilde \bE\int_0^T \int_{\mathcal{O}}|D^i\phi|_p^p\mathbbm{1}_{\{|D^i\bar{\bar u}^R|_p>R\}}\ud x\ud s\Bigg)^{\frac1p}
\Bigg(\tilde \bE\int_0^T \int_{\mathcal{O}}\mathbbm{1}_{\{|D^i\bar{\bar u}^R|_p>R\}}\ud x\ud s\Bigg)^{\frac{p-2}{2p}}.
	\end{align*}
Using \eqref{apriori_gencase}, \eqref{apriori2} and \eqref{apriori3} we obtain, for some constant $C=C(p,K_2)>0$, that
	\begin{align*}
	\tilde \bE\int_0^T\varphi(t)\int_0^te^{-\rho s}|I^c_1(s)|\ud s\ud t \leq C|\varphi|_{\infty}T\sum_{i=0}^k \Bigg(\tilde \bE\int_0^T \int_{\mathcal{O}}|D^i\phi|_p^p\mathbbm{1}_{\{|D^i\bar{\bar u}^R|_p>R\}}\ud x\ud s\Bigg)^{\frac1p}.
	\end{align*}

Observe that$|D^i\phi|_p^p\mathbbm{1}_{\{|D^i\bar{\bar u}^R|_p>R\}}$ is dominated by $|D^i\phi|_p^p \in L^1(\mathcal{O} \times \tilde \Omega\times[0,T])$. Thanks to \eqref{conv_1r} we can apply the Lebesgue dominated convergence theorem and, thinning the sequence if required, we obtain that
		\begin{align}\label{Ito0}
	\tilde \bE\int_0^T\varphi(t)\int_0^te^{-\rho s}|I^c_1(s)|\ud s\ud t \rightarrow 0 \text{ as }R \rightarrow \infty.
	\end{align}
Hence, as claimed, from \eqref{claimc} and \eqref{Ito0} we obtain
	\begin{align}
\limsup_{R \rightarrow \infty}\sum_{i=0}^k  \tilde \bE\int_0^T\varphi(t)\int_0^te^{-\rho s} \iprod{f^R_i(\xi_i^R)}{D^i\phi} \ud s \ud t= \tilde \bE\int_0^T\varphi(t)\int_0^te^{-\rho s} \iprod{Y^{\tilde{R}}}{\phi}\ud s\ud t. \end{align}
The proof of \eqref{evol_claim4} follows closely the proof of \eqref{evol_claim2} and so we omit the details.
\end{proof}
{We now} observe that because of {the} monotonicity of $f^R$, we have for any $R>1$, {any }$ \phi\in  L^{p}([0,T]\times \tilde \Omega,\ud t\otimes \ud \tilde\bP;X)$ and {any }nonnegative {funtion} $\varphi\in L^\infty([0,T],\ud t;\mathbb{R})$,
\begin{equation}\label{mono_FR}
\sum_{i=0}^k\tilde \bE\left(\int_0^T\varphi(t)
 \left(\int_0^t e^{-\rho s}\iprod{{f}^R_i(\xi_i^R) -{f^R_i}(D^i\phi)}{\xi_i^R-D^i\phi}\ud s\right)\ud t\right) \geq 0.
\end{equation}	
With Lemma \ref{4claims} in hand, we take $\limsup_{R\rightarrow \infty}$ on both sides of \eqref{mono_FR} and apply \eqref{evol_claim1}-\eqref{evol_claim4}.  We find that for all  $ \phi\in  L^{p}([0,T]\times \tilde \Omega,\ud t\otimes \ud \tilde\bP;X)$ and nonnegative $\varphi\in L^\infty([0,T],\ud t;\mathbb{R})$,
	\begin{align*}
	&  \tilde \bE\left(\int_0^T\varphi(t)
	\left(\int_0^te^{-\rho s}\iprod{Y^{\tilde{R}}(s)-{F}(\phi(s))}{\bar{\bar u}(s)-\phi(s)}\ud s\right)\ud t\right) \geq 0.
	\end{align*}
	We use the Minty method {once again and choose} $\phi=\bar {\bar u}\pm \lambda \zeta \vv$ {with} $\lambda > 0$, $\vv\in X$
	and $\zeta\in L^\infty([0,T]\times \tilde \Omega,\ud t\otimes \ud \tilde\bP;\mathbb{R})$. {We then divide} both sides by $\lambda$ {and let} $\lambda\rightarrow 0$. {We obtain}, thanks to {the} hemicontinuity of the operator $F$  
	\begin{equation}
	\tilde \bE\left(\int_0^T\varphi(t)
	\left(\int_0^t\zeta(s)\iprod{Y^{\tilde{R}}(s)-F(\bar{\bar u}(s))}{\vv}\ud s\right)\ud t\right)=0.
	\end{equation}
	Since $\varphi$, $\zeta$ and $\vv$ are arbitrary, we conclude that
	\begin{equation}
	Y^{\tilde{R}}= F(\bar {\bar u})  \quad  \ud   \tilde\bP \otimes \ud t \tn{-a.e.}
	\end{equation}
	
	Thus we have proven {that} $\bar{\bar u}$ solves \eqref{approx_eqn2} with respect to the basis $({\tilde \Omega},\tilde {\mathcal{F}}, (\bar{\bar {\mathcal{F}}}_t)_{t\ge 0},\tilde \bP,\bar {\bar W},\bar {\bar \pi})$ and this concludes the passage to the limit $R \rightarrow \infty$ (with $\tilde R$ fixed). Theorem \ref{gencasetheorem} is thus proven.\end{proof}
	
\section{Passage to the limit $\tilde{R} \rightarrow \infty$}\label{sec:tildeR}
In this section we pass to the limit $\tilde R \rightarrow \infty$ and {finally }show the existence of 
a martingale solution to \eqref{eqn}-\eqref{bc}. { Except for the term $B$, many of the steps will be similar to the passages to the limit $n\rightarrow \infty$, $R \rightarrow \infty$ and we will skip the details.}\\From the previous section we know that $\bar{\bar u}$, now denoted by $\bar{\bar u}^{\tilde{R}}$ to emphasize its dependence on $\tilde R$, solves \eqref{approx_eqn2} 
with respect to the basis $({\tilde \Omega},\tilde {\mathcal{F}}, {(\bar{\bar {\mathcal{F}}}^{\tilde R}_t)_{t\ge 0}},\tilde \bP,{ \bar {\bar W}^{\tilde R},{\bar {\bar \pi}}^{\tilde R}} )$.\\ 
We can see {that }there exists a constant $C(\gamma,\rho,c_1,c_2)> 0$, independent of $\tilde{R}$, such that
\begin{align}
\tilde \bE\sup_{s\in[0,T]}
|\bar{\bar u}^{\tilde{R}}(s)|^{2}+ \gamma
\tilde \bE\int_0^T \|{\bar{\bar u}^{\tilde{R}}(s)}\|^2\ud s + \tilde \bE\int_0^T \|{\bar{\bar u}^{\tilde{R}}(s)}\|_X^p\ud s\le C{\bE}|{ u}_0|^2=:K_3.\label{energytilR}
\end{align}
We will omit the details of the proof of the above energy estimate since it {is} derived {exactly} as in Lemma \ref{uniformboundL2} using \eqref{F3} instead of \eqref{FR3}.
Next we will prove the following bounds in fractional Sobolev spaces.
\begin{prop}\label{utilr_boundinH}
	Let $\bar {\bar{u}}^{\tilde R}$  solve the SPDE \eqref{approx_eqn2} with respect to the basis  $({\tilde \Omega},\tilde {\mathcal{F}}, {(\bar{\bar {\mathcal{F}}}^{\tilde R}_t)_{t\ge 0}},\tilde \bP,{ \bar {\bar W}^{\tilde R},{\bar {\bar \pi}}^{\tilde R}} )$
	. Since $p>2$ we can choose $1<m <\min\{\frac{p}{p-1},\frac{4p}{2+3p}\}<\frac43$.
	Then for every $\alpha\in (0,\frac 12)$ 
	there exists a constant $C=C(\alpha,\gamma,\rho,T,m,p)>0$, independent of $\tilde{R}$, such that
	\begin{equation}
		\tilde \bE\|{{\bar{ \bar u}}^{\tilde{R}}}\|^{m}_{W^{\alpha,m}(0,T;V'+X')}\le C.\label{frac_uR}
		\end{equation}
	Furthermore, due to Lemma \ref{Sobolev1} {  and an argument like in Proposition \ref{tightnessinV},} the estimate \eqref{frac_uR} implies that the laws of $\bar {\bar u}^{\tilde R}$ are tight in $L^2(0,T;H)$.
	\end{prop}	
	\begin{proof}		
		The proof follows {that} of Proposition \ref{ur_boundinH} and Proposition \ref{boundinW}  with slight modifications to the estimates \eqref{frac_B} and \eqref{frac_F} as described below. 
		First for $\alpha<1$ we have, {as for \eqref{gn}, \eqref{gR}}:		
		\begin{equation} 
		\begin{aligned}
		~\|{g^{\tilde R}}\|^m_{W^{\alpha, m}(0,T;{V'+X'})}&:=  \|{\bar{ \bar u}}_0^{\tilde{R}} -\int_0^{\cdot}\left( \gamma A {{\bar{ \bar u}}^{\tilde{R}}}(s) + B_{\tilde R}(
			{{\bar{ \bar u}}^{\tilde{R}}}(s)) +F({{\bar{ \bar u}}^{\tilde{R}}}(s))\right) \ud s\|^{m}_{W^{\alpha,m}(0,T;V'+X')} \nonumber \\
		\lesssim & ~ \Hnorm{{{\bar{ \bar u}}_0^{\tilde{R}}}}^{m}+ \int_0^T \gamma\Hnorm{A{{\bar{ \bar u}}^{\tilde{R}}}(s)}^{m}_{{V'}} \ud s + \int_0^T \Hnorm{B_{\tilde{R}}(
			{{\bar{ \bar u}}^{\tilde{R}}}(s))}_{{V'}}^{m} \ud s + \int_0^T \Hnorm{F({{\bar{ \bar u}}^{\tilde{R}}}(s))}^{m}_{{X'}} \ud s ,
		\end{aligned}
		\end{equation}
		where the hidden constants depend on $T,p,m$ but not on $\tilde R$.\\
Observe that since $m<2$, the expected values of the first two terms {above} are treated like in \eqref{fracAR}. 
Next, using \eqref{eqn:B_bound3} and observing that{  $\frac{m}{2p} +\frac{3m}4+(1-\frac{m(2+3p)}{4p})=1$ }we apply the H\"older inequality {with these exponents} to obtain the following estimates in place of \eqref{frac_B},
		\begin{align}
	\tilde \bE\int_0^T|B_{\tilde{R}}({\bar{\bar u}^{\tilde{R}}(s)})|_{V'}^{m} \ud s 
	&\leq \tilde \bE \int_0^T |{\bar{\bar u}^{\tilde{R}}(s)}|^{\frac{m}2}  \|{\bar{\bar u}^{\tilde{R}}(s)}\|^{\frac{3m}2} \ud s \nonumber\\
		&\lesssim T^{(1-\frac{m(2+3p)}{4p})}\left(\tilde \bE \int_0^T |{\bar{\bar u}^{\tilde{R}}(s)}|^{p} \ud s\right)^{\frac{m}{2p}}\left(\tilde \bE \int_0^T \|{\bar{\bar u}^{\tilde{R}}(s)}\|^{2} \ud s\right)^{\frac{3m}{4}}. \nonumber
\shortintertext{Since the embedding $X \subset H$ is continuous {we can, thanks to \eqref{energytilR}, bound this, {independently} of $\tilde R$ by}}
& T^{(1-\frac{m(2+3p)}{4p})}\left(\tilde \bE \int_0^T \|{\bar{\bar u}^{\tilde{R}}(s)}\|_X^{p} \ud s\right)^{\frac{m}{2p}}\left(\tilde \bE \int_0^T \|{\bar{\bar u}^{\tilde{R}}(s)}\|^{2} \ud s\right)^{\frac{3m}{4}} \nonumber\\
&\leq C(T,m,p,K_3). \label{frac_BtilR}
	\end{align}
Using the condition \eqref{F4} and the fact that {{}$m < \frac{p}{p-1}$},  we {{}obtain} the following bounds {{}where}  $C=C(T,m,p,k_3)>0$ {is} independent of $\tilde{R}$
\begin{align*}
\tilde \bE\int_0^T
|F(\bar{\bar u}^{\tilde{R}}(s))|^m_{X'}\ud s &\leq T^{(1-{\frac{m(p-1)}{p}})}  \left(\tilde \bE\int_0^T
|F(\bar{\bar u}^{\tilde{R}}(s))|^{\frac{p}{p-1}}_{X'}\ud s\right)^{\frac{m(p-1)}{p}}\end{align*}\begin{align}
&\lesssim C(T,m,p) \left(\tilde \bE\int_0^T\|{{\bar{\bar u}^{\tilde{R}}(s)}}\|^p_X \ud s\right)^{\frac{m(p-1)}{p}} \nonumber\\&\le C.\label{frac_FtilR}
\end{align}
Now observe that Proposition \ref{noisefraction} holds for $\beta=m$ and with $P_n(G(u^n))$, $P_n(K(u^n))$ replaced by $G(\bar{\bar u}^{\tilde R})$, $K(\bar{\bar u}^{\tilde R})$ respectively. Combining with \eqref{frac_BtilR} and \eqref{frac_FtilR} we obtain \eqref{frac_uR}.
\end{proof}
Next define,
\begin{align}
\mu_u^{\tilde{R}}(\cdot)&=\tilde{\bP}(\bar{\bar u}^{\tilde{R}}\in\cdot) \in Pr\left(L^2(0,T;H)\cap
\mathcal{D}([0,T];V'+X')\right),
\end{align}
where $Pr(\mathcal{S})$ denotes, {as before}, the set of probability measures  on a metric space $\mathcal{S}$.
\begin{prop}\label{tightnessinchi2}
Assume that $p>2$ and let $\bar {\bar u}^{\tilde{R}}$
solve the SPDE \eqref{approx_eqn2} with respect to the basis\\$({\tilde \Omega},\tilde {\mathcal{F}}, {(\bar{\bar {\mathcal{F}}}^{\tilde R}_t)_{t\ge 0}},\tilde \bP,{ \bar {\bar W}^{\tilde R},{\bar {\bar \pi}}^{\tilde R}} )$. 
Then the laws $\{\mu_u^{\tilde{R}}\}_{\tilde{R}>1}$ of $\{\bar{\bar u}^{\tilde{R}}\}_{\tilde{R}>1}$ are tight in {the space} $\sX_1$, defined in \eqref{chi1}.
\end{prop}
\begin{proof}	
A part of the proof is to show that {a suitable} Aldous condition \eqref{Aldous} is satisfied by $\bar{\bar u}^{\tilde R}$. \\
With that in mind, let $(\tau_{\tilde R})_{\tilde{R}>1}$ be a sequence of stopping times such that $0\le \tau_{\tilde R} \le T$. We 
 define {the} terms $\mathcal{M}_{1,2,3,4,5}$ as in \eqref{Mn} in Proposition \ref{tightnessinD} and \eqref{MR} in Proposition \ref{tightnessinchi1}. Thanks to the estimates \eqref{energytilR}, the terms $\mathcal{M}_{1,2,4,5}$ are treated identically as in the proof of Proposition \ref{tightnessinD}. {We emphasize here that the bounds found in \eqref{M2} are independent of $\tilde R$ and so the treatment of the term $\mathcal{M}_2$ here is carried out as in \eqref{M2}.} \\
 Thus we discuss here the term $\mathcal{M}_3$ that is treated differently. 
 
Observe that using \eqref{F4} we obtain for some $C(T,p,K_3)>0$ independent of $\tilde{R}$,
\begin{align}
\tilde \bE\left[\norm{\mathcal{M}_3(\tau_{\tilde R} +\delta)-\mathcal{M}_3(\tau_{\tilde R})}_{V'+X'}\right]	& \lesssim \tilde \bE\int_{\tau_{\tilde R}}^{\tau_{\tilde R} +\delta}|F(\bar{\bar u}^{\tilde{R}}(s))|_{X'}\ud s\nonumber\\
& \lesssim \delta^{\frac1p}\left[\tilde \bE\int_{\tau_{\tilde R}}^{\tau_{\tilde R}+\delta}|F(\bar{\bar u}^{\tilde{R}}(s))|^{\frac{p}{p-1}}_{X'}\ud s\right]^{\frac{p-1}{p}}\nonumber\\
& \lesssim \delta^{\frac1p}\left[\tilde \bE\int_{\tau_{\tilde R}}^{\tau_{\tilde R}+\delta}\|\bar{\bar u}^{\tilde{R}}(s)\|^{p}_{X}\ud s\right]^{\frac{p-1}{p}}\nonumber\\
& \le C\delta^{\frac1p} .
\end{align}
Hence the condition \eqref{Aldous}  is satisfied by $\bar{\bar u}^{\tilde R}$  with $\Xi=V'+X'$, $\alpha=1$ and $\epsilon=\min\{\frac1p,\frac14\}$ which in turn implies that, $\mu_u^{\tilde R}$, the laws of $\{\bar {\bar u}^{\tilde R}\}_{\tilde R>1}$ are tight in the space $\mathcal{D}([0,T];V'+X')$ {endowed} with the Skorohod topology.  Hence, finally applying the procedure in Proposition \ref{tightnessinchi} together with Proposition \ref{utilr_boundinH}, we finish the proof of Proposition \ref{tightnessinchi2}.
\end{proof}
Thus, using Prohorov's theorem we obtain the existence of a measure $\mu^\infty_2$ on $\Upsilon_1= H \times \sX_1 \times C([0,T];U) \times \mathcal{N}^{\# *}_{[0,\infty)\times E}$  such that the sequence $\mu^{\tilde R}$, of the joint probability law of  $(\bar{\bar u}_0^{\tilde R},
\bar{{\bar u}}^{\tilde R},\bar{{\bar W}}^{\tilde{R}},\bar {\bar \pi}^{\tilde{R}})$ indexed by some infinite set $\Lambda' \subseteq \mathbb{N}$, converges weakly to $\mu^\infty_2$ as $\tilde R \rightarrow \infty$. Next we use Theorem \ref{skorohodtheorem} again to obtain the following result.
\begin{thm}\label{skorohod3}For $\tilde R \in \Lambda'$ there exist ${{\Upsilon}_1}$-valued random variables $(\tilde{ u}_0^{\tilde R},
\tilde{{ u}}^{\tilde R},\tilde{{ W}}^{\tilde{R}},\tilde{\pi}^{\tilde{R}})$
and\\ $(\tilde{ u}_0,\tilde{ u} ,\tilde{ W}, \tilde{ \pi})$ such that
\begin{myenum}
	\item The probability laws of $(\tilde{ u}_0^{\tilde R},
	\tilde{{ u}}^{\tilde R},\tilde{{ W}}^{\tilde{R}},\tilde{\pi}^{\tilde{R}})$
	and $(\tilde{ u}_0,\tilde{ u} ,\tilde{ W}, \tilde{ \pi})$ are $\mu^{\tilde{R}}$ and $\mu_2^\infty$, respectively,
	\item $(\tilde { W}^{\tilde{ R}},\tilde{ \pi}^{\tilde{R}})=(\tilde{{W}},\tilde{{\pi}}) $ everywhere on $\tilde{\Omega}$,
	\item $(\tilde{ u}_0^{\tilde{R}},
	\tilde{{ u}}^{\tilde{R}},\tilde{{ W}}^{\tilde{R}},\tilde{\pi}^{\tilde R})$ converges almost surely to $(\tilde{ u}_0,\tilde{ u} ,\tilde{ W}, \tilde{ \pi})$ as $\tilde R \rightarrow \infty$ in the topology of ${\Upsilon}_1$:
	\begin{align}
	&\tilde{{ u}}^{\tilde{R}}_0\rightarrow \tilde{{u}}_0 \text{ in } H, \tilde{\bP}\tn{-a.s.},\label{uotilrk}	\\
	&\tilde{{ u}}^{\tilde{R}}\rightarrow
	\tilde{ u} \text { in } L^2(0,T;H)\cap \mathcal{D}([0,T];V'+X'),
	\tilde \bP\tn{-a.s.},\label{utilr_convH}
	\end{align}
	\item  $\tilde{ \pi} $, is a time-homogeneous Poisson random measure over $(\tilde{ \Omega}, {\tilde {\mathcal{F}}}, \tilde{ \bP})$ with intensity measure $\ud t \otimes \ud \nu$.
	\item $\tilde u^{\tilde R}_0=\tilde u^{\tilde R}(0)$ $\tilde \bP$-a.s. and
	$(\tilde{{ u}}^{\tilde{R}},\tilde { W}^{},\hat{\tilde{ \pi}})$ satisfies
	the following equation, $\tilde \bP$- a.s. for every $t\in[0,T]$
	\begin{equation}\label{eq_tilR}
	\begin{split}
	&\tilde{{ u}}^{\tilde{R}}(t)+\int_0^t\left[\gamma A(\tilde{{ u}}^{\tilde{R}}(s))+ \theta_{\tilde{R}}(|\tilde u^{\tilde R}|)B(\tilde{{ u}}^{\tilde{R}}(s),\tilde{{ u}}^{\tilde{R}}(s))+ F(\tilde{{ u}}^{\tilde{R}}(s))\right]\ud s
	\\
	&\qquad \qquad= \tilde{{ u}}^{\tilde{R}}_0+\int_0^tG(\tilde{{ u}}^{\tilde{R}}(s))\ud \tilde{ W}(s)+\int_{(0,t]}\int_{E_0}K(\tilde{{ u}}^{\tilde{R}}(s),\xi)\ud \hat{\tilde{\pi}}(s,\xi),
	\end{split}\end{equation}
	with respect to {the filtration} $(\tilde{\mathcal{F}}^{\tilde R}_t)_{t\geq 0}$ defined similarly to \eqref{Ft}. 
	{(See also Remark \ref{diffrv}.)}
\end{myenum}
\end{thm}
{Henceforth we will restrict our analysis to $\tilde R \in \Lambda'$.} Comparing equations \eqref{approx_eqn2} and \eqref{eq_tilR} we can observe that $\tilde{ u}^{\tilde R}$ satisfies the energy estimates \eqref{energytilR} as well.
This further implies that $\tilde { u} \in L^\infty([0,T];L^2(\tilde {\Omega},H)) \cap L^2([0,T]\times  {\tilde\Omega},\ud t \otimes \ud \tilde {\bP} ;V)\cap L^p([0,T]\times  {\tilde\Omega},\ud t \otimes  {\ud \tilde\bP};X)$ and that there exists a subsequence $\tilde R \rightarrow \infty$ such that,
\begin{align}\label{conv_utilr}
\tilde{  u}^{\tilde R} &\rightharpoonup \tilde { u} \text{ weakly in }L^2([0,T]\times \tilde {\Omega},\ud t \otimes \ud \tilde {\bP};V)\cap L^p([0,T]\times \tilde {\Omega},\ud t \otimes \ud \tilde {\bP};X),\\
\tilde { u}^{\tilde R} &\rightharpoonup \tilde{ u} \text{  weak$^*$ in } L^\infty([0,T];L^2(\tilde {\Omega},H)).
\end{align}
Observe that due to \eqref{F4} we also have, for a subsequence $\tilde R \rightarrow \infty$, the existence of $Y \in  L^{\frac{p}{p-1}}([0,T]\times \tilde {\Omega},\ud t \otimes \ud \tilde {\bP};X')$ such that
\begin{align}\label{conv_Ftilr}
F(\tilde{  u}^{\tilde R}) &\rightharpoonup Y \text{ weakly in } L^{\frac{p}{p-1}}([0,T]\times \tilde {\Omega},\ud t \otimes \ud \tilde {\bP};X').
\end{align}
Now we are in position to prove one of our main results, namely Theorem \ref{existence_martingale}. We define $(\tilde{\mathcal{F}})_{t\geq 0}$ as in \eqref{Ft} and denote $\tilde {\mathscr{S}}:= ({\tilde{ \Omega}},{\tilde{ {\mathcal{F}}}}, (\tilde{ {\mathcal{F}}}_t)_{t\ge 0},\tilde{ \bP},\tilde W, \tilde{\pi})$. In the remainder of the section we will show that if $\bE[|u_0|^2] <\infty$ and $p>2$, then $(\tilde {\mathscr{S}},\tilde u)$ is a martingale solution to \eqref{eqn}-\eqref{bc} (in the sense of Definition \ref{definition1}) with initial condition $\tilde{ {u}}_0$, that has the same law as $u_0$. 
We will do this in two steps.\\	
	\noindent \textbf{\textit{\underline{Step 1}}}:	
	We will first show that $\tilde u $ solves the following equation
	\begin{equation}\label{step1tilR}
	\begin{split}
	&\tilde u(t)+\int_0^t[\gamma A(\tilde u(s))+B(\tilde u(s))+ Y(s)]\ud s\\&\hspace{2in}=\tilde u_0+\int_0^tG(\tilde u(s))\ud\tilde{W}(s)+\int_{(0,t]}\int_{E_0}K(\tilde u(s-),\xi)\ud\hat{\tilde{\pi}}(s,\xi),\end{split}
	\end{equation}
	where $Y$ was defined in \eqref{conv_Ftilr}.
	
For that purpose we first observe that the following convergence results hold for {$\phi \in D(A)$}, as $\tilde R \rightarrow \infty$:		
	\begin{align}\label{conv_AtilR}
&	 \gamma \int_0^t \iprod{A\tilde{ u}^{\tilde R}(s)}{\phi} \ud s \rightarrow \gamma \int_0^t \iprod{A\tilde{ u}^{}(s)}{\phi} \ud s \quad \text { in }
	L^1(\tilde\Omega\times [0,T]),\\
	\label{conv_GtilR}
&	\int_0^tG(\tilde{ u}^{\tilde R}(s))\ud\tilde{W}(s)\rightarrow \int_0^tG(\tilde{ u}^{}(s))\ud\tilde{W}(s) \quad \text{ in } L^1(\tilde \Omega;L^1(0,T;H)),\\
\label{conv_KtilR}
&	\int_{(0,t]} \int_{E_0}K(\tilde{u}^{\tilde R}(s-),\xi)\ud\hat{\tilde{\pi}}(s,\xi)
	\rightarrow \int_{(0,t]} \int_{E_0} {K(\tilde{ u}^{}(s-),\xi)}{}\ud\hat{\bar{\pi}}(s,\xi) \text{ in  $L^1(\tilde\Omega; L^1(0,T;H)).$}
	\end{align}
We obtain the above convergence results {as in} the proofs of \eqref{conv_A}, Lemma \ref{conv_G} and Lemma \ref{conv_K} respectively.
{In fact the above convergence result \eqref{conv_AtilR} holds in $L^2(\tilde {\Omega} \times [0,T])$ and \eqref{conv_GtilR}-\eqref{conv_KtilR} hold in  $L^2(\tilde\Omega; L^2(0,T;H))$.
	 This is due to the availability of the $L^p(\tilde \Omega;L^p(0,T;X))$ bounds, for $p>2$, found in \eqref{energytilR} which can be used to upgrade the estimates \eqref{G2} and \eqref{eqn:estimate2}.}\\  

Next we will prove that,
\begin{align}\label{conv_tilu}
	\tilde{u}^{\tilde R} \rightarrow \tilde{u} \text{ in } L^2(\tilde{\Omega};L^2(0,T;H)) \quad \text{as $\tilde R \rightarrow \infty$}.
\end{align}
Since $p>2$, using \eqref{energytilR} we observe that for some $C=C(T,p,K_3)>0$ independent of $\tilde{R}$, we have
\begin{align*}
\sup_{\tilde R \geq 1}\tilde{\bE}|\tilde{u}^{\tilde R}|^{p}_{L^2(0,T;H)} &= \sup_{\tilde R \geq 1}\tilde{\bE} (\int_0^T|\tilde{u}^{\tilde R}|^{2} \ud s)^{\frac{p}{2}}\lesssim \sup_{\tilde R \geq 1}\tilde{\bE} \int_0^T|\tilde{u}^{\tilde R}|^{p} \ud s \lesssim \sup_{\tilde R \geq 1}\tilde{\bE} \int_0^T\|\tilde{u}^{\tilde R}\|_X^{p}\ud s \lesssim C. \end{align*}
Combining this result with \eqref{utilr_convH} and  applying the Vitali convergence theorem, we infer \eqref{conv_tilu}. Next we will prove that as $\tilde R \rightarrow \infty$,
	\begin{equation}\label{conv_BtilR}
	\int_0^t\iprod{B_{\tilde{R}}(\tilde{u}^{\tilde R},\tilde{u}^{\tilde R})}{\phi}\ud s
	\rightarrow \int_0^t\iprod{B(\tilde u,\tilde u)}{\phi}\ud s \quad \text{ in $L^1(\tilde{\Omega} \times [0,T])$.}
	\end{equation}
We begin the proof of \eqref{conv_BtilR} by observing that,
\begin{align*}
\left|\int_0^t(B_{\tilde R}(\tilde{u}^{\tilde R},\tilde{u}^{\tilde R})-B(\tilde{u}^{},\tilde{u}^{}), \phi) \ud s\right|& \leq \int_0^t\Big[|\theta_{\tilde R}(|\tilde u^{\tilde{R}}|)\left(B(\tilde{u}^{\tilde R}-\tilde u,\tilde{u}^{\tilde R})+B(\tilde{u}^{},\tilde u^{\tilde R}-\tilde{u}^{}), \phi\right)|\\
&\hspace{1in}+| \left(1-\theta_{\tilde R}(|\tilde u^{\tilde R}|) \right)(B(\tilde u,\tilde u),\phi)| \Big] \ud s\\
& :=I_1^{\tilde R}(t) + I_2^{\tilde R}(t).
\end{align*}
We use \eqref{eqn:B_bound2}	to treat $I_1^{\tilde R}$ as follows
\begin{align}
I_1^{\tilde R}(t) &\leq \int_0^t |\tilde u^{\tilde R}-\tilde u^{}|\left(\|\tilde u^{\tilde R}\|+\|\tilde u^{}\|\right)\|\phi\|_{D(A)}\ud s,\nonumber
\end{align}{which implies {that} for some $C=C(T,p)>0$ independent of $\tilde{R}$,}
\begin{align}
 \tilde{\bE}\int_0^T I^{\tilde R}_1 \ud t& \leq C\|\phi\|_{D(A)}\left(\tilde{\bE} \int_0^T|\tilde u^{\tilde R}-\tilde u^{}|^2\ud s\right)^{\frac12}\left(\tilde{\bE} \int_0^T\|\tilde u^{\tilde R}\|^2+\|\tilde u^{}\|^2 \ud s\right)^{\frac12}\label{pgeqd2}.\end{align}
Thanks to \eqref{energytilR} and \eqref{conv_tilu} we then obtain,
 \begin{align*}
  \tilde{\bE}\int_0^T I^{\tilde R}_1 \ud t& \leq C\|\phi\|_{D(A)}\left(\tilde{\bE}\int_0^T|\tilde u^{\tilde R}-\tilde u^{}|^2 \ud s \right)^{\frac12} \rightarrow 0 \text{ as } \tilde R \rightarrow \infty .
\end{align*}
	Thus thinning the sequence {{}$\tilde R$} if required, {we see that} $I^{\tilde R}_1 \rightarrow 0$ a.e. on $\tilde{\Omega} \times [0,T]$ as $\tilde R \rightarrow \infty$.

{We invoke \eqref{conv_tilu} to obtain that, for a subsequence, $\tilde u^{\tilde R} \rightarrow \tilde u$ a.e. in $\Omega \times [0,T]$. Hence,
\begin{align*}|\theta_{\tilde{R}}(|\tilde u^{\tilde R}|)-1| &\leq |\theta_{\tilde{R}}(|\tilde u^{\tilde R}|)-\theta_{\tilde{R}}(|\tilde u^{}|)|+ |\theta_{\tilde{R}}(|\tilde u^{}|)-1|\\
&\leq||\tilde u^{\tilde R}|-|\tilde u|| + |\theta_{\tilde{R}}(|\tilde u^{}|)-1|\\
& \leq |\tilde{u}^{\tilde R}- \tilde u| + \mathbbm{1}_{\{|\tilde u|>2\tilde R\}}
\rightarrow 0 \qquad  \text{ a.e. on } \Omega \times[0,T]. \end{align*}}
This is key in treating the term $I^{\tilde R}_2$. Observe that \eqref{eqn:B_bound2} gives,
\begin{align*}
I^{\tilde R}_2 &\leq \int_0^t \left|1-\theta_{\tilde R}(|\tilde u^{\tilde R}|) \right||(B(\tilde u,\tilde u),\phi)| \ud s
\leq \|\phi\|_{D(A)}\int_0^t \left|1-\theta_{\tilde R}(|\tilde u^{\tilde R}|) \right| |\tilde u|\|\tilde u\| \ud s.
\end{align*}The continuous embedding $X \subset H$ thus implies that for some $C=C(T,p)>0$\begin{align*}
\tilde{\bE} \int_0^T I^{\tilde R}_2 \ud t 
&\lesssim \|\phi\|_{D(A)}\left(\tilde{\bE}\int_0^T  \left|1-\theta_{\tilde R}(|\tilde u^{\tilde R}|) \right|^{\frac{2p}{p-2}}\ud s\right)^{\frac{p-2}{2p}} \left(\tilde{\bE} \int_0^T\|\tilde u\|_X^p \ud s\right)^{\frac{1}{p}}\left(\tilde{\bE} \int_0^T\|\tilde u\|^2 \ud s\right)^{\frac{1}{2}}\\
&\leq C\|\phi\|_{D(A)}\left(\tilde{\bE}\int_0^T \left|1-\theta_{\tilde R}(|\tilde u^{\tilde R}|) \right|^{\frac{2p}{p-2}}\ud s\right)^{\frac{p-2}{2p}}\\
&\rightarrow 0 \text{ as } \tilde R \rightarrow \infty \text{ using  the Lebesgue dominated convergence theorem}.\end{align*}
This leads to the convergence of $I^{\tilde R}_2 \rightarrow 0$ a.e. on $\tilde{\Omega} \times [0,T]$ for a subsequence $\tilde R \rightarrow \infty$.
	Hence we have shown that
	\begin{align}\label{b1}
	\int_0^t(B_{\tilde R}(\tilde{u}^{\tilde R}(s),\tilde{u}^{\tilde R}(s)) \ud s \rightarrow \int_0^t (B(\tilde{u}^{}(s),\tilde{u}^{}(s)), \phi) \ud s \qquad  \text{a.e. on $\tilde{\Omega} \times [0,T]$ as $\tilde R \rightarrow \infty$.}
	\end{align}
	Next, 
	we choose $1<q<\frac{2p}{2+p}<2$. Using \eqref{eqn:B_bound3} {we obtain that} for some $C=C(\phi,T,p,q,K_3)>0$ independent of $\tilde{R}$
	\begin{align}
	\sup_{\tilde R \geq 1}
	\tilde{\bE}\int_0^T\bigg|\int_0^t\iprod{B_{\tilde{R}}(\tilde{u}^{\tilde R},\tilde{u}^{\tilde R})}{\phi}\ud s\bigg|^{q} \ud t
	&= 	\tilde{\bE}\int_0^T\bigg|\int_0^t \theta_{\tilde{R}}(|u^{\tilde R}|)\iprod{B(\tilde{u}^{\tilde R},\tilde{u}^{\tilde R})}{\phi}\ud s\bigg|^{q} \ud t \nonumber\\
		&\lesssim T\|\phi\|_{D(A)}^q	\tilde{\bE}\left(\int_0^T|\tilde u^{\tilde R}|^{}\|\tilde u^{\tilde R}\|^{}\ud s\right)^{q} \nonumber\\
	&\hspace{-0.2in}\lesssim T^{\frac{2p-q(p+2)}{2p}+1} \|\phi\|_{D(A)}^q \left(	\tilde{\bE}\int_0^T\|\tilde u^{\tilde R}\|^{p}_{X} \ud s\right)^{\frac{q}{p}}\left(	\tilde{\bE}\int_0^T\|\tilde u^{\tilde R}\|^{2} \ud s\right)^{\frac{q}2} \nonumber\\
	&\le C(T,\phi,p,q) \label{b2} \quad \text{thanks to \eqref{energytilR}.}
	\end{align}
 Thus using the Vitali convergence theorem with \eqref{b1} and \eqref{b2} we obtain \eqref{conv_BtilR}.\\
	Accumulating all convergence results \eqref{conv_AtilR}, \eqref{conv_GtilR}, \eqref{conv_KtilR} and \eqref{conv_BtilR}, along with \eqref{uotilrk}, \eqref{utilr_convH} and \eqref{conv_Ftilr}, we see that for any $\phi\in D(A)$ and {for any} measurable set $\mathcal{S}\subset \tilde \Omega\times[0,T]$, we have
	\begin{multline}\label{convergenceresult_tilR}
	\tilde{\bE}\int_0^T\chi_{\mathcal{S}}
	\bigg[\iprod{\tilde u(s)}{\phi}\ud s
	+\int_0^t \gamma \iprod{A\tilde u(s)}{\phi}\ud s
	+\int_0^t\iprod{B(\tilde u(s),\tilde u(s)}{\phi}\ud s+\int_0^t
	\iprod{ Y}{\phi}\ud s\bigg]\ud t
	\\=\tilde{\bE}\int_0^T\chi_{\mathcal{S}}
	\bigg[\iprod{\tilde u_0}{\phi}\ud s+\int_0^t\iprod{G(\tilde u(s))}{\phi}\ud \tilde W(s)+ \int_{(0,t]}\int_{E_0}\iprod{K(\tilde u(s-),\xi)}{\phi}\ud\hat{\tilde{\pi}}(s,\xi)\bigg]\ud t.
	\end{multline} 
		This finishes the proof of \eqref{step1tilR}.\\		
	\noindent \textbf{\textit{\underline{Step 2}}}:	It remains to verify that
	\begin{align} \label{fy}
	&Y=F(\tilde{u}),\,\,\, \ud \tilde{\bP}\otimes \ud t \tn{-a.e.}
	\end{align}	
	We will apply the Minty method one last time here to prove \eqref{fy}.\\	
Using the exact argument from Step 2 in the proof of Theorem \ref{prop:passage_to_limit} that leads to inequality \eqref{minty}, {we obtain} for any $\varphi \in L^\infty([0,T],\ud t;\mathbb{R})$ such that $\varphi \geq 0$, and any  $ \phi\in  L^{p}([0,T]\times \tilde{ \Omega},\ud t\otimes  \ud\tilde{ \bP};X)$
\begin{align}\label{minty2}
\tilde \bE\int_0^T\varphi(t)\bigg(\int_0^t
2\gamma\iprod{A\tilde u(s)-A\phi(s)}{\tilde u(s)-\phi(s)}\ud s+
2\iprod{F(\phi(s))-Y}{\phi(s)-\tilde u(s)}\ud s\\
+ \rho|\tilde u(s)-\phi(s)|^2 \ud s \bigg)\ud t\geq 0.\nonumber
\end{align}
As earlier, we choose $\phi= {\tilde u}\pm \lambda \zeta \vv$ for $\lambda \ge 0$, $\vv\in X$
	and $\zeta\in L^\infty([0,T]\times \tilde{ \Omega},\ud t\otimes \ud \tilde{ \bP};\mathbb{R})$ and then divide both sides by $\lambda$. Letting $\lambda\rightarrow 0$ we obtain 
	\begin{equation}
	\tilde{\bE}\left(\int_0^T\varphi(t)
	\left(\int_0^t\zeta(s)\iprod{F(\tilde u(s))-Y^{}}{\vv}\ud s\right)\ud t\right)=0.
	\end{equation}
	Since $\varphi$, $\zeta$ and $\vv$ are all arbitrary, we conclude that
	\begin{equation}
	Y= F(\tilde { u})  \quad \ud\tilde{ \bP}\otimes \ud t \tn{-a.e.}
	\end{equation}
Thus we have proven that for the stochastic basis $(\tilde {\Omega}, \tilde{\mathcal{F}}, (\tilde{\mathcal{F}}_t)_{t\ge 0}, \tilde{\bP}, \tilde W, \tilde{\pi})$, 
	\begin{equation}
	\begin{split}
		&\tilde u(t)+\int_0^t[\gamma A(\tilde u(s))+B(\tilde u(s),\tilde u(s))+ F(\tilde u(s))]\ud s\\&\hspace{1.7in}=\tilde u_0+\int_0^tG(\tilde u(s))\ud\tilde{W}(s)+\int_{(0,t]} \int_{E_0}K(\tilde u(s-),\xi)\ud\hat{\tilde{\pi}}(s,\xi),\end{split}
\end{equation}
holds 
$\tilde{\bP}$-a.s. for all $t \in [0,T]$, where $\tilde{u}_0$ has the same law as $u_0$.
That is, we have finished the proof of Theorem \ref{existence_martingale}, namely the proof of the existence of a martingale solution $\tilde u$ to \eqref{eqn}-\eqref{bc}.
\begin{rem}
	 As mentioned {in an earlier remark following the equations \eqref{conv_AtilR}-\eqref{conv_KtilR}}, we have better convergence results \eqref{conv_GtilR}-\eqref{conv_KtilR} available in the space $L^2(\tilde \Omega;L^2(0,T;H))$. Thus we do not need to use an argument like in \eqref{mult_exp} and we can directly pass to the limit $\tilde R\rightarrow \infty$ after applying the It\^o formula using the function $|\tilde u|^2$ to the process in \eqref{eq_tilR}.
\end{rem}

\section{Global pathwise solution of \eqref{eqn}-\eqref{bc}}\label{global}
In this section we show that there exists a unique global pathwise solution to \eqref{eqn}-\eqref{bc}, i.e. {we} show the existence of $u$ such that \eqref{eqn2:pathwise} is satisfied and thus establish Theorem \ref{existence}.\\
{As explained below this proof relies on an application of a result of Gy\"ongy and Krylov \cite{GyKr} which is the infinite dimensional version of the result of Watanabe and Yamada \cite{YW} and for that purpose we start by proving the uniqueness of pathwise solutions to \eqref{eqn}-\eqref{bc}}.
\subsection{Pathwise uniqueness}\label{uniqueness}

In this section we will prove the uniqueness of martingale solutions, existence of which has been established in the previous section.

\begin{prop}\label{pathwise}
	Let $u_1$ and $u_2$
	be two martingale solutions of \eqref{eqn}-\eqref{bc} relative to the same stochastic basis $(\Omega,\mathcal{F},(\mathcal{F}_t)_{t\ge 0},\bP,W,\pi)$. Then $u_1,u_2$ are indistinguishable, that is if $u_1(0)=u_2(0)$ {a.s. then} we have
	\begin{equation}\label{indist}
	\bP[u_1(t)=u_2(t), \forall t \geq 0]=1.
	\end{equation}	
\end{prop}
\begin{proof}
	Letting $u=u_1-u_2$ {we write the equations \eqref{eqn} satisfied by $u_1$ and $u_2$ and obtain by difference the equation satisfied by $u$. We then obtain, by
	 applying the It\^{o} formula using the function $|u|^2$ to the difference equation:}
	\begin{align}
	\nonumber |u(t)|^2	 +2\int_0^t &(\gamma\iprod{Au(s)}{u(s)}+\iprod{F(u_1(s))-F(u_2(s))}{u(s)} + 2\iprod{B(u_1(s))-B(u_2(s))}{u(s)} )\ud s
	\nonumber \\
	&\hspace{-0.3in} =\int_0^t|G(u_1(s))-G(u_2(s))|_{L_2(U_0,H)}^2\ud s + 2\int_0^t\iprod{[G(u_1(s))-G(u_2(s))]\ud W(s)}{u}
	\nonumber \\
	& + 2\int_{(0,t]}\int_{E_0}\iprod{ u(s)}{K(u_1(s-),\xi)-K(u_2(s-),\xi)}
	\ud\hat\pi(s,\xi)\nonumber\\
	&+\int_{(0,t]}\int_{E_0} |K(u_1(s-),\xi)-K(u_2(s-),\xi)|^2\ud\pi(s,\xi).\label{difference}
	\end{align}
Here we recall that our assumption \eqref{F1} gives:
\begin{equation}\label{boundF}
	\iprod{F(u_1)-F(u_2)}{u}\ge 0.
\end{equation}
Next we derive upper bounds for the term $B$. Observe that due to the bilinearity of $B$ and the  assumptions \eqref{eqn:B_cancellation_property} and \eqref{rs} we have
	\begin{align}
		\iprod{B(u_1,u_1)-B(u_2,u_2)}{u}&=\iprod{\frac 12\left(B(u_1+u_2,u_1-u_2)+B(u_1-u_2,u_1+u_2)\right)}{u}\nonumber\\
		&=\frac 12\left(\iprod{B(u_1+u_2,u)}{u}+\iprod{B(u,u_1+u_2)}{u}\right) \nonumber\\
	&=-(B(u,u_1),u)\nonumber\\
		&\lesssim |u|\|u_1\|_X\|u\| \label{uniq_B}.
	\end{align}

We continue our analysis 
 by introducing the stopping time
\begin{equation}\label{tau_m}
\tau^M=\inf\{t \geq 0:
\int_0^t \|u_1\|^p_X \ud s > M\}, \quad M>0.
\end{equation}
	Then inserting \eqref{uniq_B} and {the monotonicity condition} \eqref{boundF} into equation \eqref{difference}, taking the supremum in time over $[0,t\wedge \tau^M]$, and finally taking expectations, {we obtain}
	\begin{align}
	\bE\sup_{s\in [0,t\wedge\tau^M]}\Hnorm{u(s)}^2&+\gamma\bE\int_0^{t\wedge\tau^M}\Vnorm{u(s)}^2 \ud s
		\le
\frac2{\gamma}\bE\int_0^{t\wedge\tau^M} \left( |u(s)|^2 \|u_1\|_X^2+ \frac{\gamma}{2}\|u(s)\|^2 \right)\ud s
\nonumber	\\&+2\bE\sup_{s\in [0,t\wedge\tau^M]}
	\bigg|\int_0^s\iprod{G(r,u_1(r))-G(r,u_2(r))\ud W(r)}{u}\bigg|\nonumber\\
	& + \bE\int_0^{t\wedge\tau^M}||G(u_1(s))-G(u_2(s))||_{L_2(U_0,H)}^2\ud s
\nonumber\\& +2\bE\sup_{s\in [0,t\wedge\tau^M]}
	\bigg|\int_{(0,s]}\int_{E_0}\iprod{K(r,u_1(r-),\xi)-K(r,u_2(r-),\xi)}{ u(r)}
	\ud\hat\pi(r,\xi)\bigg|\nonumber\end{align}\begin{align}&
	+\bE\int_{(0,t\wedge\tau^M]}\int_{E_0}
	|K(u_1(s-),\xi)-K(u_2(s-),\xi)|^2\ud\pi(s,\xi).\label{energyunique}
	\end{align}
	By using the Lipschitz growth assumption \eqref{eqn:G,K_H_Lipschitz}, we arrive at
	\begin{align}
	&\bE\int_0^{t\wedge\tau^M}||G(u_1(s))-G(u_2(s))||_{L_2(U_0,H)}^2\ud s\nonumber \\&\hspace{1.3in}+\bE\int_{(0,t\wedge\tau^M]}\int_{E_0}
	|K(u_1(s-),\xi)-K(u_2(s-),\xi)|^2\ud\pi(s,\xi)\nonumber
	\\&\hspace{1.3in} \le \rho\bE
	\int_0^{t\wedge\tau^M}(1+\Hnorm{u(s)}^2)\ud s.\label{boundG1}
	\end{align}
	By utilizing the BDG inequality along with the hypothesis \eqref{eqn:G,K_H_growth}, for some $C=C(\rho,T)>0$, we {see that}
	\begin{align}\label{boundG2}
	\nonumber\bE\sup_{s\in [0,t\wedge\tau^M]}
	&\bigg|\int_0^s\iprod{[G(r,u_1(r))-G(r,u_2(r))]\ud W(r)}{u}\bigg|\\
	&\lesssim \bE\left(\int_0^{t\wedge\tau^M}\iprod{G(r,u_1(r))-G(r,u_2(r))}{u}^2\ud s\right)^{\frac12}\nonumber\\\nonumber
	&\lesssim\bE\left(\int_0^{t\wedge\tau^M}\Hnorm{u(s)}^2\Hnorm{G(u_1(s))-G(u_2(s))}^2\ud s\right)^{\frac12}\\
	&\nonumber\lesssim\bE\left(\sup_{s\in [0,t\wedge\tau^M]}\Hnorm{u(s)}^2\int_0^{t\wedge\tau^M}\Hnorm{G(u_1(s))-G(u_2(s))}^2\ud s\right)^{\frac12}\\
&\nonumber \le \frac 14\bE\sup_{s\in [0,t\wedge\tau^M]}\Hnorm{u(s)}^2+C\bE\int_0^{t\wedge\tau^M}
	\Hnorm{G(u_1(s))-G(u_2(s))}^2\ud s
\\&\le \frac 14\bE\sup_{s\in [0,t\wedge\tau^M]}\Hnorm{u(s)}^2+C\bE\int_0^{t\wedge\tau^M}
	\rho\Hnorm{u(s)}^2\ud s.
	\end{align}
	And similarly 
	\begin{align}
	\nonumber\bE &\sup_{s\in [0,t\wedge\tau^M]}
	\bigg|\int_{(0,t]}\int_{E_0}\iprod{ u(s)}{K(u_1(s-),\xi)-K(u_2(s-),\xi)}
	\ud \hat\pi(s,\xi)\bigg|
	\\
	& \nonumber\lesssim \bE
	\left(\int_{(0,t]}\int_{E_0}\iprod{ u(s)}{K(u_1(s-),\xi)-K(u_2(s-),\xi)}^2
	\ud \pi(s,\xi)\right)^{\frac12}
	\\&\nonumber\lesssim \bE\left(
	\sup_{s\in [0,t\wedge\tau^M]}\Hnorm{u(s)}^2\bigg[\int_{(0,t]}\int_{E_0}\Hnorm{K(u_1(s-),\xi)-K(u_2(s-),\xi)}^2\ud\pi(s,\xi)\bigg]\right)^{\frac12}
	\\&\nonumber\le \frac 14\bE\sup_{s\in [0,t\wedge\tau^M]}\Hnorm{u(s)}^2+C\bE\bigg[\int_0^{t\wedge \tau^M}\int_{E_0}\Hnorm{K(u_1(s-),\xi)-K(u_2(s-),\xi)}^2\ud\nu(\xi)\ud s\bigg]
	\\&\le \frac 14\bE\sup_{s\in [0,t\wedge\tau^M]}\Hnorm{u(s)}^2+C\bE\int_0^{t\wedge\tau^M}
	\rho\Hnorm{u(s)}^2\ud s. \label{boundK1}
	\end{align}
	Collecting the estimates \eqref{boundG1}, \eqref{boundG2} and \eqref{boundK1} we find that
	\begin{align}\label{ineq_uniq}	\bE\sup_{s\in [0,t\wedge\tau^M]}\Hnorm{u(s)}^2\le C(\rho,T,\gamma)\bE\int_0^{t\wedge\tau^M} |u(s)|^{2}\|u_1\|_X^2\ud s.	\end{align}
	

We observe that the process $\varphi(s):= \|u_1(s)\|_X^2$ satisfies
	\begin{equation}
	\int_0^{t\wedge\tau^M}\varphi(s)\ud s\le
	 \left(\int_0^{t\wedge\tau^M} \|u_1(s)\|^p_X \ud s\right)^{\frac{2}p} \leq  M^{\frac{2}p},\qquad \bP{-a.s.}
	\end{equation}
	We now apply the stochastic Gr\"{o}nwall {lemma}  (see e.g. Lemma 5.3 in \cite{GHZ}) to \eqref{ineq_uniq} and use the fact that $\tau^M \rightarrow \infty$ as $M \rightarrow \infty$ almost surely, to deduce 
	\begin{align}\label{eqn:indist}
	\bE\left(\sup_{t\in [0,T]}\Hnorm
	{u(t)}^2\right) 
	=0.
	\end{align}

	 Since $T \geq 0$ is arbitrary, this completes the proof of Proposition \ref{pathwise}.
\end{proof}
\begin{rem}
	Note that {by fixing the stochastic basis in the proof above} one can prove the uniqueness of the pathwise solutions to \eqref{eqn}-\eqref{bc} {defined as in} Definition \ref{definition2}.
\end{rem}

\subsection{Pathwise solution}\label{path_sol}

In this section we apply the well known result by  Gy\"{o}ngy and Krylov in \cite{GyKr} that extends 
the Yamada–Watanabe theorem (cf. \cite{YW}) to infinite dimension. {This theorem} states that the existence of martingale solutions and the pathwise uniqueness imply the existence of a pathwise solution to \eqref{eqn}-\eqref{bc}.

We can now finish the proof of Theorem \ref{existence} which follows a standard procedure. 
Here we will briefly describe the key steps involved.

First using the following theorem in \cite{GyKr}, it can be shown that $\bar{\bar {u}}^{\tilde{R}}$, the sequence of solutions to \eqref{approx_eqn2}, converges in probability as $\tilde{R} \rightarrow \infty$ with respect to the original stochastic basis $(\Omega,\mathcal{F}, (\mathcal{F}_t)_{t \geq 0},\bP,W, \pi)$.

\begin{thm}\label{GKthm}
	Let $\sE$ be a Polish space. A sequence of ${\sE}$-valued random variables $\{\mathscr{Y}_n\}_{n \geq 0}$, defined on te same probability space, converges in probability if and only if for every subsequence of joint probability laws of $\{\mathscr{Y}_n\}_{n \geq 0}$, given by $\{\mu_{m_{l}},\mu_{n_{l}}\}_{l\geq 0}$, there exists a further subsequence that converges weakly to a probability measure $\mu$ such that
	\begin{align}
		\mu(\{(x,y) \in {\sE} \times {\sE} : x=y\})=1.
	\end{align}\end{thm}
In preparation for using this theorem, define \begin{align}\mathcal{Y}:=H \times \sX_1 \times H \times \sX_1 \times C([0,T];U) \times \mathcal{N}^{\#*}_{[0,\infty)\times E},\end{align}
 where  $\mathcal{N}^{\#*}_{[0,\infty)}$	denotes the space of counting measures on  $[0,\infty)\times E$ that are
finite on bounded sets.
 Now, consider two subsequences $(\mathscr{Y}_k:=\bar{\bar {u}}^{\tilde{R}_k})_{k=1}^\infty$ and $(\mathscr{Y}_k^{\#}:=\bar{\bar {u}}^{\tilde{R}^{\#}_{k}})_{k=1}^\infty$ and the double sequence $({\mathscr{Y}_k}_0 , \mathscr{Y}_k , {\mathscr{Y}^{\#}_{k0}}  , {\mathscr{Y}^{\#}_k} ,W, \pi)$. Next we denote by $\{\mu^{\tilde R_k,\tilde R^{\#}_k}\}_{k\geq 1}$ the sequence of
joint laws of the $\mathcal{Y}$-valued random variables. Then it can be shown that $\{\mu^{\tilde R_k,\tilde R^{\#}_k}\}_{k\geq 1}$ is tight and thus weakly compact over $\mathcal{Y}$ . Then by Prohorov’s theorem (see e.g., \cite{Billingsley99}), there exists a subsequence $k\rightarrow \infty$ such that $\{\mu^{\tilde R_k,\tilde R^{\#}_k}\}_{k\geq 1}$ converges weakly to a probability measure $\tilde \nu$ on $\mathcal{Y}$. Then by another application of the Skorohod representation theorem we know that there exists a filtered probability space $(\tilde \Omega^{},\tilde{\mathcal{F}}^{}, (\tilde{\sF}_t)_{t\geq 0},\tilde{\bP}^{})$ and $\mathcal{Y}$-random variables
$(\tilde {\mathscr{Y}_k}_0 , \tilde {\mathscr{Y}}_k ,\tilde {{\mathscr{Y}}^{\#}_{k0}}  ,\tilde{ {\mathscr{Y}}^{\#}_k} ,\tilde W, \tilde{\pi})$ and $(\tilde {\mathscr{Y}}_0 , \tilde {\mathscr{Y}} ,\tilde {{\mathscr{Y}}^{\#}_{0}}  ,\tilde{ {\mathscr{Y}}^{\#}} ,\tilde W, \tilde{\pi})$ such that $(\tilde {\mathscr{Y}_k}_0 , \tilde {\mathscr{Y}}_k ,\tilde W, \tilde{\pi})$ and $(\tilde {\mathscr{Y}}_0 , \tilde {\mathscr{Y}} ,\tilde W, \tilde{\pi})$ and, $(\tilde {{\mathscr{Y}}^{\#}_{k0}}  ,\tilde{ {\mathscr{Y}}^{\#}_k} ,\tilde W, \tilde{\pi})$ and $(\tilde {{\mathscr{Y}}^{\#}_{0}}  ,\tilde{ {\mathscr{Y}}^{\#}} ,\tilde W, \tilde{\pi})$ satisfy the conclusions of Theorem \ref{skorohod3}.

 Thus using the arguments in the proof of Theorem \ref{existence_martingale} following Theorem \ref{skorohod3} we conclude that $\tilde{\mathscr{Y}}$ and $\tilde{\mathscr{Y}}^{\#}$ are two martingale solutions to \eqref{eqn}-\eqref{bc} with initial conditions $\tilde{\mathscr{Y}}_0$ and $\tilde{\mathscr{Y}}^{\#}_0$ respectively with respect to the stochastic basis $(\tilde \Omega^{},\tilde{\mathcal{F}}^{},(\tilde{\mathcal{F}}_t^{})_{t\geq 0},\tilde{\bP}^{},\tilde W,\tilde{\pi})$.
Now since $\tilde{\bP}[\tilde{\mathscr{Y}}_0=\tilde{\mathscr{Y}}_0^{\#}=u_0]=1$, thanks to Proposition \ref{pathwise} we obtain $\tilde{\bP}$-a.s. that $\tilde{\mathscr{Y}}=\tilde{\mathscr{Y}}^{\#}$ for all $t\in[0,T]$.
 
Thus as mentioned above we use Theorem \ref{GKthm} to infer that the sequence of $\sX_1$-valued random variables $\{\bar{\bar u}^{\tilde R}\}_{R \geq 1}$, defined on the original stochasic basis $(\Omega,\mathcal{F},\bP)$, converges to $u$ in probability as $\tilde R \rightarrow \infty$. 
Then, along another subsequence $\tilde R \rightarrow \infty$ we have {that} the following conditions hold $\bP$-a.s.
 \begin{align}
 \bar{\bar u}_0^{\tilde R} &\rightarrow u_0 \text{ in }H,\\
\bar{\bar u}^{\tilde R} &\rightarrow u \text{ in } L^2(0,T;H)\cap \mathcal{D}([0,T];V'+X').
 \end{align}

Applying the identical argument as in Section \ref{sec:tildeR} we obtain that the limit $u$ is indeed the unique global pathwise solution to \eqref{eqn}-\eqref{bc}, i.e., $u$ satisfies \eqref{eqn2:pathwise}.\\
This finishes the proof of Theorem \ref{existence}. }


\section{Examples}\label{Examples}
\subsection{The Ladyzensaya-Smagorinsky equations}\label{smago}
The first application of Theorem \ref{existence} that we want to present, corresponds to a variation of the Navier-Stokes equations introduced by Ladyzhenskaya \cite{Lad67}, \cite{Lad68}; the equations also appears in relation with the Smagorinsky model of turbulence \cite{Sma63}. For $p >2$ {we} look for a vector valued function $u$ and {a} scalar valued function $p_*$ that satisfy:
\begin{equation} \label{e11.132d}
\begin{cases}
\ud u + [(u \cdot \nabla)u- \text{div}((1+|\nabla u|^2)^{\frac{p-2}{2}}\nabla u) + \nabla p_*]\ud t \\
\hspace{2in}= G(u) \ud W + \int_{E_0} K(u,\xi) \ud \hat{\pi} & \tn{in } (0,T) \times \sO \\
\text{div }u=0   &\tn{in } (0,T) \times \sO\\
u = 0 & \tn{on } (0,T) \times  \partial \sO \\
u(0) = u_0 & \tn{on } \{t=0\} \times \sO,
\end{cases}
\end{equation}
We consider 
$X=\{ \varphi \in (W^{1,p}_0(\mathcal{O}))^d, \text{ div } \varphi =0\}$, $V=\{\varphi \in (H^1_0(\mathcal{O}))^d, \text{ div } \varphi =0\}$ and $H=\{ \varphi \in (L^2(\mathcal{O}))^d, \text{ div }\varphi=0, \varphi \cdot \nu =0 \text{ on } \partial \mathcal{O}\}$ where $\nu$ is the unit {outward} normal on $\partial \mathcal{O}$.{ Here $F(u) \in X'$ is {associated with the operator} $-\text{div}((1+|\nabla u|^2)^{\frac{p-2}{2}}\nabla u)$ and is defined by $ (F(u),v)=\int_{\mathcal{O}}(1+|\nabla u|^{2})^{\frac{p-2}{p}}\nabla u \nabla v\ud x$ for $u,v \in X$. As described in Section \ref{sec:functional_framework} the bilinear term $B:V \times V \rightarrow V'$ is defined by $(B(u,v),w)= \int_{\mathcal{O}}u_{i}\partial_i v_j w_j \ud x$.}\\
It is easy to see that the monotonicity condition \eqref{F1} i.e. $\iprod{F(u)-F(v)}{u-v}_{X,X'}\ge 0, \,\,\forall u,v\in X$ is satisfied. \\
Additionally we also have for some $c_p^*>0$ that (see \cite{Lad67}),
\begin{align}\label{F1''}
\iprod{F(u)-F(v)}{u-v}\ge c_p^*\|u-v\|^2 \qquad \forall u,v\in X.
\end{align}
Indeed define $w^r= ur+(1-r)v$ and ${f}( v)=(1+|\nabla v|^2)^{\frac{p-2}{2}}\nabla v$. Then we have that
\begin{align*}
f(u)-f(v)&=\int_0^1 \frac{\ud f(w^r)}{\ud r} \ud r\\
& =\nabla(u-v)\int_0^1 \left((p-2)(1+|\nabla w^r|^2)^{\frac{p-2}2}+ (1+|\nabla w^r|^2)^{\frac{p-2}2}\right) \ud r.
\end{align*}
This gives,
\begin{align*}
\iprod{F(u)-F(v)}{u-v}=\iprod{f(u)-f(v)}{\nabla u-\nabla v}&=\iprod{\int_0^1 \frac{\ud f(w^r)}{\ud r} \ud r}{\nabla u- \nabla v},\\
&\geq (p-1)|\nabla u-\nabla v|^2.
\end{align*}
In fact it has also been shown that a stronger condition (F1') is satisfied (see e.g. \cite{GM}):\\
(F1') (Strong monotonicity): For $p>2$ there exists a $c_p>0$ such that,
{\begin{align}\label{F1'}
\iprod{F(u)-F(v)}{u-v}\ge c_p\|u-v\|^p_X \qquad \forall u,v\in X.
\end{align}}
It is easy to see that the hypothesis (F2) is satisfied for any $p \in \R$ (see e.g. \cite{Lio69}).\\
{As proved in \cite{TT20}, the hypotheses \eqref{gen_convex} - \eqref{gc} are satisfied for any $p>2$ and $k=1$ by {considering} the corresponding convex functional for any $v \in X$
\begin{align}
{J}(v)= \int_{\mathcal{O}} \mathscr{J}_1(\nabla v)dx=  \int_{\mathcal{O}} \frac{1}{p}(1+|\nabla v|^2)^{\frac{p}2} dx,
\end{align}
where $\mathscr{J}_1: \R^{d^2} \rightarrow \R$ for $ x \in \R^{d^2}$ is given by,
\begin{align}
\mathscr{J}_1(x)= \frac{1}{p}(1+|x|^2)^{\frac{p}2}. 
\end{align} This also leads to (F3) and (F4). \\}
The cancellation property \eqref{eqn:B_cancellation_property} is a classical one; a proof can be found {e.g.} in \cite{Tem95}, \cite{T_NSE}. Next we will show that the term $B$ satisfies \eqref{eqn:B_bound2}-\eqref{eqn:B_bound3}. In {the most restrictive case, in} space dimensions $d=3$ we use the Ladyzenskaya inequality to obtain 
\begin{align}
|\Vdualpair{B(v,u)}{w}|= |\Vdualpair{B(v,w)}{u}|& \leq |v|_{(L^2(\mathcal{O}))^d} |\nabla w|_{(L^4(\mathcal{O}))^d}|u|_{(L^4(\mathcal{O}))^d}\nonumber\\
& \lesssim |v||u|^{\frac14}\|u\|^{\frac34} \| w\|^{\frac14}|Aw|^{\frac34}\nonumber\\&
 \lesssim |v|\|u\||Aw|.
\end{align}
When $d=3$ we also have the embeddings $H^1(\mathcal{O}) \subset L^6(\mathcal{O})$ and $H^{\frac12}(\mathcal{O}) \subset L^3(\mathcal{O})$ (see, e.g. \cite{Br11}) which together with the Gagliardo-Nirenberg interpolation inequality give us,
\begin{align}
|\Vdualpair{B(v,u)}{w}| & \leq |v|_{(L^3(\mathcal{O}))^3} |\nabla u|_{(L^2(\mathcal{O}))^3}|w|_{(L^6(\mathcal{O}))^3}\nonumber\\
& \lesssim |v|^{\frac12}\|v\|^{\frac12} \|u\| \|w\|.
\end{align}

The case $d=2$ is easier to prove {and so we omit the calculations}. \\
We will next prove that the assumption \eqref{rs} is satisfied.

Observe that for $p\geq d$ we have, $X \subset (L^\infty(\mathcal{O}))^d$ and 
 hence we readily obtain
\begin{align}\label{pgeqd}
	|(B(u,u),w)| =& \lesssim |u|| w|_{(L^\infty(\mathcal{O}))^d}	|\nabla u|\nonumber\\
& \lesssim |u| | w |_{X}\|u\|\nonumber\\
& \lesssim  |u|^2 \|w \|_X^2+ \|u\|^2 .
\end{align}
{In conclusion Theorem \ref{existence} applies to equations \eqref{smago} when $p \geq d$ and {$p>2$}.}
\begin{rem}{ We now show how by a slight change in the proof of Proposition \ref{pathwise} we can cover other related situations. Indeed, we observe} that the hypothesis \eqref{rs} in Theorem \ref{existence} was used only in Section \ref{uniqueness} in the proof of Proposition \ref{pathwise} to prove pathwise uniqueness.
However thanks to the stronger monotonocity properties \eqref{F1''} and \eqref{F1'} available for the operator $F$ we can extend and thus improve the uniqueness result {for values of $p$} smaller than $d$ in the following way; note that we still asume $p>2$, so that we cover here the cases $2<p\leq d$. \end{rem}
For example, consider the case when
 $d=3>p \geq \frac{7}{3}$. Observe that $L^p(0,T;X) \cap L^\infty(0,T;H) \subset L^{ \mathcal{S}}(0,T;(L^{ \mathcal{R}}(\mathcal{O}))^3)$ if for some $\theta \in (0,1)$ we can write $\frac{1}{\mathcal{S}}=\frac{1-\theta}{p}$ and $\frac1{ \mathcal{R}} = \frac{(1-\theta)(3-p)}{3p} + \frac{\theta}{2}$. Indeed, for $\frac{7}{3}\leq p<3$ we can let $\theta=\frac{2(3p-7)}{5(p-2)}$. For this $\theta$ we write $\frac1\sigma:=1-\frac1p-\frac1{  \mathcal{R}} =\frac1{   \mathcal{S}} + \frac{\frac3{  \mathcal{R}}}{6}>0$ which gives,
\begin{align}
\nonumber|(B(u,u),w)| \leq |u|_{(L^\sigma(\mathcal{O}))^d} \|u\|_X |w|_{(L^{  \mathcal{R}}(\mathcal{O}))^3} 
& \leq |u|^{\frac2{   \mathcal{S}}}_{(L^2(\mathcal{O}))^3}|u|^{\frac3{  \mathcal{R}}}_{(L^6(\mathcal{O}))^3}\|u\|_X|w|_{(L^{  \mathcal{R}}(\mathcal{O}))^3} \nonumber\\
& \leq |u|^{\frac2{   \mathcal{S}}}\|u\|^{\frac3{  \mathcal{R}}}\|u\|_X|w|_{(L^{  \mathcal{R}}(\mathcal{O}))^3}\nonumber\\
& \leq |u|^2|w|^{   \mathcal{S}}_{(L^{  \mathcal{R}}(\mathcal{O}))^3}+\|u\|^{2} +\|u\|^p_X.\label{pleqd}
\end{align}
Hence, using \eqref{F1''} and \eqref{F1'}, we obtain the following estimates in place of \eqref{energyunique} in the proof of Proposition \ref{pathwise},
\begin{align}
\bE\sup_{s\in [0,t\wedge\tau^M]}\Hnorm{u(s)}^2&+ c_p^*\, \bE\int_0^{t\wedge\tau^M}\Vnorm{u(s)}^2 \ud s+c_p \,
\bE\int_0^{t\wedge\tau^M}\Vnorm{u(s)}_X^p\ud s \nonumber\\
&\le
\bE\int_0^{t\wedge\tau^M}  |u(s)|^2 \|u_1\|^{   \mathcal{S}}_{(L^{  \mathcal{R}}(\mathcal{O}))^3}+ \|u(s)\|^2 +\|u(s)\|_X^p\ud s
\nonumber	
\\&+2\bE\sup_{s\in [0,t\wedge\tau^M]}
\bigg|\int_0^s\iprod{G(r,u_1(s))-G(r,u_2(s))\ud W(s)}{u}\bigg|\nonumber\\
& + \bE\int_0^{t\wedge\tau^M}||G(u_1(s))-G(u_2(s))||_{L_2(U_0,H)}^2\ud s
\nonumber\\
& +2\bE\sup_{s\in [0,t\wedge\tau^M]}
\bigg|\int_{(0,t]}\iprod{\int_{E_0} u(s)}{K(u_1(s-),\xi)-K(u_2(s-),\xi)}
\ud\hat\pi(s,\xi)\bigg|\nonumber\\
&+\bE\int_{[0,t\wedge\tau^M]}\int_{E_0}
|K(u_1(s-),\xi)-K(u_2(s-),\xi)|^2\ud\pi(s,\xi).\label{unique_lady}
\end{align}
Similarly instead of \eqref{tau_m} we define
\begin{equation}
\tau^M=\inf\{t \geq 0:
\int_0^t \|u_1\|^{   \mathcal{S}}_{(L^{  \mathcal{R}}(\mathcal{O}))^3} > M\}, \quad M>0.
\end{equation}
The rest of the proof of Proposition \ref{pathwise} follows identically.
  
  Since all the hypotheses are satisfied, Theorem \ref{existence} applies producing the existence of a {unique} global pathwise solution to \eqref{e11.132d} for any $p>2$ when $d=2$ and any $p\geq \frac73$ when $d=3$.
  
  \begin{rem}In the deterministic setting, Ladyzenkaya in \cite{Lad68}, showed that weak solutions to \eqref{e11.132d} exist and are globally unique in time for any $p\geq 3$ for space dimensions $d=2,3$ (see also Chapter 2, Section 5 in \cite{Lio69} and \cite{DW13}).\end{rem} 
 { \begin{rem}\label{rem_plaplace}
  	Note that Theorem \ref{existence_martingale} gives the existence {(without uniqueness)} of a martingale solution to \eqref{e11.132d} for all $p>2$. The same is true for the case where the monotone operator in \eqref{e11.132d} is replaced by the operator $  F(u)=-\sum^d_{i=1}\frac{\partial}{\partial x_i}\left( |\nabla u|^{p-2}\frac{\partial u}{\partial x_i}\right)$. This is because we require the presence of the linear term $A=-\Delta$ only while proving pathwise uniqueness of martingale solutions (in obtaining the estimates \eqref{unique_lady}) which eventually leads to the existence of a unique pathwise solution as described in Section \ref{path_sol}. See below Remark \ref{imp_equi} for a comment on how the equicoercivity result \eqref{FR3} is utilized in proving the results from earlier sections when the term $A$ is absent.
  \end{rem}}
  \subsection{Equations without the operator $A$ and with $B=0$}\label{withoutA}
  \subsubsection{Evolution equations with $p$-Laplacian}\label{ex_plaplace}
  Consider the following equations for some $p > 2$:
  \begin{equation}
  \begin{cases}\label{plaplace}
  \ud u -  \left[\sum^d_{i=1}\frac{\partial}{\partial x_i}\left( |\nabla u|^{p-2}\frac{\partial u}{\partial x_i}\right)\right]\ud t= G(u) \ud W + \int_{E_0} K(u,\xi) \ud \hat{\pi} & \tn{in } (0,T) \times \sO \\
  u = 0 & \tn{on } (0,T) \times  \partial \sO \\
  u(0) = u_0 & \tn{on } \{t=0\} \times \sO.
  \end{cases}
  \end{equation}
  Here $X=(W^{1,p}_0(\mathcal{O}))^d$ and $H=(L^2(\mathcal{O}))^d$. Even though the operator $A$ is absent from equations \eqref{plaplace}, the natural choice for $V$ is the space $(H^1_0(\mathcal{O}))^d$.
  {The convex functional $J$ in \eqref{J_form} associated with the monotone operator $F(u)=-  \sum^d_{i=1}\frac{\partial}{\partial x_i}\left[ |\nabla u|^{p-2}\frac{\partial u}{\partial x_i}\right]$ is given by $$J(v)=\frac1p\int_{\mathcal{O}}|\nabla v|^p \ud x.$$}
  We again refer to \cite{TT20} for a proof that the conditions \eqref{gen_convex}-\eqref{gc} are satisfied by $\mathscr{J}_1(x)=\frac1p|x|^p$. { Observe that when $B=0$, even though we used a special eigenbasis dependent on the term $A$ in the definition \eqref{e11.12},	the absence of the term $A$ does not hinder the application of the Galerkin method. One can use an appropriate orthonormal basis of the space $H$ to carry out the analysis in Section \ref{galerkin}.}
  Theorem \ref{existence} then gives the existence of a unique pathwise solution to \eqref{plaplace} for all $p \geq 2$.
  \begin{rem}\label{imp_equi}
  	For the equations \eqref{plaplace}, note that the equicoercivity result for $F^R$, stated in \eqref{FR3} and derived in \cite{TT20}, is key in getting the energy estimates in Lemma \ref{uniformboundL2} and Lemma \ref{gencase_uniformboundL2} independent of the parameter $R$. This is important due to the absence of the term $A$ which usually provides estimates in the space $L^2(\Omega; L^2(0,T;V))$. {Note also that while proving uniqueness of pathwise solutions, unlike in the equation discussed in Remark \ref{rem_plaplace}, the absence of the term $B$ compensates for the absence of the term $A$ in the derivation of the estimates \eqref{energyunique}.}   	
  		   	
  \end{rem}

\subsubsection{}\label{porous}
Consider the following equations:
\begin{equation} \label{porous_eqn}
\begin{cases}
\ud u + { \Delta}\left[|{\Delta u}|^{p-2}{\Delta u}\right]\ud t = G(u) \ud W + \int_{E_0} K(u,\xi) \ud \hat{\pi} & \tn{in } (0,T) \times \sO \\
u  = 0, \,\, {\frac{\partial u}{\partial \nu}=0} & \tn{on } (0,T) \times  \partial \sO \\
u(0) = u_0 & \tn{on } \{t=0\} \times \sO,
\end{cases}
\end{equation}
where $\nu$ is the unit {outward} normal on $\partial \mathcal{O}$.
Here we take $X=W^{2,p}_0(\mathcal{O})$, $V=H^2_0(\mathcal{O})$ and $H=L^2(\mathcal{O})$.
{The monotone, hemicontinuous operator $F(u)={ \Delta}\left[|{\Delta u}|^{p-2}{\Delta u}\right]$ is a variant of the non-linear operator appearing in the classical porous medium equation (cf. \cite{A86}) describing the flow of an ideal gas through a porous medium where the scalar valued function $u$ represents density.
Observe that the convex functional $J$ associated with the aforementioned monotone operator is given by
\begin{align*}
J(v)= \frac1p\int_{\mathcal{O}} |\Delta v|^p \ud x.
\end{align*}
Since we look for a scalar valued solution $u$ to \eqref{porous_eqn}, we prove for $x \in \mathbb{R}^{d^2}$, that $\mathscr{J}_2(x)=|x|^p$ satisfies the conditions \eqref{gen_convex}-\eqref{gc}, which were assumed in the general setting of a vector valued solution to the SPDEs \eqref{spde}. We will omit the computations since they are identical to the ones for $\mathscr{J}_1(x)$ in Section \eqref{ex_plaplace}. The assumptions \eqref{F3}-\eqref{F4} follow from \eqref{gc}. Observe here as well that the absence of the term $A$ does not create any problem; see Remark \ref{imp_equi}. Thus Theorem \ref{existence} applies and gives us the existence of a unique pathwise solution to \eqref{porous_eqn} for any $p\geq 2$.}

\subsection{Polynomial nonlinearity}\label{ex_poly}
We let $\mathcal{O}$ be an open, bounded subset of $\mathbb{R}^d$ with {a sufficiently} smooth boundary. Now consider the following equations which are a generalization of the stochastic Chafee-Infante equations.
\begin{equation} \label{e11.132a}
\begin{cases}
\ud u + [- \Delta u + P(u)]\ud t = G(u) \ud W + \int_{E_0} K(u,\xi) \ud \hat{\pi} & \tn{in } (0,T) \times \sO \\
u = 0 & \tn{on } (0,T) \times  \partial \sO \\
u(0) = u_0 & \tn{on } \{t=0\} \times \sO,
\end{cases}
\end{equation}
The operator $P$ can be a polynomial of odd order with a leading positive coefficient: $P(u)=\sum^{2q+1}_{i=0}  a_iu^i$ with $a_{2q+1}>0$.  Also in view of applications in \cite{CNT19} it could be a pseudo-polynomial of even order, sum of a term $a_{2q}|u|^{2q-1}u,$ $a_{2q}>0$, and of lower order terms of similar form.\\
The functional framework in Section \ref{sec:functional_framework} applies to the monotone operator {$F=-\Delta + P$} with $H=L^2(\mathcal{O}), V=H^1_0(\mathcal{O}),$  $X=L^{2q+2}(\mathcal{O})\cap H^1_0(\mathcal{O})$.

Notice that due to the absence, {in these equations}, of any additional nonlinearity like the term $B$, {the dimension} $d\geq 1$ can be arbitrarily large as opposed to the restriction $d \leq 3$ imposed in \cite{DHI} (see also Example 3 pg. 5698 in \cite{CNTT20}). Note also that for the same reasons, we do not require {the} strong monotonicity condition \eqref{F1'} and instead can get the desired results as long {as} we have the condition \eqref{F1}: $\iprod{F(u)-F(v)}{u-v}_{X,X'}\ge 0$ satisfied for any $ u,v\in X$. The authors in \cite{TT20} provided detailed calculations showing that this monotonicity condition holds along with the other hypotheses \eqref{F3}-\eqref{F4} with $p=2q+2$.\\ Thus Theorem \ref{existence} gives the existence of a global pathwise solution to \eqref{e11.132a} for any $q \geq 0$ and $d\geq 1$. 

{ 
	\subsection{Other examples}\label{other}
In this section we consider the following equations. 
\begin{equation} \label{x1}
\begin{cases}
\ud u - \frac{\partial}{\partial x_1}\left[(|\frac{\partial u}{\partial x_1}|^{p-2})\frac{\partial u}{\partial x_1}\right] \ud t = G(u) \ud W + \int_{E_0} K(u,\xi) \ud \hat{\pi} & \tn{in } (0,T) \times \sO \\
u\nu_1  = 0 & \tn{on } (0,T) \times  \partial \sO\\
u(0) = u_0 & \tn{on } \{t=0\} \times \sO,
\end{cases}
\end{equation}
{where $\nu=(\nu_1,...,\nu_d)$ is the unit outward normal along $\partial \sO$.}
Here we consider the spaces $X=\{u\in L^p(\mathcal{O}); \,  \frac{\partial u}{\partial x_1} \in L^p(\mathcal{O}), \, u\nu_1 = 0 \text{ on }  \partial\sO\}$ , $V=\{u\in L^2(\mathcal{O}); \,\frac{\partial u}{\partial x_1} \in L^2(\mathcal{O}),\, u\nu_1  = 0\text{ on }  \partial\sO  \}$ and $H=L^2(\mathcal{O})$. The boundary condition $u\nu_1=0$ makes sense on $\partial\sO$, considering the vector $U=(u,0,...,0)$ which belongs to $L^2(\sO)$, with div$(U)$ belonging also to $L^2(\sO)$, when $u$ and $\frac{\partial u}{\partial x_1} \in L^2(\sO)$. We then apply the classical trace theorem from Navier-Stokes equations theory, see e.g. Theorem 1.2, Ch. I in \cite{Tem95}. The space $V$ is Hilbert under the scalar product $(u,v)+(\frac{\partial u}{\partial x_1},\frac{\partial v}{\partial x_1})$. \\
In this case the associated convex functional for any $v\in X$ is given by $J(v)=\int_{\mathcal{O}} |\nabla v \cdot e_1|^p \ud x$ and $\mathscr{J}_1(x)=|x \cdot e_1|^p$ where $e_1=(1,0,...,0), x \in \mathbb{R}^d$. With straightforward calculations it can be seen that for $x \in \mathbb{R}^d$, a variant of 
 the conditions \eqref{gen_convex}-\eqref{gc} is satisfied with $|x|_{l^p(\mathbb{R}^d)}$ replaced by $|x \cdot e_1|_{l^p(\R)}$. 
The functions $f_1^R$ are defined as follows:
\begin{align}
f_1^{R}(x)=f_1(x)\mathbbm{1}_{\{|x \cdot e_1|_{p}<R\}} +  \left[ f_1 \left(\frac{Rx}{|x|_{p}}\right) + \left( 1-\frac{R}{|x|_{p}}\right) x  D^2 \mathscr{J}_1\left(\frac{Rx}{|x|_{p}}\right) \right] \mathbbm{1}_{\{|x \cdot e_1|_{p}>R\}},
\end{align}
The rest of the analysis carries out identically as with the example described in Section \ref{ex_plaplace} due to the similarities between the monotone operators. Observe that here the condition $X=W^{k,p}(\mathcal{O})\cap V$, which was assumed only for the simplicity of our analysis, is not required. In general, it suffices for the functions $\mathscr{J}_i$'s to have polynomial growth so that the constructions defined in \eqref{gen_def_FR}-\eqref{gen_def_fR} approximate the monotone operator $F$ and give us the desired existence result.
Finally, an application of Theorem \ref{existence} gives us the existence of a unique pathwise solution to the SPDE \eqref{x1} for any $p > 2$.}

\begin{appendices}
	\appendix 
\section{About the Skorohod representation theorem by Justin Cyr}

Here we present a proof of the Skorohod representation theorem.\footnote{In this article as well as in related articles  \cite{CNTT20}, \cite{1LayerShallow}, \cite{EulerPaper} we relied on a specialized version of the Skorohod theorem stated as Theorem C.1 in \cite{BHR}. In this appendix we provide, for the sake of completeness, an alternate proof of Theorem C.1 in \cite{BHR} due to Justin Cyr who has co-authored the articles \cite{CNTT20}, \cite{1LayerShallow}, \cite{EulerPaper}. }

\begin{thm}\label{skorohodtheorem}
	Let $(\Omega, \sF, \bP)$ be a probability space and $U^1, U^2$ be two complete separable metric spaces. Let $\chi_n \colon \Omega \ra U^1 \times U^2$, $n \in \N$, be a sequence of weakly convergent random variables with laws $\left(\mu_n \right)_{n=1}^\infty$. 
	For $i = 1, 2$ let $\pi_i \colon U^1 \times U^2 \ra U^i$ denote the natural projection, i.e.,
	\[
	\pi_i(\chi^1, \chi^2) = {\chi^i}, \qquad \forall (\chi^1, \chi^2) \in U^1 \times U^2.
	\]
	Finally, let us assume that the random variables $\pi_1(\chi_n) = \chi_n^1$ on $U^1$ share the same law, independent of $n$.
	
	Then, there exists a family of $U^1 \times U^2$-valued random variables $\left(\bar{\chi}_n\right)_{n=1}^\infty$ on the probability space $(\tilde{\Omega}, \tilde{\sF}, \tilde{\bP}) := ([0,1) \times [0,1), \tn{ Borel sets}, \tn{Lebesgue measure})$ and a random variable $\bar{\chi}_\infty$ on $(\tilde{\Omega}, \tilde{\sF}, \tilde{\bP})$ such that the following statements hold:
	\begin{myenum}
		\item The law of $\bar{\chi}_n$ is $\mu_n$ for every $n \in \N$,
		\item $\bar{\chi}_n \ra \bar{\chi}_\infty$ in $U^1 \times U^2$, $\tilde{\bP}$-a.s.,
		\item $\bar{\chi}_n^1 = \bar{\chi}_\infty^1$ everywhere on $\tilde{\Omega}$ for every $n \in \N$.
	\end{myenum}
\end{thm}

Note that in applications to martingale solutions for stochastic PDEs, $U^1$ is the noise space and $U^2$ is the solution space. A key part of the extension, part $iii)$, says that the noise processes for the sequence on the new probability space can be chosen in a way that does not depend on $n$. 

\begin{proof}
	Let $\mu_\infty$ denote the weak limit of the sequence $\left(\mu_n \right)_{n=1}^\infty$, which exists by hypothesis. The proof begins by defining partitions of $U^1$ and $U^2$. Let $\left(x_i\right)_{i=1}^\infty$ and $\left(y_j \right)_{j=1}^\infty$ be countable dense subsets of $U^1$ and $U^2$, respectively, and let $\left(r_k\right)_{k=1}^\infty$ be a sequence of positive numbers such that $r_k \downarrow 0$ and\footnote{$B(x,r)$ denotes the closed ball of radius $r > 0$ centered at the point $x$ in a metric space. We use this notation in both spaces $U^1$ and $U^2$.} $\mu_\infty(\partial B(x_i, r_k) \times U^2) = 0$ and $\mu_\infty(U^1 \times \partial B(y_j, r_k)) = 0$ for all $i,j, k \in \N$. It is possible to choose $r_k$ with this property because, for example, for each fixed $i$ there are at most countably many $r > 0$ such that $\mu_\infty(\partial B(x_i, r) \times U^2) > 0$. 
	
	We will define a sequence $\left( \sO^k \right)_{k=1}^\infty$ of nested countable partitions of $U^1$ and a sequence $\left( \sC^k \right)_{k=1}^\infty$ of nested countable partitions of $U^2$. For each fixed $k \in \N$ and multi-index $\alpha \in \N^k$, define
	\[
	\sO^k_{\alpha} := \{ x \in U^1 : \tn{for each } \ell \leq k, \alpha_\ell \tn{ is the minimal index } i \tn{ such that } x \in B(x_{i}, r_\ell) \}.
	\]
	The proof in \cite{BHR} uses an inductive construction to define an equivalent partition, denoted $O^k_{i_1,\ldots,i_k}$. The sets $O^k_{i_1,\ldots,i_k}$ defined in \cite{BHR} are empty if $i_1,\ldots,i_k$ is not monotone increasing but otherwise $O^k_{i_1,\ldots,i_k}$ is equal to the set $\sO^k_{\alpha}$ defined above, when $\alpha = (i_1,\ldots,i_k)$. In particular, the inductive construction shows that the sets $\sO^k := \{ \sO^k_{\alpha} : \alpha \in \N^k\}$ are Borel-measurable. Since $\left(x_i\right)_{i=1}^\infty$ is dense in $U^1$ it is clear that $U^1 = \bigcup_{\alpha \in \N^k} \sO^k_\alpha$ for each $k$. The minimality property ensures that $\sO^k_\alpha \cap \sO^k_\beta = \varnothing$ if $\alpha \neq \beta$. Thus, $\sO^k$ is a partition of $U^1$.
	
	Note that the partitions are nested. That is, each set $\sO^k_\alpha$ in partition $\sO^k$ is itself partitioned by the sets $\{\sO^{k+m}_{(\alpha, \tilde{\alpha})} : \tilde{\alpha} \in \N^m \}$ in partition $\sO^{k+m}$.
	
	We define nested countable partitions $\sC^k = \{ \sC^k_\beta : \beta \in \N^k\}$ of $U^2$ for each $k \in \N$ by the same construction, thereby obtaining nested countable partitions $\sO^k \times \sC^k := \{\sO^k_\alpha \times \sC^k_\beta : \alpha, \beta \in \N^k \}$ of $U^1 \times U^2$. By our choice of the sequence $\left(r_k \right)_{k=1}^\infty$ we have $\mu_\infty(\partial (\sO^k_\alpha \times \sC^k_\beta)) = 0$ for every $k \in \N$ and $\alpha, \beta \in \N^k$.
	
	For each measure in the sequence $\mu_1, \mu_2, \ldots$ and including $\mu_\infty$, and for each partition $\sO^k \times \sC^k$ of $U^1 \times U^2$ we define a nested countable partition of rectangles for $\tilde{\Omega} := [0,1) \times [0,1)$. In a change from the proof in \cite{BHR} we do this construction for $\mu_\infty$ as well to obtain the limit $\bar{\chi}_\infty$; so, we allow $n = \infty$ in the notation below.
	
	\begin{defn}
		For each $n \in \bar{\N}$ and $k \in \N$ define a partition $I^{n,k}$ of the square $[0,1) \times [0,1)$ as follows. For $\alpha, \beta \in \N^k$ we put
		\begin{align}
		I^{n,k}_{\alpha, \beta} &:= \Big[ \mu_n^1\Big( \bigcup_{\alpha' < \alpha} \sO^k_{\alpha'} \Big),  \mu_n^1\Big( \bigcup_{\alpha' \leq \alpha} \sO^k_{\alpha'} \Big) \Big)  \nonumber \\[5pt]
		&\qquad \times \Big[ \mu_n\Big( \sO^k_\alpha \times \bigcup_{\beta' < \beta} \sC^k_{\beta'} \ \Big| \ \sO^k_\alpha \times U^2 \Big) , \mu_n\Big( \sO^k_\alpha \times \bigcup_{\beta' \leq \beta} \sC^k_{\beta'} \ \Big| \ \sO^k_\alpha \times U^2 \Big)  \Big).  \label{eqn:C.1} 
		\end{align}
	\end{defn}
	
	Remarks about \eqref{eqn:C.1}:
	\begin{itemize} 
		\item The set unions, e.g.~$\bigcup_{\alpha' < \alpha} \sO^k_{\alpha'}$, are indexed by lexicographic order on $\N^k$. These are finite unions only when $\alpha$ is of the form $\alpha = (1, 1, \ldots, 1, a)$ for $a \in \N$, otherwise they are countably infinite.
		\item $\mu_n( A \ | \ B)$ denotes the conditional probability of $A$ given $B$ under $\mu_n$.
		\item $\mu^1_n$ denotes the marginal of $\mu_n$ on $U^1$, i.e., $\mu^1_n(A) := \mu_n(A \times U^2)$ for $A \in \sB(U^1)$. In other words, $\mu^1_n$ is the law of $\pi_1(\chi_n)$. 
	\end{itemize}
	
	For $k \in \N$ and each measure $\mu_n$ we are partitioning the unit square into countably many rectangles $I^{n,k} := \{ I^{n,k}_{\alpha, \beta} : \alpha, \beta \in \N^k\}$ with horizontal lengths $\mu_n^1(\sO^k_\alpha) = \mu_n(\sO^k_\alpha \times U^2)$ and vertical lengths $\mu_n(\sO^k_\alpha \times \sC^k_\beta \ | \ \sO^k_\alpha \times U^2)$. Note that the horizontal lengths $\mu_n^1(\sO^k_\alpha)$ do not depend on $n$ (including $n = \infty$), since we assume that the marginals $\mu^1_n$ are the same for all $n \in \N$. The area of rectangle $I^{n,k}_{\alpha, \beta}$ is $\mu_n(\sO^k_\alpha \times \sC^k_\beta)$. The rectangles $I^{n,k}$ cover the entire unit square $\tilde{\Omega}$ because $\sO^k \times \sC^k$ partitions $U^1 \times U^2$ and $\mu_n$ is a probability measure. 
	
	The partitions $I^{k, n}$ are nested as $k$ increases. That is, for $k_1 < k_2$ the rectangle $I^{n,k_1}_{\alpha, \beta}$ in partition $I^{n,k_1}$ is itself partitioned by the rectangles $\{I^{n,k_2}_{(\alpha, \tilde{\alpha}), (\beta,\tilde{\beta})} : \tilde{\alpha}, \tilde{\beta} \in \N^{k_2-k_1} \}$ from partition $I^{n, k_2}$. 
	
	Now we define the $U^1 \times U^2$ random variables $\left(\bar{\chi}_n \right)_{n=1}^\infty$ and $\bar{\chi}_\infty$ on the new probability space $(\tilde{\Omega}, \tilde{\sF}, \tilde{\bP})$. 
	
	First, for each $k \in \N$ and $\alpha \in \N^k$ we choose an element $x^k_\alpha$ of the set $\sO^k_\alpha$ from the countable dense subset $\left(x_i \right)_{i=1}^\infty$. For instance, let $i^*$ be the minimal index $i$ such that $x_i \in \sO^k_\alpha$ and set $x^k_\alpha := x_{i^*}$. Similarly, for each $k \in \N$ and $\beta \in \N^k$ choose an element $y^k_\beta$ of $\sC^k_\beta$ from the countable dense set $\left(y_j \right)_{j=1}^\infty$. 
	
	Second, for each $n \in \bar{\N}$ and $k \in \N$ we define a $U^1 \times U^2$-valued random variable $Z_{n,k} \colon \tilde{\Omega} \ra U^1 \times U^2$ by setting
	\[
	Z_{n,k}(\bar{\omega}) := (Z^1_{n,k}(\bar{\omega}), Z^2_{n,k}(\bar{\omega})) := (x^k_\alpha, y^k_\beta) \qquad \tn{if } \bar{\omega} \in I^{k, n}_{\alpha, \beta} \ (\alpha, \beta \in \N^k). 
	\] 
	Each $Z_{n,k}$ is constant on a countable set of measurable rectangles and is therefore measurable. The values of $Z_{n,k}$ do not depend on $n$; only the rectangles depend on $n$ and even then only in the vertical axis. 
	
	Third, we pass to the limit as $k \ra \infty$. We claim that the sequence $\left(Z_{n,k}\right)_{k=1}^\infty$ is Cauchy in $U^1 \times U^2$ for each fixed $n \in \bar{\N}$. Indeed, consider $Z_{n,k_1}(\bar{\omega})$ and $Z_{n,k_2}(\bar{\omega})$ for positive integers $k_1 < k_2$ and $\bar{\omega} \in \tilde{\Omega}$. The key  is that $\bar{\omega}$ belongs to a rectangle of the form $I^{n,k_1}_{\alpha, \beta}$ in partition $I^{n,k_1}$ and a rectangle of the form $I^{n,k_2}_{(\alpha, \tilde{\alpha}), (\beta,\tilde{\beta})}$ in partition $I^{n,k_2}$, where $\tilde{\alpha}, \tilde{\beta} \in \N^{k_2-k_1}$. The points $Z^1_{n,k_1}(\bar{\omega}) = x^{k_1}_{\alpha}$ and $Z^1_{n,k_2}(\bar{\omega}) = x^{k_2}_{(\alpha , \tilde{\alpha})}$ both belong to the set $\sO^{k_1}_{\alpha}$ in $U^1$, since $\sO^{k_2}_{(\alpha, \tilde{\alpha})} \subseteq \sO^{k_1}_{\alpha}$; likewise, $y^{k_1}_{\beta}, y^{k_2}_{(\beta, \bar{\beta})} \in \sC^{k_1}_\beta$. Since $\sO^{k_1}_\alpha$ and $\sC^{k_1}_\beta$ have diameter bounded by $2 r_{k_1}$, we obtain the following estimate, uniformly in $\bar{\omega}$:\footnote{$d$ denotes a complete metric for the topology on $U^1 \times U^2$. We can take $d := d^1 + d^2$ where $d^i$ is a complete metric on $U^i$.}
	\begin{equation} \label{eqn:C.3}
	d\big(Z_{n,k_1}(\bar{\omega}), Z_{n,k_2}(\bar{\omega}) \big) \leq 4r_{k_1}, \qquad \tn{ for all } k_2 \geq k_1, \ \tilde{\omega} \in \tilde{\Omega}.
	\end{equation}
	Therefore, $\left(Z_{n,k}\right)_{k=1}^\infty$ is Cauchy in $U^1 \times U^2$ for each fixed $n$. Since $U^1 \times U^2$ is complete we can define $U^1 \times U^2$-valued random variables $\bar{\chi}_n \colon \tilde{\Omega} \ra U^1 \times U^2$ for each $n \in \bar{\N}$ by
	\[
	\bar{\chi}_n(\bar{\omega}) := \lim_{k \ra \infty} Z_{n, k}(\bar{\omega}).
	\]
	This limit exists for every $\bar{\omega} \in \tilde{\Omega}$ and the resulting function $\bar{\chi}_n$ is Borel-measurable.
	
	It remains to show that the random variables $\left( \bar{\chi}_n \right)_{n=1}^\infty$ and $\bar{\chi}_\infty$ have the following properties:
	\begin{myenum}
		\item The law of $\bar{\chi}_n$ is $\mu_n$ for every $n \in \bar{\N}$.
		\item The projection $\pi_1(\bar{\chi}_n) = \bar{\chi}_n^1$ onto $U^1$ does not depend on $n$ (including $n=\infty$).
		\item $\bar{\chi}_n \ra \bar{\chi}_\infty $ in $U^1 \times U^2$, $\tilde{\bP}$-a.s.
	\end{myenum}
	
	We save the trickiest part, passing to the limit as $n \ra \infty$, for last. 
	
	For property i), note that for fixed $n \in \bar{\N}$, $k \in \N$ and $\alpha, \beta \in \N^k$ we have
	\begin{equation*}
	\tilde{\bP}[\bar{\chi}_n \in \sO^k_\alpha \times \sC^k_\beta] = \tn{Area}(I^{n,k}_{\alpha, \beta}) = \mu_n(\sO^k_\alpha \times \sC^k_\beta).
	\end{equation*}
	Therefore, the law of $\bar{\chi}_n$ agrees with $\mu_n$ on the sets $\Pi :=\{\sO^k_\alpha \times \sC^k_\beta : k \in \N, \alpha, \beta \in \N^k \}$. $\Pi$ is closed under finite intersections (i.e., $\Pi$ is a $\pi$-system) and generates the Borel $\sigma$-algebra on $U^1 \times U^2$. It follows from the $\pi$-$\lambda$ theorem (e.g., Theorem 3.3 in \cite{Bil95}) that the law of $\bar{\chi}_n$ is equal to $\mu_n$.
	
	For property ii), recall that, by hypothesis, the endpoints of the horizontal interval 
	\[
	\Big[ \mu_n^1\Big( \bigcup_{\alpha' < \alpha} \sO^k_{\alpha'} \Big),  \mu_n^1\Big( \bigcup_{\alpha' \leq \alpha} \sO^k_{\alpha'} \Big) \Big) 
	\] 
	of rectangle $I^{n, k}_{\alpha, \beta}$ do not depend on $n$. The value of $\beta$ for which $\bar{\omega} \in I^{n, k}_{\alpha, \beta}$ varies with $n$ but the value of $\alpha$ does not. Therefore, as $n$ varies with $k$ fixed the projection onto $U^1$ of 
	\(
	Z_{n,k}(\bar{\omega}) = (x^k_\alpha, y^k_\beta)
	\)
	remains unchanged. Thus, after passing to the limit as $k \ra \infty$ we see that $\bar{\chi}^1_n$ does not depend on $n$. 
	
	For property iii) consider the set $B \subseteq \tilde{\Omega}$ defined by 
	\[
	B := \bigcup_{k=1}^\infty \bigcap_{N=1}^\infty \bigcup_{n=N}^\infty \{ d(\bar{\chi}_n , \bar{\chi}_\infty) \geq 4 r_k \}.
	\]
	It is clear that $\bar{\chi}_n \ra \bar{\chi}_\infty$ as $n \ra \infty$ on $B^c$. Therefore, we must show that $\tilde{\bP}(B) = 0$. For this it is sufficient to show that for every $k \in \N$ and every $\ep > 0$ there exists an $N \in \N$ such that 
	\begin{equation} \label{eqn:show}
	\tilde{\bP} \Big( \bigcup_{n=N}^\infty \{ d(\bar{\chi}_n , \bar{\chi}_\infty) \geq 4 r_k \} \Big) \leq \ep.
	\end{equation}
	For a given $k \in \N$ and $\ep > 0$ we can find a positive integer $K$ and multi-indexes $\alpha_1, \beta_1, \ldots, \alpha_K, \beta_K$ in $\N^k$ such that 
	\begin{equation} \label{eqn:B_0_small_msre}
	\mu_{\infty} \Big( \Big( \bigcup_{i=1}^{K} \sO^k_{\alpha_i} \times \sC^k_{\beta_i} \Big)^c \Big) < \frac{\ep}{2},
	\end{equation}
	because $\mu_{\infty} \Big( \bigcup_{\alpha, \beta \in \N^k} \sO^k_{\alpha} \times \sC^k_{\beta}  \Big)  = 1$. To shorten the notation, define 
	\[
	I^k_i := I^{\infty, k}_{\alpha_i, \beta_i}, \ 1 \leq i \leq K, \quad \tn{and} \quad I^k_0 := \Big( \bigcup_{i=1}^{K} I^{\infty, k}_{\alpha_i, \beta_i} \Big)^c.
	\]
	
	Now, if $d(\bar{\chi}_n(\bar{\omega}) , \bar{\chi}_\infty(\bar{\omega})) \geq 4 r_k$ then either $\bar{\omega} \in I^k_0$ (which has small probability by \eqref{eqn:B_0_small_msre}) or $\bar{\omega}$ belongs to $I^{\infty, k}_{\alpha_i, \beta_i}$ for some $i = 1, \ldots,K$. In the second case $\bar{\omega}$ cannot belong to $I^{n, k}_{\alpha_i, \beta_i}$, otherwise $\bar{\chi}_n(\bar{\omega})$ and $\bar{\chi}_\infty(\bar{\omega})$ would both belong to the closure of $\sO^k_{\alpha_i} \times \sC^k_{\beta_i}$, which has radius $2 r_k$. Therefore, we have
	\begin{align*}
	\bigcup_{n=N}^\infty \{ d(\bar{\chi}_n , \bar{\chi}_\infty) \geq 4 r_k \} &\subseteq I^k_0 \cup \bigcup_{n=N}^\infty \bigcup_{i=1}^K I^k_i \setminus I^{n, k}_{\alpha_i, \beta_i} \\[5pt]
	&=  I^k_0 \cup \Big( \bigcup_{i=1}^K I^k_i \setminus \Big( \bigcap_{n=N}^\infty  I^{n, k}_{\alpha_i, \beta_i} \Big) \Big)
	\end{align*}
	and the estimate 
	\begin{equation} \label{eqn:estimate}
	\tilde{\bP} \Big( \bigcup_{n=N}^\infty \{ d(\bar{\chi}_n , \bar{\chi}_\infty) \geq 4 r_k \} \Big) \leq \frac{\ep}{2} + \sum_{i=1}^K \tilde{\bP} \Big( I^k_i \setminus \Big( \bigcap_{n=N}^\infty  I^{n, k}_{\alpha_i, \beta_i} \Big) \Big).
	\end{equation}
	To estimate $\tilde{\bP} \Big(I^k_i \setminus \Big( \bigcap_{n=N}^\infty  I^{n, k}_{\alpha_i, \beta_i} \Big) \Big)$ for each $i$, recall that all of the rectangles have the same horizontal component of length $\mu_\infty^1(\sO^k_{\alpha_i})$. Therefore, we just need to estimate the differences between the endpoints of the vertical components. We have 
	\begin{align*}
	\tilde{\bP} \Big(I^k_i &\setminus \Big( \bigcap_{n=N}^\infty  I^{n, k}_{\alpha_i, \beta_i} \Big) \Big) \leq  \\[5pt]
	& \mu_\infty^1(\sO^k_{\alpha_i}) \cdot \Big[ \Big|\mu_\infty \Big( \sO^{k}_{\alpha_i} \times  \bigcup_{\beta' < \beta_i} \sC^k_{\beta'} \ \Big| \ \sO^{k}_{\alpha_i} \times U^2  \Big) - \sup_{n \geq N} \mu_n \Big(  \sO^{k}_{\alpha_i} \times \bigcup_{\beta' < \beta_i} \sC^k_{\beta'} \ \Big| \ \sO^{k}_{\alpha_i} \times U^2  \Big) \Big| \\[5pt]
	&+ \Big|\mu_\infty \Big(\sO^{k}_{\alpha_i} \times \bigcup_{\beta' \leq \beta_i}  \sC^k_{\beta'} \ \Big| \ \sO^{k}_{\alpha_i} \times U^2  \Big) - \inf_{n \geq N} \mu_n \Big(  \sO^{k}_{\alpha_i} \times \bigcup_{\beta' \leq \beta_i} \sC^k_{\beta'} \ \Big| \ \sO^{k}_{\alpha_i} \times U^2  \Big) \Big| \Big].
	\end{align*}
	We chose the sequence $\left( r_k \right)_{k=1}^\infty$ in such a way that the boundary of every set appearing above has measure $0$ under $\mu_\infty$. Therefore it follows from the weak convergence $\mu_n \Rightarrow \mu_\infty$ that we can choose $N \in \N$ large enough so that the RHS above is bounded by $\ep/(2K)$. Combining this estimate with \eqref{eqn:estimate} yields \eqref{eqn:show} and establishes the desired $\tilde{\bP}$-a.s.~convergence.
\end{proof}
\end{appendices}
\section{Acknowledgment}
This work was supported by the Research Fund
of Indiana University. 

\bibliography{NTT21_biblio}
\bibliographystyle{plain}

\end{document}